\documentclass{article}
\usepackage{graphicx}
\usepackage[dvipsnames]{xcolor} 
\usepackage{enumitem}
\usepackage{float}
\usepackage{graphicx}
\usepackage{subcaption}
\usepackage[super]{nth}
\usepackage{amsmath}
\usepackage{bm}
\usepackage{hyperref}
\usepackage{bigints}
\usepackage[dvipsnames]{xcolor}  
\usepackage{lineno} 
\usepackage{mathtools}
\usepackage{array}
\usepackage{amssymb}
\usepackage{mathtools}
\usepackage{multirow}
\usepackage{tabularray}
\usepackage{booktabs} 
\usepackage{authblk}
\usepackage{xcolor}

\usepackage[utf8]{inputenc}
\usepackage{amsmath}
\usepackage{amsthm}
\usepackage{amssymb}
\usepackage{graphicx}
\usepackage{amsfonts}
\usepackage{latexsym}
\usepackage[title]{appendix}
\usepackage{color}
\usepackage{hyperref}
\usepackage{nomencl}
\makenomenclature

\newtheorem{theorem}{Theorem}[section]
\newtheorem{lemma}[theorem]{Lemma}
\newtheorem{proposition}[theorem]{Proposition}
\newtheorem{corollary}[theorem]{Corollary}
\newtheorem{definition}[theorem]{Definition}

\newtheorem{example}[theorem]{Example}
\newtheorem{remark}[theorem]{Remark}

\numberwithin{figure}{section}
\numberwithin{equation}{section}

\newcommand{\vertiii}[1]{{\left\vert\kern-0.25ex\left\vert\kern-0.25ex\left\vert #1 
    \right\vert\kern-0.25ex\right\vert\kern-0.25ex\right\vert}}

\begin{document}
\title{{\bf A kernel method for the learning of Wasserstein geometric flows}}

\author{Jianyu Hu$^{1}$, Juan-Pablo~Ortega$^{1}$, and Daiying Yin$^{1}$}


\maketitle

\begin{abstract}
Wasserstein gradient and Hamiltonian flows provide geometric formulations of many partial differential equations. We develop a structure-preserving kernel method for jointly estimating the external potential and interaction kernel governing these flows, assuming a known internal energy. The inverse problem is formulated through a trajectory-dependent observation operator and regularized in a product of reproducing kernel Hilbert spaces, yielding an explicit finite representation of the estimator. We characterize joint non-identifiability through the null space of the observation operator and derive a Fourier criterion on the torus, distinguishing uniqueness of the regularized estimator from recovery of the underlying physical components. We establish reconstruction error bounds and convergence rates under source conditions and suitable regularization and mesh refinement, together with a stability estimate for Hamiltonian flows with vanishing internal energy. Numerical experiments demonstrate joint reconstruction and illustrate how unobservable interaction modes can prevent accurate potential recovery despite small observation-equation residuals. A comparison with sparse learning illustrates the dependence of reconstruction accuracy on the chosen dictionary.
\end{abstract}

\makeatletter
\addtocounter{footnote}{1}{%
\footnotetext{Jianyu Hu, Juan-Pablo Ortega, and Daiying Yin are with the Division of Mathematical Sciences, School of Physical and Mathematical Sciences, Nanyang Technological University, Singapore. Their email addresses are {\texttt{Jianyu.Hu@ntu.edu.sg}}, {\texttt{Juan-Pablo.Ortega@ntu.edu.sg}}, and {\texttt{YIND0004@e.ntu.edu.sg}}, respectively.}
}
\makeatother


\section{Introduction}

Many differential systems arising in physics can be formulated using an energy functional and a
linear operator that connects the time derivative of the state to the gradient of the energy. Typically, the {\it energy functional} comprises internal, potential, and interaction energy terms, expressed explicitly in the following form:
\begin{equation}\label{energy-functional}
\mathcal{F}(\rho)= \int_{\mathbb{R}^d}U(\rho(x))\mathrm dx + \int_{\mathbb{R}^d}V(x)\rho(x)\mathrm dx + \frac{1}{2}\iint_{\mathbb{R}^d\times \mathbb{R}^d}W(x-y)\rho(x)\rho(y)\mathrm dx\mathrm dy.
\end{equation}
The energy functional is defined on the \emph{density manifold}, which is the space of probability density functions on configuration space equipped with the optimal transport (Wasserstein) distance~\cite{lafferty1988density,villani2009optimal,ambrosio2008gradient}.
Two prominent classes of such systems are the \emph{Wasserstein gradient flow systems}:
\begin{equation}\label{gradient-flow-o}
\partial_t \rho_t = \nabla \cdot \big(\rho_t \nabla (U'(\rho_t)+V + (W \ast \rho_t))\big),\quad \rho_0 = \mu_0,
\end{equation}
and the \emph{Wasserstein Hamiltonian flow systems}:
\begin{equation}\label{hamiltonian-flow-o}
\partial_{tt}\rho_t + \Gamma_W(\partial_t\rho_t,\partial_t\rho_t)=\nabla \cdot \big(\rho_t \nabla (U'(\rho_t)+V + (W \ast \rho_t))\big),\quad \rho_0 = \mu_0,
\end{equation}
where ``\(*\)'' is the convolution operator, and \(\Gamma_W\) is the Christoffel symbol with respect to the Levi-Civita connection \cite{abraham2012manifolds} associated to a natural metric on the density manifold that will be defined later on (see Definition \ref{Otto metric}, \cite{chow2020wasserstein}, and \cite[Proposition 11]{li2018geometry} for more details).

The energy functional in \eqref{energy-functional} encodes the constitutive mechanisms governing the observed density evolution. The internal energy $U$ describes local density-dependent effects such as nonlinear diffusion, the potential $V$ represents external forces or confinement, and the interaction kernel $W$ characterizes nonlocal interactions between individuals or particles. In many applications, these mechanisms are not directly observable, whereas macroscopic density profiles can be measured or reconstructed at successive times \cite{lund2014nonparametric,srivastava2025inference}. This leads to the inverse problem considered here: given discrete observations of a density trajectory $\rho_t$ and a known internal energy density $U$, we seek to recover the external potential $V$ and interaction potential $W$. 

To formulate the inverse problem, first consider an idealized smooth trajectory $\{\rho_t\}_{t\in[0,T]}$. Define the observation operator
\begin{equation*}
A(\phi,\psi)(t,x)
:=\Delta_{\rho_t}(\phi+\psi*\rho_t)(x),
\end{equation*}
with $\Delta_\rho:=\nabla\cdot(\rho\nabla)$, where $\phi$ and $\psi$ denote candidate external and interaction potentials, respectively. Then both Wasserstein dynamics yield the inverse equation
\begin{equation}\label{intro:inverse-equation}
A(V,W)=f,
\end{equation}
with $f(t,x)=\partial_t\rho_t(x)-[\Delta_{\rho_t}U'(\rho_t)](x)$ for Wasserstein gradient flows \eqref{gradient-flow-o}, and
$f(t,x)=\partial_{tt}\rho_t(x)+\Gamma_W(\partial_t\rho_t,\partial_t\rho_t)(x)-[\Delta_{\rho_t}U'(\rho_t)](x)$
for Wasserstein Hamiltonian flows \eqref{hamiltonian-flow-o}.
Although the forward map from $(V,W)$ to the generated trajectory is generally nonlinear, $A$ is linear jointly in $(\phi,\psi)$ once the observed trajectory is fixed. The problem studied here is therefore a trajectory-dependent linear inverse problem for the two energy components.

A natural approach to \eqref{intro:inverse-equation} is to minimize the density-weighted least-squares loss
\begin{equation*}
R(\phi,\psi):=\int_0^T\int_{\mathbb R^d}|A(\phi,\psi)(t,x)-f(t,x)|^2\rho_t(x)\,dx\,dt.
\end{equation*}
This formulation estimates the energy components within the prescribed Wasserstein geometric equations, so the resulting model retains the gradient or Hamiltonian form.
However, the observation operator may have a nontrivial null space: distinct candidate pairs $(\phi,\psi)$ can produce the same observation-operator values along the observed trajectory. Consequently, minimizing the residual need not uniquely identify the physical potentials. Moreover, numerical differentiation of discrete observations introduces errors that motivate regularization.

In this work, we seek candidate potentials $(\phi,\psi)$ in a product $\mathcal H_{K_1}\times\mathcal H_{K_2}$ of reproducing kernel Hilbert spaces (RKHSs), associated with kernels $K_1,K_2$, and consider the regularized loss
\begin{equation}\label{intro:regularized-loss}
R_\lambda(\phi,\psi)
:=R(\phi,\psi)+\lambda_1\|\phi\|_{\mathcal H_{K_1}}^2
+\lambda_2\|\psi\|_{\mathcal H_{K_2}}^2,
\qquad \lambda_1,\lambda_2>0.
\end{equation}
The RKHS structure provides explicit representers for the differential and convolutional observation functionals through the differential reproducing property \cite[Theorem 2.7]{RCSP2} (see also \cite{zhou2008derivative}). For finitely many observations, these yield computable kernel matrices and a finite, observation-dependent representation of the regularized empirical estimator. The penalties in \eqref{intro:regularized-loss} stabilize the inversion and select a unique regularized minimizer without removing the underlying ambiguity in the physical potentials. In Section~\ref{subsec:joint-identifiability}, we further use the Fourier characterization of translation-invariant RKHSs on the torus to characterize admissible cancellation directions and establish a Fourier criterion for joint identifiability within the chosen product RKHS.

\paragraph{Main results and structure of the paper.}

To implement the regularized formulation \eqref{intro:regularized-loss}, we construct its discrete counterpart from density observations
\begin{equation}\label{discretized version}
\left\{\rho(t_l,x_n):n=1,\ldots,N,\ l=1,\ldots,L\right\}.
\end{equation}
For the mesh-dependent analysis, we consider the corresponding formulation on a bounded one-dimensional domain $\Omega=[a,b]$, with a uniform space--time mesh
\begin{equation*}
\begin{aligned}
x_n = a + n\Delta x,\quad t_l = l\Delta t,\quad \Delta x = \frac{b - a}{N}, \quad \Delta t = \frac{T}{L},\quad n=1,\dots,N,\quad l=1,\dots,L.
\end{aligned}
\end{equation*}
Let $A^\delta$ denote the numerical approximation of $A$ defined in \eqref{neumericalA}, and let $f^\delta$ be the corresponding forcing approximation, defined in \eqref{eq: f_gradient} for Wasserstein gradient flows and \eqref{eq: f_Hamiltonian} for Wasserstein Hamiltonian flows. We study the following minimization problem
\begin{equation}\label{exp-pro}
(\widehat V_{\lambda,NL},\widehat W_{\lambda,NL}):=\operatorname*{arg\,min}_{(\phi,\psi)\in\mathcal H_{K_1}\times\mathcal H_{K_2}} \widehat R_{\lambda,NL}(\phi,\psi),
\end{equation}
where the regularized empirical loss functional is given by
\[
\begin{aligned}
\widehat R_{\lambda,NL}(\phi,\psi):=\frac{T|\Omega|}{NL}\sum_{l=1}^L\sum_{n=1}^N|A^\delta(\phi,\psi)(t_l,x_n)-f^\delta(t_l,x_n)|^2\rho_{t_l}(x_n)+\lambda_1\|\phi\|_{\mathcal H_{K_1}}^2+\lambda_2\|\psi\|_{\mathcal H_{K_2}}^2.
\end{aligned}
\]

We now summarize the outline and the main contributions of the paper. 
\begin{enumerate}
\item In Section~\ref{background}, we review the geometric framework underlying Wasserstein gradient and Hamiltonian flows, including Otto's metric and the associated differential operators on the density manifold. We present the gradient, Hamiltonian, and Lagrangian formulations needed to express the dynamics in terms of the unknown energy components.

\item In Section~\ref{Structure-preserving kernel ridge regression and numerical schemes}, we present an operator-theoretic framework that explicitly represents the solutions (Proposition \ref{Operator_representation}) of the optimization problem \eqref{exp-pro} in the product space of two RKHSs associated to two different kernel functions. Applying the so-called differential reproducing property \cite[Theorem 2.7]{RCSP2}, we derive a closed-form representation of the regularized estimators in Propositions~\ref{Operator_representation} and~\ref{Kernel representation}.  
Additionally, Section~\ref{subsec:joint-identifiability} characterizes ambiguities in the joint recovery of $V$ and $W$ and establishes a Fourier criterion for identifiability on the torus. This analysis distinguishes uniqueness of the regularized estimator from identifiability of the physical potentials and explains why nonstationarity alone does not ensure joint recovery.

\item In Section~\ref{error_analysis}, we analyze the convergence rates of our estimator to the ground truth by decomposing the error into three components: the approximation error, the mesh-induced discretization error, and the scheme-dependent numerical error. The approximation error analysis typically relies on standard source conditions~\cite{plato2018optimal,lu2019nonparametric}. Additionally, we introduce an auxiliary minimization problem to aid in the analysis.  
The mesh-induced discretization error (Theorem \ref{discretization-error}) and the scheme-dependent numerical error (Theorem \ref{numrical-error}) are then derived using the operator representations developed in Section~\ref{Structure-preserving kernel ridge regression and numerical schemes} and the regularity properties of the density flows.  

More precisely, if the regularization parameters \( \lambda_1 \) and \( \lambda_2 \), as well as the time and space discretization parameters \( L \) and \( N \) (introduced in \eqref{discretized version}), satisfy the following scaling conditions:
\begin{align*}
\lambda_1 = \lambda_2 \propto N^{-\alpha}, \quad L \propto N^{\beta}, \quad \alpha, \beta > 0,
\end{align*}
then for \( \alpha \in \left(0, \frac{1}{3}\right) \), the convergence rate of the total reconstruction error (Theorem~\ref{Total reconstruction error}) is given by:
\begin{equation*}
\Big\|\big(\widehat{V}_{\lambda, NL}, \widehat{W}_{\lambda, NL}\big) - \big(V, W\big)\Big\|_{\mathcal{H}_{K_1} \times \mathcal{H}_{K_2}} 
\propto 
\begin{cases}
N^{-\min\{\alpha\gamma, \frac{1}{2}(\beta - 3\alpha)\}}, & \quad 3\alpha < \beta \leq 1, \\[8pt]
N^{-\min\{\alpha\gamma, \frac{1}{2}(1 - 3\alpha)\}}, & \quad \beta > 1,
\end{cases}
\end{equation*}
where \( \gamma \) is the regularity parameter in the so-called {\it source condition} that will be introduced later on in \eqref{sou-con}.

For Wasserstein Hamiltonian flows with $U=0$, we use the push-forward formulation in \cite{chow2020wasserstein} to bound the Wasserstein distance between the true and learned density trajectories in terms of potential reconstruction errors, under the assumptions of the stability theorem. This extends the gradient-flow stability analysis in \cite{carrillo2025sparse} to the Hamiltonian setting and complements the Euclidean trajectory estimates for finite-dimensional systems in \cite{feng2021learning,RCSP2}.

\item In Section~\ref{Numerical experiments}, we demonstrate simultaneous reconstruction of $V$ and $W$ for one-dimensional Wasserstein gradient flows and two-dimensional Wasserstein Hamiltonian flows. The Hamiltonian experiments additionally use velocity-potential observations to construct the forcing through the first-order Hamiltonian system.
For each class, we also examine a nonstationary trajectory with unobservable interaction modes, showing that small equation residuals do not guarantee recovery of the underlying potentials when identifiability problems arise. These examples illustrate the identifiability analysis in Section~\ref{subsec:joint-identifiability}. Moreover, we compare our kernel method with the sparse-learning framework of \cite{carrillo2025sparse}, using polynomial, Fourier, and Gaussian dictionaries, and examine the dependence of sparse reconstruction on dictionary choice in the tested benchmark. Computations use finite kernel-center subspaces of the product RKHS, as detailed in Section~\ref{subsec:num-setting}. The code for all experiments is available at \url{https://github.com/jianyuhu/kernel-learning-for-Wasserstein-flows}.
\end{enumerate}

\paragraph{Related work.}
Wasserstein formulations arise in porous medium flows \cite{Carrillo2000AsymptoticLO,vazquez2007porous,otto2001geometry}, population dynamics \cite{bodnar2006integro,carrillo2019population}, and collective motion \cite{topaz2006nonlocal}. They connect to the heat equation \cite{Boltzmann1964LecturesOG}, Wasserstein geodesics \cite{benamou2000computational,carlen2003constrained}, Fokker--Planck equations \cite{Jordan1996THEVF,Furioli2017FokkerPlanckEI,Toscani2006KineticMO}, and geometric formulations related to the Vlasov and Schr\"odinger equations \cite{Carrillo2003KineticER,Lott2004RicciCF,chow2019discrete}, as well as Hamiltonian PDEs \cite{ambrosio2008hamiltonian,chow2020wasserstein}.

Data-driven identification of differential equations has been approached using sequentially thresholded least-square method~\cite{rudy2017data}, variants of LASSO-type algorithms~\cite{kang2021ident,rudy2019data}, iterative greedy algorithms such as subspace pursuit~\cite{he2023group}, gradient-descent-based algorithms designed for solving sparse regression or \(L^1\)-minimization problems~\cite{schaeffer2017learning} and energy variational approaches \cite{hu2024energetic,lu2024learning}.  
More directly related to the present problem, \cite{lang2022learning} learns interaction kernels in mean-field equations from discrete space--time density observations through regularized least squares on data-adaptive hypothesis spaces. The method in \cite{carrillo2025sparse} identifies nonlocal interaction potentials in nonlinear aggregation--diffusion equations using a quadratic variational loss and sparsity regularization over a prescribed dictionary, leading to a basis-pursuit formulation \cite{wright2022high}. That work also establishes reconstruction error bounds and stability estimates for the corresponding solutions in Wasserstein distance. In dictionary-based formulations, the choice of candidate functions affects the available representations and reconstruction accuracy. 

The self-test framework of \cite{gao2024self} constructs quadratic losses for weak-form operators depending linearly on the unknown parameters and analyzes identifiability and well-posedness through the associated Hessian. For joint estimation, it identifies a linear cancellation between the external and interaction potentials and imposes symmetry on the interaction potential to exclude this particular ambiguity. Our identifiability analysis characterizes trajectory-dependent cancellation directions within the chosen product RKHS and gives a Fourier criterion for joint identifiability on the torus.

RKHS methods provide a well-established framework for nonparametric approximation, statistical learning, and regularized inverse problems \cite{scholkopf2002learning,steinwart2008support}. Representer-type results often reduce regularized optimization over an infinite-dimensional RKHS to a finite-dimensional problem determined by finitely many observations. These results extend beyond point evaluations to bounded linear functionals and operator-based observations \cite{saitoh2016theory,schaback2006kernel}. Differential observations can be treated through operator-induced kernels or differential reproducing identities \cite{zhou2008derivative}. Kernel methods involving differential operators have also been extensively studied in PDE approximation \cite{franke1998solving,wendland2004scattered,schaback2007convergence,schaback2016all}, including recent physics-informed formulations \cite{doumeche2024physics,batlle2025error}. Regularization and source conditions provide established tools for analyzing ill-posedness and convergence \cite{de2006discretization,blanchard2018optimal}.

Structure-preserving kernel methods have also been developed for finite-dimensional gradient, Hamiltonian, and Poisson dynamical systems \cite{feng2021learning,miller2023learning,RCSP2,RCSP3,CampoloHuOrtegaYin2025}. These methods recover potentials or Hamiltonians by exploiting the prescribed geometric relation between the energy and the dynamics. We follow this approach for Wasserstein flows, where the unknown external and interaction potentials enter through density-dependent differential and convolutional operators. Our reconstruction analysis builds on Hilbert-space regularization theory and accounts for the discretization of the density-dependent observation operator and forcing. In addition, the Hamiltonian stability estimate for $U=0$ extends the gradient-flow analysis in \cite{carrillo2025sparse} by controlling Wasserstein distances between density trajectories in terms of potential reconstruction errors.

\section{Wasserstein gradient and Hamiltonian flows}\label{background}

In this section, we define some elements of calculus on the density manifold and present the geometric structures underlying Wasserstein gradient and Hamiltonian flows defined on it. These geometric frameworks have been extensively studied in the literature; see the classical books \cite{villani2009optimal, ambrosio2008gradient}, as well as contributions such as \cite{otto2001geometry,lott2008some, chow2020wasserstein}. 
In Section~\ref{Hamiltonian structure}, we present a symplectic form on what we call the characterization space $\mathcal{P}_+(M) \times C^\infty(M)/\mathbb{R}$ (see below), that provides a geometric perspective on the Wasserstein Hamiltonian flows discussed in \cite{chow2020wasserstein, Wu2025}. 
This symplectic structure was introduced in \cite{khesin2019geometry} in the study of the geometry of the Madelung transform.
As explained later in Appendix \ref{apped: density manifold intro}, this geometric approach differs from the Hamiltonian structure on Wasserstein spaces in \cite{ambrosio2008hamiltonian, Gangbo2011} and from the Poisson structure on the density manifold \cite{lott2008some}. {For completeness, we provide details in Appendix~\ref{apped: density manifold intro}.}

\subsection{Calculus on the density manifold}

We first recall the density-manifold framework on a compact, connected manifold. The learning problem is subsequently formulated on Euclidean and periodic domains.

Let $(M,g)$ be a smooth, compact, connected Riemannian manifold without boundary of dimension $d$, where $g$ denotes the Riemannian metric tensor. This metric defines a volume measure $d\operatorname{vol}_M $ (see \cite{do:carmo:1993}) that is locally expressed as
\begin{align*}
d\operatorname{vol}_M = \sqrt{\operatorname{det}(g_{ij})}\, dx^1\cdots dx^d,
\end{align*}
where $g_{ij}$ is the local representation of the metric tensor $g$ in a set of local coordinates $(x^1,\dots,x^d)$ of $M$.

We define the probability density manifold $\mathcal{P}_{+}(M)$ as the set of densities on $M$ given by
\begin{align*}
\mathcal{P}_{+}(M) = \left\{ \rho \in C^\infty(M)\mid \rho > 0,\, \int_M \rho\, d\operatorname{vol}_M = 1 \right\},
\end{align*}
{This space can be endowed with the structure of an infinite-dimensional Fr\'echet manifold in the convenient setting introduced in \cite{kriegl1997convenient} (see \cite{lott2008some}).} For each $\rho \in \mathcal{P}_{+}(M) $, the tangent space $T_{\rho}\mathcal{P}_{+}(M) $ can be written as  
\begin{align*}
T_{\rho}\mathcal{P}_{+}(M)= \left\{\sigma\in C^\infty(M)~\bigg|~\int_M \sigma\, d\operatorname{vol}_M = 0\right\}.
\end{align*}

Denote by $C^\infty(M)/\mathbb{R}$ the space of smooth functions on $M$ defined up to the addition of constants. For any $\rho\in \mathcal{P}_{+}(M) $ and  $\Phi \in C^\infty(M)$ define the operator ${h}_\rho: C^\infty(M) \rightarrow C^\infty(M)$ by 
\begin{align}\label{V-rho}
{h}_\rho(\Phi) := -\nabla \cdot (\rho\nabla\Phi).
\end{align}
This operator $h_\rho$, referred to as the {\it identification map} in \cite{otto2001geometry}, drops to a linear isomorphism from the quotient space $C^\infty(M)/\mathbb{R}$ to the tangent space $T_{\rho}\mathcal{P}_{+}(M)$ (see also \cite{lott2008some}).

Consider again the {\it weighted Laplacian} $\Delta_\rho := \nabla\cdot(\rho\nabla)$ and denote by $\left(-\Delta_\rho\right)^{\dagger}: T_{\rho} \mathcal{P}_{+}(M) \rightarrow C^\infty(M)/\mathbb{R}$ the pseudo inverse operator of $\left(-\Delta_\rho\right)$. It is straightforward to verify that
\[
(-\Delta_\rho)^\dagger (-\Delta_\rho)(-\Delta_\rho)^\dagger = (-\Delta_\rho)^\dagger.
\]
Given $\sigma_1 = h_{\rho}(\Phi_1)$ and $\sigma_2 = h_{\rho}(\Phi_2)$, we have 
\begin{align*}
\int_M \langle\nabla \Phi_1, \nabla \Phi_2\rangle\, \rho\, d\operatorname{vol}_M
&= \int_M \Phi_1(-\Delta_\rho)\Phi_2\, d\operatorname{vol}_M = \int_M h_{\rho}(\Phi_1)(-\Delta_\rho)^\dagger(-\Delta_\rho)(-\Delta_\rho)^\dagger h_{\rho}(\Phi_2)\, d\operatorname{vol}_M\\
&= \int_M \sigma_1(-\Delta_\rho)^\dagger\sigma_2\, d\operatorname{vol}_M.
\end{align*}
This identity leads to the following important definition. 

\begin{definition}
[{\bf Otto's Riemannian metric}] 
\label{Otto metric}
Otto's Riemannian metric {$g_W(\rho) : T_{\rho}\mathcal{P}_{+}(M)\times T_{\rho}\mathcal{P}_{+}(M)\rightarrow \mathbb{R}$ for each $\rho\in\mathcal{P}_{+}(M) $ is defined as}
\begin{align}\label{inn-pro}
g_W(\rho)(\sigma_1,\sigma_2) := \int_M \langle\nabla \Phi_1, \nabla \Phi_2\rangle_g\, \rho\, d\operatorname{vol}_M=\int_M \sigma_1(-\Delta_\rho)^\dagger\sigma_2\, d\operatorname{vol}_M,
\end{align}
for any $\rho \in \mathcal{P}_{+}(M)  $  and for any tangent vectors $\sigma_1=h_{\rho}(\Phi_1), \sigma_2=h_{\rho}(\Phi_2)\in T_{\rho}\mathcal{P}_{+}(M)$.
\end{definition}
It can be shown that the geodesic distance on $\mathcal{P}_{+}(M) $ associated to Otto's Riemannian metric is the Wasserstein two-metric $W _2 $ in that space {(see \cite[Lecture 18]{ambrosio2021lectures}, \cite{benamou2000computational} or \cite{otto2001geometry} ).}

The cotangent bundle of the density manifold is defined as
\begin{align*}
T^*\mathcal{P}_{+}(M) := \bigcup_{\rho \in \mathcal{P}_{+}(M)} T^*_{\rho}\mathcal{P}_{+}(M),
\end{align*}
where $T^*_{\rho}\mathcal{P}_{+}(M):=L(T_{\rho}\mathcal{P}_{+}(M),\, \mathbb{R})$ denotes the space of continuous linear functionals on $T_{\rho}\mathcal{P}_{+}(M)$. 
We represent a subspace of the cotangent space $T^*_{\rho}\mathcal{P}_{+}(M)$ at a given density $\rho\in\mathcal{P}_{+}(M)$ by defining the map $f _\rho : C^\infty(M)/\mathbb{R} \to T^*_{\rho}\mathcal{P}_{+}(M)$ by
\begin{align}\label{iso}
\langle f _\rho(\Phi), \sigma\rangle := g_W(\rho)(\sigma, -\Delta_\rho \Phi), \quad \forall\, \sigma \in T_{\rho}\mathcal{P}_{+}(M),
\end{align}
where $\langle\cdot, \cdot\rangle$ denotes the pairing between the tangent and cotangent spaces, and $g_W$ is Otto's Riemannian metric as defined in \eqref{inn-pro}. { Since $g_W$ is a weak Riemannian metric, the associated map $\flat_\rho: T_{\rho}\mathcal{P}_{+}(M)\to T^*_{\rho}\mathcal{P}_{+}(M) $ with $\sigma \mapsto g_W(\rho)(\sigma, \cdot)$
is not surjective. Its image defines a proper subspace of the cotangent space.} 

Propositions \ref{vrho map} and \ref{characterization} prove that both the tangent $T\mathcal{P}_{+}(M) $ and {the subspace $\mathrm{Im}(\flat)$ of cotangent bundle $T^\ast \mathcal{P}_{+}(M)$} are vector bundle isomorphic to the trivial bundle $\mathcal{P}_+(M)\times C^\infty(M)/\mathbb{R}$ via the maps $h: \mathcal{P}_+(M)\times C^\infty(M)/\mathbb{R}\longrightarrow T\mathcal{P}_{+}(M) $ and $f: \mathcal{P}_+(M)\times C^\infty(M)/\mathbb{R}\longrightarrow T^\ast \mathcal{P}_{+}(M) $ defined by
\begin{equation}
\label{bundle isomorphisms}
h(\rho,\Phi):=(\rho,h_\rho(\Phi))\quad\mbox{and}\quad f(\rho,\Phi):=(\rho,f_\rho(\Phi)),
\end{equation}
respectively, with $h _\rho:C^\infty(M)/\mathbb{R} \longrightarrow  T_\rho\mathcal{P}_{+}(M)$ and $f_\rho:C^\infty(M)/\mathbb R\longrightarrow\operatorname{Im}(\flat_\rho)$ the maps defined in \eqref{V-rho} and \eqref{iso}. In light of these facts, we will refer to the trivial bundle $\mathcal{P}_+(M)\times C^\infty(M)/\mathbb{R} $ as the {\it characterization space} of both $T\mathcal{P}_{+}(M)$ and the regular cotangent subbundle $\operatorname{Im}(\flat)$.

Now note that from \eqref{intermediate with gw} in Proposition \ref{characterization}, it follows that for every \( \sigma \in T_{\rho}\mathcal{P}_{+}(M) \) and \(\alpha\in\operatorname{Im}(\flat_\rho)\), there exists a unique \( \Phi \in C^\infty(M)/\mathbb{R} \) such that \( \Phi = f_\rho^{-1}(\alpha) \) and:
\begin{align}
\label{l2pairing and dual}
\langle \alpha, \sigma \rangle = g_W(\rho)(\sigma, -\Delta_\rho \Phi) = \langle \sigma, \Phi \rangle_{L^2}.
\end{align}
Note that the $L ^2 $-pairing in the right hand side of \eqref{l2pairing and dual} is well-defined on $C^\infty(M)/\mathbb{R}$ because $\int_M \sigma\, d\operatorname{vol}_M = 0 $.
The identity \eqref{l2pairing and dual} implies that the \( L^2 \)-duality between the tangent space \( T_{\rho}\mathcal{P}_{+}(M) \) and the space \( C^\infty(M)/\mathbb{R} \) can be used to characterize the duality between the tangent spaces \( T_{\rho}\mathcal{P}_{+}(M) \) and the regular cotangent spaces $\operatorname{Im}(\flat_\rho)$.

In particular, let  \( \mathcal{F} \in C^\infty(\mathcal{P}_{+}(M)) \) be a smooth functional admitting a smooth first variation. We define the {\it functional derivative} \( \delta \mathcal{F}/\delta \rho \in C^\infty(M)/\mathbb{R} \) of $\mathcal{F} $ at $\rho $, and it is defined as the unique element that by \eqref{l2pairing and dual} satisfies that
\begin{align}
\label{def funct der}
\left\langle \mathbf{d}\mathcal{F}(\rho), \sigma \right\rangle=\lim_{\epsilon \to 0} \frac{1}{\epsilon} \big[\mathcal{F}(\rho + \epsilon \sigma) - \mathcal{F}(\rho)\big] = \left\langle \frac{\delta \mathcal{F}}{\delta \rho}, \sigma \right\rangle_{L^2}, \quad \forall\, \sigma \in T_{\rho}\mathcal{P}_{+}(M).
\end{align}
We now define the {\it Wasserstein gradient} $\operatorname{grad}_W \mathcal{F}(\rho) \in T _\rho \mathcal{P}_{+}(M)$ of such a functional $\mathcal F$ as the standard gradient with respect to Otto's metric, that is, $\operatorname{grad}_W \mathcal{F}(\rho)$ is the unique element that satisfies the equality
\begin{equation*}
g_W(\rho)(\operatorname{grad}_W \mathcal{F}(\rho), \sigma)=\left\langle \mathbf{d}\mathcal{F}(\rho), \sigma \right\rangle, \quad \forall\, \sigma \in T_{\rho}\mathcal{P}_{+}(M).
\end{equation*}
As a result of \eqref{l2pairing and dual} and \eqref{def funct der} we can write:
\begin{align}
\label{wasser gradient}
{\operatorname{grad}_W \mathcal{F}(\rho) =- \nabla \cdot \left( \rho \nabla \frac{\delta \mathcal{F}}{\delta \rho} \right).}
\end{align}

\begin{example}[{\bf Linear potential energy}]\normalfont
Consider the linear potential energy functional
\begin{align*}
\mathcal{F}(\rho) = \int_{M} V \rho\, d\operatorname{vol}_M.
\end{align*}
The Wasserstein gradient of this energy is given by {$
\operatorname{grad}_W \mathcal{F}(\rho) = -\Delta_\rho V$},
since for all $\sigma \in  T_{\rho}\mathcal{P}_{+}(M)$ we have that:
\begin{align*}
\left\langle \mathbf{d}\mathcal{F}(\rho), \sigma \right\rangle=\lim_{\epsilon\to 0}\frac{1}{\epsilon}\left[\mathcal{F}(\rho + \epsilon \sigma)-\mathcal{F}(\rho)\right]
= \lim_{\epsilon\to 0}\frac{1}{\epsilon}\int_M V(\rho+\epsilon \sigma - \rho)\, d\operatorname{vol}_M
= \int_M V\sigma\, d\operatorname{vol}_M,
\end{align*}
which guarantees that $\frac{\delta \mathcal{F}}{\delta \rho}=V $ and implies that {$
\operatorname{grad}_W \mathcal{F}(\rho) = -\Delta_\rho V$} by \eqref{wasser gradient}.
\end{example}

\begin{example}[{\bf Interaction energy}]\normalfont
For the interaction terms below, take $M=\mathbb T^d$ and assume $W(z)=W(-z)$. Consider the interaction energy functional defined by
\begin{align*}
\mathcal{F}(\rho) = \frac{1}{2}\iint_{M\times M} W(x - y)\rho(x)\rho(y)\, d\operatorname{vol}_M(x)\, d\operatorname{vol}_M(y).
\end{align*}
Then, the Wasserstein gradient of this energy is {$\operatorname{grad}_W \mathcal{F}(\rho) = -\Delta_\rho(W*\rho)$},
where the convolution $W*\rho$ of the interaction kernel function $W$ with the density $\rho$ is defined as
\begin{align*}
(W*\rho)(x) := \int_{M} W(x - y)\rho(y)\, d\operatorname{vol}_M(y).
\end{align*}
Indeed,
\begin{align*}
&\left\langle \mathbf{d}\mathcal{F}(\rho), \sigma \right\rangle=\lim_{\epsilon\to 0}\frac{1}{\epsilon}\left[\mathcal{F}(\rho + \epsilon\sigma)-\mathcal{F}(\rho)\right] \\
&= \lim_{\epsilon\to 0}\frac{1}{2\epsilon}\iint_{M\times M} W(x - y)\left[(\rho(x)+\epsilon\sigma(x))(\rho(y)+\epsilon\sigma(y)) - \rho(x)\rho(y)\right] d\operatorname{vol}_M(x)d\operatorname{vol}_M(y)\\
&= \int_M (W*\rho)\sigma\, d\operatorname{vol}_M,
\end{align*}
which implies that the functional derivative is given by $\frac{\delta \mathcal{F}}{\delta \rho}=W*\rho$.
\end{example}

\begin{example}[{\bf Nonlinear internal energy}]\normalfont
Consider the nonlinear internal energy functional
\begin{align*}
\mathcal{F}(\rho) = \int_{M} U(\rho)\, d\operatorname{vol}_M.
\end{align*}
Then, the Wasserstein gradient is given by
{$\operatorname{grad}_W \mathcal{F}(\rho) = -\Delta_\rho U'(\rho)$},
since the equalities
\begin{align*}
\left\langle \mathbf{d}\mathcal{F}(\rho), \sigma \right\rangle=\lim_{\epsilon\to 0}\frac{1}{\epsilon}\left[\mathcal{F}(\rho+\epsilon\sigma)-\mathcal{F}(\rho)\right]
= \lim_{\epsilon\to 0}\frac{1}{\epsilon}\int_M \left[U(\rho+\epsilon\sigma)-U(\rho)\right] d\operatorname{vol}_M = \int_M U'(\rho)\sigma\, d\operatorname{vol}_M,
\end{align*}
imply that the functional derivative is given by $\frac{\delta \mathcal{F}}{\delta \rho}=U'(\rho)$.
\end{example}

\subsection{The Wasserstein gradient flow}
Many differential systems arising in physics can be formulated via a potential energy functional $\mathcal{F}$ and a linear operator relating the system's time derivatives to gradients of this potential energy. Let $\mathcal{F}:\mathcal{P}_{+}(M)\to \mathbb{R}$ be a potential energy functional. The corresponding {\it Wasserstein gradient flow} is the solution of the differential equation
\begin{align*}
\partial_t \rho_t = -\operatorname{grad}_W \mathcal{F}(\rho_t).    
\end{align*}
Equations of this form, derived from various choices of energy functionals, are deeply connected to well-studied partial differential equations; see~\cite[pp.~430--433]{villani2009optimal} for further details. In physical modeling, the energy functional $\mathcal{F}$ often comprises internal, potential, and interaction terms, typically given by
\begin{equation}\label{potential-energy}
    \mathcal{F}(\rho) = \int_M U(\rho)\,d\operatorname{vol}_M 
    + \int_M V\rho\,d\operatorname{vol}_M 
    + \frac{1}{2}\iint_{M\times M} W(x - y)\rho(x)\rho(y)\,d\operatorname{vol}_M(x)d\operatorname{vol}_M(y).
\end{equation}
Below, we provide illustrative examples of Wasserstein gradient flows.
Table~\ref{table: gradient flows} summarizes some representative energy functionals and their associated gradient flow equations.

\begin{table}
\caption{Examples of the energy functionals and their corresponding gradient flow equations}
\label{table: gradient flows}
\centering
\begin{center}
\begin{tblr}{
vline{2,3} = {-}{},
hline{1,6} = {-}{0.1em},
hline{2,3,4,5} = {-}{},
}
& Energy functional & Gradient flows    \\
1 & $\mathcal{F}(\rho)= \int_{M}\rho\log\rho d{\operatorname{vol}}_M$ & $\partial_t \rho_t=\Delta \rho_t$   \\
2& $\mathcal{F}(\rho)= \int_{M}\rho\log\rho d{\operatorname{vol}}_M+\int_{M}V\rho d{\operatorname{vol}}_M$ & $\partial_t \rho_t=\Delta \rho_t+\Delta_{\rho_t}V$    \\
3& $\mathcal{F}(\rho)=\frac{1}{m-1} \int_{M}\rho^m d{\operatorname{vol}}_M$ & $\partial_t \rho_t=\Delta \rho_t^m$  \\
4& $\mathcal{F}(\rho)= \frac12\iint_{M\times M}W(x-y)\rho(x)\rho(y)d{\operatorname{vol}}_M(x)d{\operatorname{vol}}_M(y)$ & $\partial_t \rho_t=\Delta_{\rho_t} (W*\rho_t)$   
\end{tblr}
\end{center}
\end{table}

\subsection{The Wasserstein Hamiltonian flow}
The {\it Wasserstein Hamiltonian flow} has been introduced in \cite{chow2020wasserstein,Wu2025}. The standard way to define Hamilton's equations and their associated flows in geometric mechanics \cite{Abraham1978, Marsden1994} consists of associating Hamiltonian vector fields to Hamiltonian (energy) functions using a predefined symplectic structure. When the phase space of the physical system is the cotangent bundle of a finite-dimensional manifold, a canonical symplectic form is always available, which is obtained as the differential of the so-called Liouville one-form (see \cite{Abraham1978} for details). In infinite-dimensional settings like ours, the situation is technically more convoluted since many equivalent definitions of differential forms in finite dimensions become inequivalent; see \cite[Section 33]{kriegl1997convenient} and \cite[Appendix E]{schmeding2022introduction} for more details. 


In this section, we shall proceed by defining a symplectic structure on the characterization space $\mathcal{P}_+(M) \times C^\infty(M)/\mathbb{R}$ of the regular cotangent subbundle introduced in the previous section. We will show that the resulting Hamilton's equations coincide with those of the Wasserstein Hamiltonian flow in \cite{chow2020wasserstein,Wu2025}.
This symplectic structure was presented in \cite{khesin2019geometry} to study the geometry of the Madelung transform.
Additionally, employing the Legendre transformation, Hamilton's equations will also be written in Lagrangian form.

\subsubsection{A symplectic structure on the characterization space}\label{Hamiltonian structure}
As we just pointed out, endowing the cotangent bundle $T^*\mathcal{P}_{+}(M)$ with a canonical symplectic form is not straightforward due to the infinite-dimensional character of the density manifold $\mathcal{P}_{+}(M)$. We will instead proceed by using the vector bundle isomorphisms introduced in \eqref{bundle isomorphisms} that allow us to represent both $T\mathcal{P}_{+}(M)$ and $\operatorname{Im}(\flat)$ by the trivial bundle $\mathcal{P}_{+}(M) \times C^\infty(M)/\mathbb{R}$.

We start by noticing that since the space $C^\infty(M)/\mathbb{R}$ is a vector space, we can hence regard its tangent space  at any point $\Phi\in C^\infty(M)/\mathbb{R}$  as the space $C^\infty(M)/\mathbb{R}$ itself. Consequently, the tangent space $T_{(\rho,\Phi)}(\mathcal{P}_{+}(M) \times C^\infty(M)/\mathbb{R})$ at point $(\rho,\Phi)$ can be written as $T_{\rho}\mathcal{P}_{+}(M) \times C^\infty(M)/\mathbb{R}$. 

Define now the 2-form $\omega\in \Omega^2 \left( \mathcal{P}_{+}(M) \times C^\infty(M)/\mathbb{R})\right)$ given by
\begin{align}\label{symplectic-form}
\omega = \int_{M} d\Phi\wedge d\rho ~d{\operatorname{vol}_M}.  
\end{align}
More specifically, \eqref{symplectic-form} means that for any two elements $(\sigma_1,\psi_1),(\sigma_2,\psi_2)\in T_{(\rho,\Phi)}(\mathcal{P}_{+}(M) \times C^\infty(M)/\mathbb{R})\simeq T_{\rho}\mathcal{P}_{+}(M) \times C^\infty(M)/\mathbb{R}$, we have
\begin{align*}
\omega(\rho,\Phi)((\sigma_1,\psi_1),(\sigma_2,\psi_2)) =\int_{M}(\psi_1\sigma_2-\psi_2\sigma_1) d{\operatorname{vol}_M},
\end{align*}
as $d\rho(\sigma_i,\psi_i)=\sigma_i$ and $d\Phi(\sigma_i,\psi_i)=\psi_i$ for $i=1,2$. 
The following result shows that this 2-form is weakly symplectic, which makes the representation space $\mathcal{P}_{+}(M) \times C^\infty(M)/\mathbb{R}$ into a weak symplectic manifold.
 
\begin{theorem}\label{thm:symplectic form}
The 2-form in \eqref{symplectic-form} is a {weak} symplectic form on the characterization space $\mathcal{P}_{+}(M) \times C^\infty(M)/\mathbb{R}$.   Moreover, using Otto’s Riemannian metric, the symplectic form can be expressed as
\begin{align*}
    \omega(\rho,\Phi)\big((\sigma_1, \psi_1), (\sigma_2, \psi_2)\big) = g_W(\rho)(\sigma_2,-\Delta_\rho\psi_1)-g_W(\rho)(\sigma_1,-\Delta_\rho\psi_2),
\end{align*}
for any $(\sigma_1,\psi_1),(\sigma_2,\psi_2)\in T_{(\rho,\Phi)}(\mathcal{P}_{+}(M) \times C^\infty(M)/\mathbb{R})$.
\end{theorem}

Let $\mathcal{H} : \mathcal{P}_{+}(M) \times C^\infty(M)/\mathbb{R} \to \mathbb{R}$ be a Hamiltonian functional on the characterization space admitting smooth variational derivatives. The associated Hamiltonian vector field $X_{\mathcal{H}} \in \mathfrak{X}\left(\mathcal{P}_{+}(M) \times C^\infty(M)/\mathbb{R}\right)$ is determined by the equality
\begin{align*}
    \omega(\rho, \Phi)\big((\sigma,\psi),X_{\mathcal H}(\rho,\Phi)\big) = d\mathcal{H}(\rho, \Phi) \cdot (\sigma, \psi),
\end{align*}
for all $(\sigma, \psi) \in T_\rho \mathcal{P}_{+}(M) \times C^\infty(M)/\mathbb{R}$. This relation yields
\begin{align*}
    X_{\mathcal{H}}(\rho, \Phi) = \left( \frac{\delta \mathcal{H}}{\delta \Phi},\, -\frac{\delta \mathcal{H}}{\delta \rho} \right).
\end{align*}
Consequently, the Hamiltonian system associated with $\mathcal{H}$ can be written as
\begin{equation}\label{Ham-for}
\begin{cases}
    \partial_t \rho  = \frac{\delta \mathcal{H}}{\delta \Phi}(\rho, \Phi), \\[6pt]
    \partial_t \Phi =  -\frac{\delta \mathcal{H}}{\delta \rho}(\rho, \Phi).
\end{cases}
\end{equation}
The solutions of \eqref{Ham-for} determine what we call the \textit{Wasserstein Hamiltonian flow}. This formulation is consistent with the Hamiltonian equations presented in \cite{chow2020wasserstein, Wu2025}.
In physical modeling, the Hamiltonian functional consists in many situations of a kinetic energy plus a potential energy term:
\begin{align}\label{kin+pon}
\mathcal{H}(\rho, \Phi) = \frac{1}{2} \int_M |\nabla \Phi|_g^2\, \rho\, d\operatorname{vol}_M + \mathcal{F}(\rho),
\end{align}
where the potential energy functional $\mathcal{F}$ is typically taken as in \eqref{potential-energy}. With this choice, Hamilton's equations \eqref{Ham-for} reduce to
\begin{equation}\label{Ham-for1}
\begin{cases}
    \partial_t \rho_t  = -\nabla\cdot(\rho_t\nabla\Phi_t), \\[6pt]
    \partial_t \Phi_t =  -\frac{1}{2}|\nabla\Phi_t|_g^2-\frac{\delta }{\delta \rho_t}\mathcal{F}(\rho_t).
\end{cases}
\end{equation}
We now present several illustrative examples of Wasserstein Hamiltonian flows (\cite{chow2020wasserstein} for additional details).

\begin{example}[{\bf Linear Vlasov Equation}]\normalfont
Consider the linear Vlasov equation
\[
    \partial_t f + v\cdot\nabla_x f - \nabla V(x)\cdot\nabla_v f = 0,
\]
describing a particle density function $f(t,x,v)$ evolving on $[0,T]\times\mathbb{T}_x^d\times\mathbb{R}_v^d$, where $\mathbb{T}^d $ is a $d$-torus and  $V \in C^{\infty}(\mathbb{T}^d) $ is a potential function defined on it. For monokinetic solutions $f(t,x,v)=\rho(t,x)\delta(v-\nabla\Phi(t,x))$, the density and velocity potential satisfy~\eqref{Ham-for1} with linear potential energy $\mathcal F(\rho)=\int_{\mathbb T^d}V(x)\rho(x)\,dx$, as long as this smooth, single-valued potential-flow representation remains valid.
\end{example}

\begin{example}[{\bf Linear Schr\"odinger Equation}]\normalfont
Consider the Schr\"odinger equation
\[
    i\partial_t\Psi = -\frac{1}{2}\Delta\Psi + V(x)\Psi.
\]
Where $\rho>0$ and a smooth phase is available, the Madelung (Bohm) transform $\Psi=\sqrt\rho\,e^{i\Phi}$ gives the hydrodynamic system
\[
\begin{cases}
\partial_t\rho_t + \nabla\cdot(\rho_t\nabla\Phi_t) = 0,\\[6pt]
\partial_t\Phi_t + \frac{1}{2}|\nabla\Phi_t|^2 = -V+\frac12\frac{\Delta\sqrt{\rho_t}}{\sqrt{\rho_t}},
\end{cases}
\]
corresponding precisely to the Hamiltonian flow~\eqref{Ham-for1}, with energy functional including linear potential energy and Fisher information:
\[
    \mathcal{F}(\rho) = \int_{\mathbb{T}^d}V(x)\rho(x)\,dx + \frac{1}{8}\int_{\mathbb{T}^d}|\nabla\log\rho(x)|^2\rho(x)\,dx.
\]

\end{example}

\noindent {\bf Poisson structure.}
The symplectic form \eqref{symplectic-form} enables us to define a canonical  bracket 
$$\{\cdot,\cdot\}: C^\infty\left(\mathcal{P}_+(M)\times C^\infty(M)/\mathbb{R}\right)\times C^\infty\left(\mathcal{P}_+(M)\times C^\infty(M)/\mathbb{R}\right)\to C^\infty\left(\mathcal{P}_+(M)\times C^\infty(M)/\mathbb{R}\right),$$ which is given by:
\begin{align}\label{Poi-bra}
\{\mathcal{H},\mathcal{G}\} := \omega\left(X_{\mathcal{H}},X_{\mathcal{G}}\right)=\int_M \left(\frac{\delta \mathcal{H}}{\delta\rho} \frac{\delta \mathcal{G}}{\delta\phi} - \frac{\delta \mathcal{H}}{\delta\phi} \frac{\delta \mathcal{G}}{\delta\rho} \right)d\mathrm{vol}_M,
\end{align}
for all $\mathcal{H},\mathcal{G}\in C^\infty\left(\mathcal{P}_+(M)\times C^\infty(M)/\mathbb{R}\right)$. The following Corollary, whose proof can be found in Appendix \ref{Proof of Corollary Poisson}, shows that this bracket endows $C^\infty\left(\mathcal{P}_+(M)\times C^\infty(M)/\mathbb{R}\right)$ with a Poisson algebra structure.

\begin{corollary}
\label{Poissons structure}
The bracket defined in \eqref{Poi-bra} is a Poisson bracket on $C^\infty\left(\mathcal{P}_+(M)\times C^\infty(M)/\mathbb{R}\right)$.   
\end{corollary}

\subsubsection{The Wasserstein Hamiltonian flow in Lagrangian form}

We now introduce the Legendre inverse transformation, which maps the characterization space $\mathcal{P}_{+}(M) \times C^\infty(M)/\mathbb{R}$ to the tangent bundle $T\mathcal{P}_{+}(M)$, thereby allowing us to reformulate the Wasserstein Hamiltonian flow in Lagrangian form.

\begin{definition}
Let $\mathcal{L}$ be a Lagrangian defined on the tangent bundle $T\mathcal{P}_{+}(M)$. The corresponding Hamiltonian $\mathcal{H}: \mathcal{P}_{+}(M)\times C^\infty(M)/\mathbb{R} \to \mathbb{R}$ is given by the Legendre transformation:
\begin{align*}
    \mathcal{H}(\rho, \Phi) := \sup_{\sigma \in T_{\rho}\mathcal{P}_{+}(M)} \left\{ \left\langle \Phi, \sigma \right\rangle_{L^2} - \mathcal{L}(\rho, \sigma) \right\}.
\end{align*}
For the kinetic-minus-potential Lagrangians considered below, the inverse Legendre transformation is
\begin{align*}
    \mathcal{L}(\rho, \sigma) := \sup_{\Phi \in C^\infty(M)/\mathbb{R}} \left\{ \left\langle \Phi, \sigma \right\rangle_{L^2} - \mathcal{H}(\rho, \Phi) \right\}.
\end{align*}
\end{definition}

\begin{remark}\normalfont
When the relevant derivatives exist and the corresponding extrema are attained, the Legendre transformation can be expressed as
\begin{align*}
    \Phi = \frac{\delta \mathcal{L}}{\delta \sigma}, \qquad \mathcal{H}(\rho, \Phi) = \left\langle \Phi, \sigma \right\rangle_{L^2} - \mathcal{L}(\rho, \sigma),
\end{align*}
where the variable $\sigma $ in the right-hand side of the second equality has to be understood as a function of the variables $\rho$ and $\Phi$ determined by the first equality (whenever that is possible).
Similarly, the inverse Legendre transformation takes the form
\begin{align*}
    \sigma = \frac{\delta \mathcal{H}}{\delta \Phi}, \qquad \mathcal{L}(\rho, \sigma) = \left\langle \Phi, \sigma \right\rangle_{L^2} - \mathcal{H}(\rho, \Phi),
\end{align*}
where the variable $\Phi $ in the right-hand side of the second equality has to be understood as a function of the variables $\rho$ and $\sigma $ determined by the first equality (whenever that is possible).
\end{remark}

Using the definition of the inverse Legendre transformation, the corresponding Lagrangian of the Hamiltonian functional \eqref{kin+pon} can be expressed as
\begin{align*}
\mathcal{L}(\rho, \sigma) &= \sup_{\Phi \in C^\infty(M)/\mathbb{R}} \left\{ \left\langle\Phi,\sigma\right\rangle_{L^2} - \frac{1}{2}\int_M |\nabla \Phi|_g^2\, \rho\, d\operatorname{vol}_M - \mathcal{F}(\rho) \right\} \\
&= \sup_{\Phi \in C^\infty(M)/\mathbb{R}} \left\{ \int_M \left(\langle \nabla \Phi, \nabla((-\Delta_\rho)^\dagger \sigma) \rangle_g - \frac{1}{2} |\nabla \Phi|_g^2\right)\, \rho\, d\operatorname{vol}_M \right\} - \mathcal{F}(\rho) \\
&= \frac{1}{2} \int_M |\nabla((-\Delta_\rho)^\dagger \sigma)|_g^2\, \rho\, d\operatorname{vol}_M - \mathcal{F}(\rho).
\end{align*}
Accordingly, the Wasserstein Hamiltonian flow \eqref{Ham-for1} can be written in Lagrangian form as
\begin{equation*}
\partial_{tt} \rho_t + \Gamma_W(\partial_t \rho_t, \partial_t \rho_t) = -\operatorname{grad}_{W} \mathcal{F}(\rho_t),
\end{equation*}
where $\operatorname{grad}_{W}$ denotes the Wasserstein gradient operator and $\Gamma_W$ is the so-called Christoffel symbol, given explicitly by
\begin{align*}
    \Gamma_W(\partial_t \rho_t, \partial_t \rho_t) 
    = -\left(\Delta_{\partial_t \rho_t} \Delta_{\rho_t}^\dagger \partial_t \rho_t 
    + \frac{1}{2} \Delta_{\rho_t} \left| \nabla \Delta_{\rho_t}^\dagger \partial_t \rho_t \right|^2 \right).
\end{align*}
This expression corresponds to the Levi-Civita connection associated with the Otto metric (see \cite[Proposition 11]{li2018geometry} for more details).

\section{Structure-preserving kernel ridge regression}\label{Structure-preserving kernel ridge regression and numerical schemes}

In this section, we propose a kernel-based approach for the simultaneous recovery of the potential $V$ and interaction kernel $W$ in Wasserstein gradient and Hamiltonian flows from trajectory data of the density flows
\begin{align}\label{data}
\left\{\rho(t_l,x_n), n=1,\cdots, N, l=1,\cdots, L\right\},  
\end{align}
where $(t_l,x_n)$ is a time and space mesh, typically chosen to be uniform within the domain. The most important feature of the methodology that we propose is that it is intrinsically structure-preserving, that is, the recovered system will possess prescribed underlying geometric structures of the type introduced in the previous section despite potential estimation and approximation errors committed by the method.
It is important to note that this setting is applicable to both Wasserstein gradient and Hamiltonian flow systems. In other words, the trajectory data \eqref{data} could originate from either the gradient flow 
\begin{align}\label{gradiet-flow}
\partial_t \rho_t - \Delta_{\rho_t}U'(\rho_t) = \Delta_{\rho_t}\left(V+W*\rho_t\right)   ,
\end{align}
or the Wasserstein Hamiltonian flow
\begin{align}\label{Hamiltonian-flow}
\partial_{tt} \rho_t + \Gamma_W(\partial_t \rho_t,\partial_t \rho_t) - \Delta_{\rho_t}U'(\rho_t) = \Delta_{\rho_t}\left(V+W*\rho_t\right) .
\end{align}
The internal energy  $U$ is typically either the Fisher information  $U(\rho) = \rho |\nabla \log \rho|^2$  or the Keller-Segel type energy  $U(\rho) = \frac{1}{m-1}\rho^m$  for  $m > 1$ , which is assumed to be known and therefore does not affect the analysis. To sum up, the learning problem aims to recover the potential $V$ and the interaction kernel $W$ from the trajectory flow data \eqref{data} in the knowledge of the internal energy $U$ and using a structure-preserving kernel regression method. 

\noindent {\bf Structure-preserving kernel ridge regression.}\quad
The idea of structure-preservation in the context of kernel regression is that we search the potential function $V$ and interaction kernel $W$ in two RKHS $\mathcal{H}_{K_1}$ and $\mathcal{H}_{K_2}$, respectively. 
This approach naturally preserves the Hamiltonian structure of the unknown Hamiltonian system. To explicitly describe the method, we will be solving the following loss regularized minimization problem:
\begin{align}
(V^{*}_\lambda, W^{*}_\lambda) &:= \mathop{\arg\min}\limits_{(\phi,\psi) \in \mathcal{H}_{K_1} \times \mathcal{H}_{K_2}} R_{\lambda}(\phi, \psi), \label{exp-pro1} \\
R_{\lambda}(\phi, \psi) &:= \|A(\phi, \psi) - A(V, W)\|_{L^2}^2 + \lambda_1 \|\phi\|_{\mathcal{H}_{K_1}}^2 + \lambda_2 \|\psi\|_{\mathcal{H}_{K_2}}^2, \label{exp-fun1}
\end{align}
where $\lambda_1, \lambda_2 > 0$ are regularization parameters, and the operator $A$ is defined by:
\begin{equation}\label{operatorA}
A(\phi, \psi)(t, x) = \Delta_{\rho_t}(\phi + \psi \ast \rho_t)(x) = \nabla \cdot \big(\rho_t(x) \nabla (\phi(x) + (\psi \ast \rho_t)(x))\big).
\end{equation}
The $L^2$ norm in \eqref{exp-fun1} is defined on the space 
$L^2\left([0,T] \times \mathbb{R}^d, \{\rho_t\}_{t \in [0,T]}\right)$,
which consists of all $L^2$-integrable elements
on $[0,T] \times \mathbb{R}^d$ with respect to the family $ \{\rho_t\}_{t \in [0,T]}$. Explicitly, this space is given by:
\begin{align*}
L^2\left([0,T] \times \mathbb{R}^d, \{\rho_t\}_{t \in [0,T]}\right) := \left\{f \ \Big| \ \|f\|^2_{L^2} := \int_0^T \int_{\mathbb{R}^d} |f(t, x)|^2 \rho_t(x) \, \mathrm{d}x \mathrm{d}t < \infty \right\}.
\end{align*}
Here, $R_\lambda$ is called the {\it regularized expected loss functional} and $(V_\lambda^*,W_\lambda^*)\in\mathcal{H}_{K_1}\times \mathcal{H}_{K_2}$ is called the {\it best-in-class function} with the minimal associated in-class regularized expected loss functional.

When the available data \eqref{data} are discrete in space and time, we approximate the regularized expected loss functional defined in~\eqref{exp-fun1} using numerical integration methods. In our numerical approximation, we shall initially apply the Euler scheme for simplicity. In practice, however, one may utilize higher-order numerical integration schemes to improve accuracy.   

Also for simplicity, we assume that the data are provided on a uniform (regular) mesh over the domain \([0,T]\times \Omega\), with \(\Omega=[a,b]\subset\mathbb{R}\). 
Specifically, we define the discrete spatial and temporal points of the mesh as:
\begin{equation}\label{mesh}
\begin{aligned}
x_n = a + n\Delta x,\quad t_l = l\Delta t,\quad \Delta x = \frac{b - a}{N}, \quad \Delta t = \frac{T}{L},\quad n=1,\dots,N,\quad l=1,\dots,L.
\end{aligned}
\end{equation}

Let $\rho_l^n := \rho(t_l, x_n)$. 
We define the standard forward finite difference operators $\delta_t^{+}\rho$ and $\delta_x^{+}\rho$ to approximate $\partial_t \rho$ and $\partial_x\rho$, respectively,
\begin{align*}
\left(\delta_t^{+} \rho\right)_l^n=\begin{cases}
-\rho^n_l / \Delta t, &\text { if } l=L \\
\frac{\rho_{l+1}^n-\rho_{l}^n}{\Delta t}, &\text { if } l<L
\end{cases}, \quad 
\left(\delta_x^{+} \rho\right)_l^n=\begin{cases}
-\rho_{l}^n / \Delta x, &\text { if } n=N \\
\frac{\rho_{l}^{n+1}-\rho_{l}^n}{\Delta x}, &\text { if } n<N
\end{cases} .
\end{align*}
Using the numerical scheme described above, we introduce the following numerical approximation:
\begin{equation}\label{neumericalA}
\begin{aligned}
A^{\delta}(\phi, \psi)(t_l, x_n) := \left(\delta_x^{+} \rho\right)_l^n \partial_x(\phi + \psi * \rho_{t_l})(t_l, x_n) + \rho_l^n \partial_{xx}(\phi + \psi * \rho_{t_l})(t_l, x_n).
\end{aligned}
\end{equation}  
For the gradient flow \eqref{gradiet-flow}, we define:  
\begin{align}\label{eq: f_gradient}
f^\delta(t_l, x_n) := \left(\delta_t^{+} \rho\right)_l^n - \left(\delta_x^{+} \rho\right)_l^n \partial_x U'(\rho_l^n) - \rho_l^n \partial_{xx} U'(\rho_l^n),    
\end{align}  
and for the Wasserstein Hamiltonian flow \eqref{Hamiltonian-flow}, we define:  
\begin{align}\label{eq: f_Hamiltonian}
f^\delta(t_l, x_n) := \left(\delta_{tt}^{+} \rho\right)_l^n + \Gamma_W\left(\left(\delta_t^{+} \rho\right)_l^n, \left(\delta_t^{+} \rho\right)_l^n\right) - \left(\delta_x^{+} \rho\right)_l^n \partial_x U'(\rho_l^n) - \rho_l^n \partial_{xx} U'(\rho_l^n),    
\end{align}  
where $\delta_{tt}^{+} \rho$ approximates the second-order temporal derivative of $\rho$ defined as:  
\begin{align*}
\left(\delta_{tt}^{+} \rho\right)_l^n =
\begin{cases} 
-\left(\delta_{t}^{+} \rho\right)_l^n / \Delta t, & \text{if } l = L, \\  
\frac{\left(\delta_{t}^{+} \rho\right)_{l+1}^n - \left(\delta_{t}^{+} \rho\right)_l^n}{\Delta t}, & \text{if } l < L.
\end{cases}
\end{align*} 
{Equation~\eqref{eq: f_Hamiltonian} gives a density-based approximation of the Hamiltonian forcing. In the Hamiltonian experiments, we instead use observations of both $\rho$ and the velocity potential $\Phi$ to construct the forcing through the first-order Hamiltonian system; see Section~\ref{Numerical experiments}.}

With these approximations, we formulate the empirical loss minimization problem as follows:
\begin{align}
(\widehat{V}_{\lambda,NL},\widehat{W}_{\lambda,NL}) &:= \mathop{\arg\min}\limits_{(\phi,\psi)\in \mathcal{H}_{K_1}\times \mathcal{H}_{K_2}} \ \widehat{R}_{\lambda,NL}(\phi,\psi), \label{emp-pro1} \\
\widehat{R}_{\lambda,NL}(\phi,\psi) &:= \frac{T|\Omega|}{NL} \sum_{n,l=1}^{N,L} \left| A^{\delta}(\phi,\psi)(t_l,x_n) - f^{\delta}(t_l,x_n) \right|^2 \rho_{t_l}(x_n) 
+ \lambda_1 \|\phi\|_{\mathcal{H}_{K_1}}^2 + \lambda_2 \|\psi\|_{\mathcal{H}_{K_2}}^2, \label{emp-fun1}
\end{align}
where $|\Omega|$ denotes the Lebesgue measure of $\Omega$. 
The functional $\widehat{R}_{\lambda,NL}$ is referred to as the {\it regularized empirical loss functional}. We refer to the minimizer $(\widehat{V}_{\lambda,NL}, \widehat{W}_{\lambda,NL})$ as the {\it structure-preserving kernel estimator} of the potential function $V$ and the interaction kernel $W$. 
Suppose all the data in \eqref{data} are positive. Otherwise, any non-positive data would not contribute to the regularized empirical loss functional \eqref{emp-fun1} and can therefore be omitted.

{
\begin{remark}[Degeneracy and data dependence]\label{degenercy probleml}\normalfont
For the observations \eqref{data}, define
\begin{equation}\label{null-space}
\mathcal H_{\mathrm{null}}
:=\left\{(\phi,\psi)\in\mathcal H_{K_1}\times\mathcal H_{K_2}:
A^\delta(\phi,\psi)(t_l,x_n)=0,\quad l=1,\ldots,L,\ n=1,\ldots,N\right\}.
\end{equation}
Pairs differing by an element of $\mathcal H_{\mathrm{null}}$ produce the same discrete observations. This ambiguity depends on the observed density trajectory. Such data dependence also arises in finite-dimensional structure-preserving learning problems, where non-identifiability may be related to geometric degeneracies, such as Casimir functions in Poisson systems \cite{RCSP2,RCSP3}. Here, the ambiguity may also involve nonconstant perturbations that compensate between the external and interaction potentials along the observed trajectory. Regularization yields a unique estimator but does not remove this physical ambiguity, which we study in detail later in Remark~\ref{Uniqueness} and Section~\ref{subsec:joint-identifiability}.
\end{remark}
}

\subsection{An operator framework for the learning problem}

In this subsection, we first explore a key property of the reproducing kernel Hilbert space (RKHS). We then show that the minimization problems \eqref{exp-pro1}-\eqref{exp-fun1} and \eqref{emp-pro1}-\eqref{emp-fun1} admit unique solutions, which we can express explicitly using an operator representation.

Let $\rho,\pi$ be two probability densities on $\mathbb{R}^d$, and let $K$ be a kernel function in $\mathbb{R}^d\times\mathbb{R}^d$. For $m,n\in \mathbb{N}$, denote by $\Delta_{(\rho,\pi)}^{(m,n)}K$ the actions of the operator $\Delta_\rho$ applied $m$ times to the first variable and the operator $\Delta_\pi$ applied $n$ times to the second variable of the kernel function $K$. Denote by $K*\rho$ the convolution of the kernel $K$ with $\rho$, defined as:
\begin{align*}
(K*\rho)(x,y):= \int_{\mathbb{R}^d} K(x-z,y)\rho(z)\mathrm{d}z.  
\end{align*}
Define $K**(\rho,\pi)$ the twice-convolved kernel function, as:
\begin{align*}
(K**(\rho,\pi))(x,y):= \int_{\mathbb{R}^d} \int_{\mathbb{R}^d} K(x-z_1,y-z_2)\rho(z_1)\pi(z_2)\mathrm{d}z_1 \mathrm{d}z_2.  
\end{align*}
If $\rho=\pi$, we simply denote $\Delta_{\rho}^{(m,n)}=\Delta_{(\rho,\rho)}^{(m,n)}$ and $K**\rho=K**(\rho,\rho)$. 

Applying the differential reproducing property \cite[Theorem 2.7]{RCSP2}, we obtain the following result.

\begin{proposition}\label{rep-pro}
Let $\rho\in C^1(\mathbb{R}^d)$ be a probability density. Let $K\in C_b^5(\mathbb{R}^d\times\mathbb{R}^d)$ be a Mercer kernel. The following two statements hold:
\begin{description}
\item [(i)] For all $x\in\mathbb{R}^d$, we have that $\Delta_{\rho}^{(1,0)}K(x,\cdot), \Delta_{\rho}^{(1,0)}(K*\rho)(x,\cdot)\in\mathcal{H}_{K}$.
\item [(ii)]  For any {$f\in\mathcal{H}_{K}$}, the reproducing property holds for the operator $\Delta_{\rho}$, that is, 
{\begin{align}\label{lap-rep-pro}
\Delta_{\rho}f(x) = \langle f, \Delta_{\rho}^{(1,0)}K(x,\cdot)\rangle_{\mathcal{H}_{K}}, \quad \text{ for all } x\in\mathbb{R}^d.   
\end{align}}
Furthermore, we have the following convolutional differential reproducing property:
{\begin{align}\label{con-rep-pro}
\Delta_{\rho}(f*\rho)(x) = \langle f, \Delta_{\rho}^{(1,0)}(K*\rho)(x,\cdot)\rangle_{\mathcal{H}_{K}},\quad \text{ for all } x\in\mathbb{R}^d. 
\end{align}}
\end{description}
\end{proposition}

\begin{proof}
By the differential reproducing property \cite[Theorem 2.7]{RCSP2},
the condition $K\in C_b^5(\mathbb{R}^d\times\mathbb{R}^d)$ ensures that all $x\in\mathbb{R}^d$, we have {$\Delta_{\rho}^{(1,0)}K(x,\cdot), \Delta_{\rho}^{(1,0)}(K*\rho)(x,\cdot)\in\mathcal{H}_{K}$.
Furthermore, for each $f\in\mathcal{H}_{K}$, it holds that
\begin{align*}
\Delta_{\rho}f(x)&= \nabla\rho(x)\cdot \nabla f(x)+\rho(x)\Delta f(x)=\langle f, \nabla\rho(x)\cdot \nabla K(x,\cdot)\rangle_{\mathcal{H}_{K}}+\langle f,\rho(x)\Delta K(x,\cdot)\rangle_{\mathcal{H}_{K}}.
\end{align*} 
Hence, equation \eqref{lap-rep-pro} is obtained. Using the definition of the convolution operator, we derive:
\begin{align*}
\Delta_{\rho}(f*\rho)(x)&= ((\Delta_{\rho}f)*\rho)(x) = \int_{\mathbb{R}^d}\left\langle f,\Delta_{\rho}^{(1,0)}K(x-y,\cdot)\right\rangle_{\mathcal{H}_{K}}\rho(y)\mathrm dy\\
&=\left\langle f,\int_{\mathbb{R}^d}\Delta_{\rho}^{(1,0)}K(x-y,\cdot)\rho(y)\mathrm dy\right\rangle_{\mathcal{H}_{K}}=\left\langle f, \Delta_{\rho}^{(1,0)}(K*\rho)(x,\cdot)\right\rangle_{\mathcal{H}_{K}}.
\end{align*}
This establishes quation \eqref{con-rep-pro}.}
\end{proof}

{
\begin{remark}
\normalfont
It is important to note that $f\in \mathcal{H}_{K}$ does not necessarily imply that the convolution $f*\rho\in \mathcal{H}_{K}$. As a result, the  equality 
$
\Delta_{\rho}(f*\rho)(x) = \left\langle f*\rho, \Delta_{\rho}^{(1,0)}K(x,\cdot)\right\rangle_{\mathcal{H}_{K}},
$ with $x\in\mathbb{R}^d$, may {\it not} hold. 
However, the kernel section of the convolution $(K*\rho)(x,\cdot)$ still lies in the RKHS $\mathcal{H}_{K}$ and so does $\Delta_{\rho}^{(1,0)}(K*\rho)(x,\cdot)$ due to the differential reproducing property \cite[Theorem 2.7]{RCSP2}. This is why the convolutional differential reproducing property \eqref{con-rep-pro} remains valid.
\end{remark}
}

{
\begin{remark}
\normalfont
Let $\pi\in C^1(\mathbb{R}^d)$ be a probability density. By Proposition \ref{rep-pro}, the Mercel kernel $K\in C_b^5(\mathbb{R}^d\times\mathbb{R}^d)$ ensures that for all for all $y\in\mathbb{R}^d$, we have $\Delta_{\pi}^{(1,0)}K(y,\cdot), \Delta_{\pi}^{(1,0)}(K*\pi)(y,\cdot)\in\mathcal{H}_{K}$. Then applying equation \eqref{lap-rep-pro} with $f=\Delta_{\pi}^{(1,0)}K(y,\cdot)$ and equation \eqref{con-rep-pro}  with $ f=\Delta_{\pi}^{(1,0)}(K*\pi)(y,\cdot)$, we obtain:
\begin{align*}
\Delta_{(\rho,\pi)}^{(1,1)}K(x,y)= \left\langle \Delta_{\rho}^{(1,0)}K(x,\cdot), \Delta_{\pi}^{(1,0)}K(y,\cdot)\right\rangle_{\mathcal{H}_{K}},  
\end{align*}
and
\begin{align*}
\Delta_{(\rho,\pi)}^{(1,1)}(K**(\rho,\pi))(x,y)= \left\langle \Delta_{\rho}^{(1,0)}(K*\rho)(x,\cdot), \Delta_{\pi}^{(1,0)}(K*\pi)(y,\cdot)\right\rangle_{\mathcal{H}_{K}}.  
\end{align*}
\end{remark}
}

By Proposition \ref{rep-pro}, we conclude that the operator $A$ in \eqref{operatorA} is well-defined, provided that the kernels used to define the RKHSs are such that  $K_1, K_2 \in C_b^5(\mathbb{R}^d \times \mathbb{R}^d)$. The following property establishes that $A$ is a bounded linear operator from $\mathcal{H}_{K_1} \times \mathcal{H}_{K_2}$ to $L^2\left([0,T] \times \mathbb{R}^d, \{\rho_t\}_{t \in [0,T]}\right)$. The detailed proof of this result is provided in Appendix \ref{Proof of Proposition Bound-A}.

\begin{proposition}\label{Bound-A}
Let the $K_1$ and $K_2$ be two Mercer kernels satisfying $K_1,K_2\in C_b^5(\mathbb{R}^d\times\mathbb{R}^d)$. Suppose that the family of probability measures $\{\rho_t\}_{t\in[0,T]}$ satisfying $C_T=\sup_{t\in[0,T]}\|\rho_t\|_{C_b^1}<\infty$.
Then, the operator $A$ defined in \eqref{operatorA} is bounded linear from $\mathcal{H}_{K_1}\times\mathcal{H}_{K_2}$ to $L^2\left([0,T]\times\mathbb{R}^d,\{\rho_t\}_{t\in[0,T]}\right)$ with operator norm $\|A\|\leq \sqrt{2T C_T(\kappa_1^2+\kappa_2^2)} $. Moreover, the adjoint operator $A^*$ of $A$ maps from $L^2\left([0,T]\times\mathbb{R}^d,\{\rho_t\}_{t\in[0,T]}\right)$ to $\mathcal{H}_{K_1}\times\mathcal{H}_{K_2}$ and is given by
\begin{align*}
    A^{\ast}g = \int_0^T\int_{\mathbb{R}^d} g(t,x)\begin{bmatrix}
    \Delta_{\rho_t}^{(1,0)}K_1(x,\cdot)\\
    \Delta_{\rho_t}^{(1,0)}(K_2\ast \rho_t)(x,\cdot)
    \end{bmatrix}^\top\rho_t(x)\mathrm dx \mathrm dt.
\end{align*}
Moreover, the operator $B:= A^{\ast}A: \mathcal{H}_{K_1}\times \mathcal{H}_{K_2}\rightarrow \mathcal{H}_{K_1}\times \mathcal{H}_{K_2}$ given by
\begin{align}\label{operatorB}
B(\phi,\psi)=\int_0^T\int_{\mathbb{R}^d} \Delta_{\rho_t}(\phi+\psi\ast\rho_t)(x)\begin{bmatrix}
    \Delta_{\rho_t}^{(1,0)} K_1(x,\cdot)\\
    \Delta_{\rho_t}^{(1,0)}(K_2\ast \rho_t)(x,\cdot)
    \end{bmatrix}^\top\rho_t(x)\mathrm dx \mathrm dt,    
\end{align}
is trace class with trace $\operatorname{Tr}(B)\leq 4T C_T(\kappa_1^2+\kappa_2^2)$. 
\end{proposition}

We now handle the empirical case. Let us denote by \(A_{NL}^{\delta}: \mathcal{H}_{K_1}\times \mathcal{H}_{K_2}\to\mathbb{R}^{NL}\) the empirical version of the operator \(A\), defined in~\eqref{operatorA} and constructed using the numerical schemes described above. Specifically, we define:
\begin{equation}\label{A_NL_delta}
\begin{aligned}
A_{NL}^{\delta}(\phi,\psi) &:= \sqrt{\frac{T|\Omega|}{NL}}\left(A^{\delta}(\phi,\psi)(t_1,x_1),\,A^{\delta}(\phi,\psi)(t_1,x_2),\,\dots,\,A^{\delta}(\phi,\psi)(t_L,x_N)\right)^\top \\
&= \sqrt{\frac{T|\Omega|}{NL}}\Delta_{\rho_L}^{\delta}\left(\phi + \psi*\rho_L\right)(X_N),
\end{aligned}
\end{equation}
where the operator \(A^{\delta}\) is given explicitly in~\eqref{neumericalA}. Now, for arbitrary vectors \(\mathbf{u},\mathbf{v}\in\mathbb{R}^{NL}\), we introduce the weighted inner product:
\begin{align}\label{C_rho}
\langle\mathbf{u},\mathbf{v}\rangle_{C_{\rho}} := \mathbf{u}^\top C_{\rho}\,\mathbf{v},
\end{align}
where the weighting matrix \(C_{\rho}\in\mathbb{R}^{NL\times NL}\) is defined as the diagonal matrix:
\[
C_{\rho} := \operatorname{diag}\{\rho_{t_1}(x_1),\,\rho_{t_1}(x_2),\,\dots,\,\rho_{t_L}(x_N)\}.
\]
The corresponding norm induced by this inner product on \(\mathbb{R}^{NL}\) is denoted by \(\|\cdot\|_{C_{\rho}}\).  
Using a similar argument as in the proof of Proposition \ref{Bound-A}, we can show that the operator $A_{NL}^\delta: \mathcal{H}_{K_1} \times \mathcal{H}_{K_2} \to \mathbb{R}^{NL}$ is bounded and linear.

\begin{corollary}\label{Bound-ANL}
Let $K_1$ and $K_2$ be two Mercer kernels satisfying $K_1,K_2\in C^4(\Omega\times\Omega)$. Suppose that the family of probability measures $\{\rho_t\}_{t\in[0,T]}$ satisfying $C_T=\sup_{t\in[0,T]}\|\rho_t\|_{C_b^2}<\infty$.
Then, the operator $A_{NL}^\delta$ defined in \eqref{A_NL_delta} is bounded linear from $\mathcal{H}_{K_1}\times\mathcal{H}_{K_2}$ to $\mathbb{R}^{NL}$ with operator norm $\|A_{NL}^\delta\|\leq \sqrt{2T C_T(\kappa_1^2+\kappa_2^2)} $. Moreover, the adjoint operator $A^{\delta*}_{NL}$ of $A_{NL}^\delta$, mapping from $\mathbb{R}^{NL}$ to $\mathcal{H}_{K_1}\times\mathcal{H}_{K_2}$, is defined as
\begin{equation}\label{operatorANLstar}
\begin{aligned}
A_{NL}^{\delta*} \mathbf{u}=\sqrt{\frac{T|\Omega|}{NL}}\mathbf{u}^\top C_\rho\left(\Delta^{(1,0)\delta}_{\rho_{L}}K_1(X_N,\cdot),\Delta^{(1,0)\delta}_{\rho_{L}}(K_2\ast \rho_{L})(X_N,\cdot)\right),
\end{aligned}
\end{equation}
{where $\Delta^{(1,0)\delta}_{\rho_{L}}K$ denotes the operation $\Delta^{\delta}_{\rho_{L}}$ defined in \eqref{A_NL_delta} acting on the first variable of the kernel function $K$.}
The operator $B^{\delta}_{NL}:= A_{NL}^{\delta\ast}A_{NL}^\delta: \mathcal{H}_{K_1}\times \mathcal{H}_{K_2}\rightarrow \mathcal{H}_{K_1}\times \mathcal{H}_{K_2}$ given by
\begin{equation}
\begin{aligned}\label{operatorBNL}
B^\delta_{NL}(\phi,\psi) =\frac{T|\Omega|}{NL}{\Delta^\delta_{\rho_{L}}}^\top(\phi+\psi*\rho_{L})(X_N) C_{\rho} \left(
\Delta^{(1,0)\delta}_{\rho_{L}}K_1(X_N,\cdot),
\Delta^{(1,0)\delta}_{\rho_{L}}(K_2\ast \rho_{L})(X_N,\cdot)\right).
\end{aligned}
\end{equation}
is trace class with trace $\operatorname{Tr}(B_{NL})\leq 4T C_T(\kappa_1^2+\kappa_2^2)$. 
\end{corollary}

Having defined the operator \( A \) and $A^\delta_{NL}$ and established their properties in Proposition \ref{Bound-A} and Corollary \ref{Bound-ANL}, we are now prepared to derive an explicit operator representation for the minimizer solving the inverse learning problems \eqref{exp-pro1}--\eqref{exp-fun1} and \eqref{emp-pro1}-\eqref{emp-fun1}. The proof of the following proposition is provided in Appendix \ref{proof prop with the solution}.

\begin{proposition}[{\bf Operator representation of the learning problem solutions}]\label{Operator_representation}  For all \(\lambda_1, \lambda_2 > 0\), the minimization problems \eqref{exp-pro1}--\eqref{exp-fun1} and \eqref{emp-pro1}-\eqref{emp-fun1} admit unique solutions and they are given by
\begin{equation}
\begin{aligned}\label{ope-rep-emp-esti}
(V_{\lambda}^*, W_\lambda^*) &= \left(B + \lambda I \right)^{-1} B(V, W),
\end{aligned}
\end{equation}
and 
\begin{equation}\label{kernel-estimator}
\left(\widehat{V}_{\lambda,NL},\, \widehat{W}_{\lambda,NL}\right) 
= \sqrt{\frac{T|\Omega|}{NL}}\left(B_{NL}^{\delta} + \lambda I\right)^{-1} A_{NL}^{\delta*}\,f_{NL}^{\delta},
\end{equation}
where the matrix is defined as $\lambda = \operatorname{diag}\{\lambda_1,\lambda_2\}$ and \( f_{NL}^{\delta}\in\mathbb{R}^{NL} \) is the vectorized data given explicitly by
\begin{equation*}
f_{NL}^{\delta} := \left(f^{\delta}(t_1,x_1),\,f^{\delta}(t_1,x_2),\,\dots,\,f^{\delta}(t_L,x_N)\right)^\top.
\end{equation*}  
\end{proposition}

\subsection{A kernel representation of the learning problem}

The operator representation \eqref{kernel-estimator} of the solution to the minimization problem does not carry in its wake explicit computational instructions that can be used to construct algorithms of use in practical applications. 
Therefore, we now further derive a representation of the kernel estimator $\left(\widehat{V}_{\lambda,NL}, \widehat{W}_{\lambda,NL}\right)$ of the minimization problem \eqref{emp-pro1}-\eqref{emp-fun1} that facilitates computation and is the analog of the classical Representer Theorem \cite{Mohri:learning:2012} adapted to our setup.  We first define the {\it Gram matrix} as follows:
\begin{align}\label{gram-matrix}
G_{NL}:= 
C_\rho\left(\lambda_2\Delta_{\rho_{L}}^{(1,1)\delta}K_1(X_N,X_N)+\lambda_1\Delta_{\rho_{L}}^{(1,1)\delta}(K_2\ast \ast\rho_{L})(X_N,X_N)\right)C_\rho.
\end{align}
By \cite[Proposition 3.4]{RCSP2}, it can be shown that the Gram matrix $G_{NL}$ is symmetric and positive semi-definite. Furthermore, we have the following property.
\begin{proposition}[{\bf A Representer Theorem for the solution}]\label{Kernel representation}
Suppose that $K_1$ and $K_2$ are Mercer kernel such that $K_1,K_2\in C^4(\Omega\times\Omega)$. Then for every $\lambda_1,\lambda_2>0$, the estimator of the regularized empirical loss functional $\widehat{R}_{\lambda,NL}$ in \eqref{emp-fun1} can be written as
\begin{align*}
\widehat{V}_{\lambda,NL}= \left\langle \widehat{\mathbf{C}}_1,\Delta_{\rho_{L}}^{(1,0)\delta} K_1(X_N,\cdot) \right \rangle_{C_\rho}\quad \text{and}\quad \widehat{W}_{\lambda,NL}=  \left\langle \widehat{\mathbf{C}}_2,\Delta_{\rho_{L}}^{(1,0)\delta} (K_2 * \rho_{L})(X_N,\cdot) \right \rangle_{C_\rho}, 
\end{align*}
where $\langle\cdot,\cdot\rangle_{C_\rho}$ is the weighted product in $\mathbb{R}^{NL}$ defined in \eqref{C_rho} and the coefficients $\widehat{\mathbf{C}}_1,\widehat{\mathbf{C}}_2\in \mathbb{R}^{NL}$ are given by
\begin{equation}\label{coeffient}
\begin{cases}
\widehat{\mathbf C}_1 &= \lambda_2\left(G_{NL}+\lambda_1\lambda_2\frac{NL}{T|\Omega|} C_\rho\right)^{-1}
C_\rho f_{NL}^\delta, \\
\widehat{\mathbf C}_2&=\lambda_1\left(G_{NL}+\lambda_1\lambda_2\frac{NL}{T|\Omega|} C_\rho\right)^{-1}
C_\rho f_{NL}^\delta. 
\end{cases}
\end{equation}
\end{proposition}

\begin{proof}
We first define a subspace $\mathcal{H}_{K_1}^{NL} \times \mathcal{H}_{K_2}^{NL}$ of the RKHS $\mathcal{H}_{K_1} \times \mathcal{H}_{K_2}$ as follows:
\begin{align*}
\mathcal{H}_{K_1}^{NL} \times \mathcal{H}_{K_2}^{NL} := &\operatorname{span}\{\Delta_{\rho_{t_l}}^{(1,0)\delta}K_1(x_n,\cdot) : l=1,\cdots,L, \, n=1,\cdots,N\} \\
\times &\operatorname{span}\{\Delta_{\rho_{t_l}}^{(1,0)\delta}(K_2*\rho_{t_l})(x_n,\cdot) : l=1,\cdots,L, \, n=1,\cdots,N\}.
\end{align*}
From the definition of the operator $B_{NL}^\delta$ in \eqref{operatorBNL}, it follows that $\mathcal{H}_{K_1}^{NL} \times \mathcal{H}_{K_2}^{NL}$ is an invariant subspace of $B_{NL}^\delta$. Consequently, there exist coefficients $\widehat{\mathbf{C}}_1, \widehat{\mathbf{C}}_2 \in \mathbb{R}^{NL}$ such that the estimator in \eqref{ope-rep-emp-esti} can be expressed as:
\begin{align}\label{estimator}
\widehat{V}_{\lambda,NL}= \left\langle \widehat{\mathbf{C}}_1,\Delta_{\rho_{L}}^{(1,0)\delta} K_1(X_N,\cdot) \right \rangle_{C_\rho}\quad \text{and}\quad \widehat{W}_{\lambda,NL}=  \left\langle \widehat{\mathbf{C}}_2,\Delta_{\rho_{L}}^{(1,0)\delta} (K_2 * \rho_{L})(X_N,\cdot) \right \rangle_{C_\rho}, 
\end{align}
Applying $B_{NL}^\delta+\lambda I$ on both sides of the above equation and combing the definition of $A_{NL}^{\delta*}$ in \eqref{operatorANLstar}, we obtain that 
\begin{align*}
&\frac{T|\Omega|}{NL}\left(\Delta^\delta_{\rho_{L}}(\widehat{V}_{\lambda,NL}+\widehat{W}_{\lambda,NL}*\rho_{L})(X_N)\right)^\top C_\rho \left(
\Delta_{\rho_{L}}^{(1,0)\delta}K_1(X_N,\cdot),
\Delta_{\rho_{L}}^{(1,0)\delta}(K_2\ast \rho_{L})(X_N,\cdot)\right)\\
=\ &  \frac{T|\Omega|}{NL}{f_{NL}^\delta}^\top C_\rho\left(
\Delta_{\rho_{L}}^{(1,0)\delta}K_1(X_N,\cdot),
\Delta_{\rho_{L}}^{(1,0)\delta}(K_2\ast \rho_{L})(X_N,\cdot)\right)-\left(\lambda_1\widehat{V}_{\lambda,NL},\lambda_2\widehat{W}_{\lambda,NL}\right).
\end{align*}
Hence by the expression of the estimator $\left(\widehat{V}_{\lambda,NL},\widehat{W}_{\lambda,NL}\right)$ in \eqref{estimator}, we have that
\begin{equation*}
\begin{cases}
\lambda_1\frac{NL}{T|\Omega|}{\widehat{\mathbf C}}_1^\top C_\rho+\left( {\widehat{\mathbf C}}_1^\top C_\rho\Delta_{\rho_{L}}^{(1,1)\delta}K_1(X_N,X_N)+ {\widehat{\mathbf C}}_2^\top C_\rho\Delta_{\rho_{L}}^{(1,1)\delta}(K_2\ast\ast \rho_{L})(X_N,X_N)\right)C_\rho={f_{NL}^\delta}^\top C_\rho,\\
\lambda_2\frac{NL}{T|\Omega|} {\widehat{\mathbf C}}_2^\top C_\rho+\left( {\widehat{\mathbf C}}_1^\top C_\rho\Delta_{\rho_{L}}^{(1,1)\delta}K_1(X_N,X_N)+ {\widehat{\mathbf C}}_2^\top C_\rho\Delta_{\rho_{L}}^{(1,1)\delta}(K_2\ast\ast \rho_{L})(X_N,X_N)\right)C_\rho={f_{NL}^\delta}^\top C_\rho.
   \end{cases}
\end{equation*}
Since the Gram matrix $G_{NL}$ defined in \eqref{gram-matrix} is positive semi-definite and the matrix $C_\rho$ is positive by assumption, the coefficients $\widehat{\mathbf{C}}_1$ and $\widehat{\mathbf{C}}_2$ are unique and given by \eqref{coeffient}. The results follow.   
\end{proof}

\begin{remark}[Data-dependent representation]\normalfont
Although \eqref{emp-pro1} minimizes over the full RKHS product space, Proposition~\ref{Kernel representation} shows that its solution admits a finite representation determined by the observation functionals. Thus, no prescribed finite dictionary is required to derive the estimator.
The kernel representation of the estimator in Proposition \ref{Kernel representation} demonstrates that the predefined set of basis functions in \cite{carrillo2025sparse,gao2024self} can be interpreted as a basis of the subspace $\mathcal{H}_{K_1}^{NL} \times \mathcal{H}_{K_2}^{NL}$.
\end{remark}

{
\begin{remark}[Uniqueness and identifiability]\label{Uniqueness}\normalfont
Since $A_{NL}^\delta$ is bounded and linear, $\mathcal H_{\mathrm{null}}=\ker A_{NL}^\delta$ is closed, and
\begin{equation*}
\mathcal H_{K_1}\times\mathcal H_{K_2}
=\mathcal H_{\mathrm{null}}\oplus\mathcal H_{\mathrm{null}}^\perp.
\end{equation*}
For $\lambda_1,\lambda_2>0$, the regularization term makes the empirical loss strictly convex, yielding the unique estimator in Proposition~\ref{Operator_representation}. Its first-order optimality condition in any direction $(\phi,\psi)\in\mathcal H_{\mathrm{null}}$ gives
\[
\lambda_1\langle\widehat V_{\lambda,NL},\phi\rangle_{\mathcal H_{K_1}}
+\lambda_2\langle\widehat W_{\lambda,NL},\psi\rangle_{\mathcal H_{K_2}}=0.
\]
Thus $(\lambda_1\widehat V_{\lambda,NL},\lambda_2\widehat W_{\lambda,NL})\in\mathcal H_{\mathrm{null}}^\perp$. The estimator minimizes the weighted RKHS penalty among pairs with the same fitted observations, but need not recover unobservable components of the true potentials. We now discuss concrete identifiability conditions in Section~\ref{subsec:joint-identifiability}.
\end{remark}}

\subsection{Identifiability of the potential and interaction energies}\label{subsec:joint-identifiability}

As discussed in Remark~\ref{degenercy probleml}, the identifiability of the underlying energy functional depends on the observed density trajectory and the admissible function spaces. A particular difficulty arises in the joint recovery of the external potential $V$ and the interaction potential $W$, since they enter the Wasserstein gradient flow through the effective potential $V+W*\rho_t$. Consequently, variations in these two potentials may compensate for one another along the observed trajectory, even when the corresponding energy components are different. We investigate this ambiguity and its dependence on the temporal variation of the density. For an explicit characterization, we consider the connected flat torus $\mathbb T^d$, with periodic convolution, in this subsection. Throughout this subsection, we use the normalized Haar measure
$d\mu(x)=(2\pi)^{-d}\,dx$ on $\mathbb T^d$, with
\[
(f*g)(x)=\int_{\mathbb T^d}f(x-y)g(y)\,d\mu(y),
\qquad
\widehat f(k)=\int_{\mathbb T^d}f(x)e^{-ik\cdot x}\,d\mu(x).
\]

Recall the continuous observation operator $A$ defined in \eqref{operatorA}. For a given density trajectory $\rho=(\rho_t)_{t\in[0,T]}$, define
\begin{equation}\label{eq:joint-ident-nullspace}
\mathcal N_\rho:=\ker A=\left\{(\phi,\psi)\in\mathcal H_{K_1}\times\mathcal H_{K_2}:\Delta_{\rho_t}(\phi+\psi*\rho_t)=0,\ \forall t\in[0,T]\right\}.
\end{equation}
Unlike the discrete null space $\mathcal H_{\mathrm{null}}$ in \eqref{null-space}, which depends on the observation points, $\mathcal N_\rho$ describes the ambiguity associated with continuous observation of the entire density trajectory.

To isolate the information supplied by temporal variations of the density, fix an arbitrary reference time $t_0\in[0,T]$ and introduce the operator
\begin{equation}\label{eq:joint-ident-temporal}
\mathcal D_\rho\psi(t):=\nabla\left[\psi*(\rho_t-\rho_{t_0})\right],\qquad \psi\in\mathcal H_{K_2}.
\end{equation}
Although $\mathcal D_\rho$ depends on $t_0$, its null space is independent of this choice: $\psi\in\ker\mathcal D_\rho$ if and only if $t\mapsto\nabla(\psi*\rho_t)$ is constant. Thus, $\ker\mathcal D_\rho$ consists of interaction perturbations whose contributions to the effective force remain unchanged along the observed trajectory. Such perturbations can be indistinguishable from changes in the external potential, provided that the latter can compensate for them within its admissible function space.

\begin{lemma}[Joint null-space characterization]\label{lem:joint-ident-nullspace}
Let $\rho=(\rho_t)_{t\in[0,T]}$ be a smooth, strictly positive density trajectory on $\mathbb T^d$. Then
\begin{equation*}
\mathcal N_\rho=\left\{(\phi,\psi)\in\mathcal H_{K_1}\times\mathcal H_{K_2}:\psi\in\ker\mathcal D_\rho,\ \nabla\phi=-\nabla(\psi*\rho_{t_0})\right\}.
\end{equation*}
\end{lemma}

\begin{proof}
Let $(\phi,\psi)\in\mathcal N_\rho$ and set $f_t=\phi+\psi*\rho_t$. Since $\Delta_{\rho_t}f_t=0$, integration by parts on $\mathbb T^d$ gives
\[
0=\int_{\mathbb T^d}f_t\Delta_{\rho_t}f_t\,d\mu(x)=-\int_{\mathbb T^d}\rho_t|\nabla f_t|^2\,d\mu(x).
\]
The strict positivity of $\rho_t$ and connectedness of $\mathbb T^d$ imply $\nabla f_t=0$ for every $t$. Evaluating at $t_0$ yields $\nabla\phi=-\nabla(\psi*\rho_{t_0})$, while subtracting the corresponding identities at $t$ and $t_0$ gives $\mathcal D_\rho\psi=0$. Conversely, these two conditions imply $\nabla(\phi+\psi*\rho_t)=0$ for every $t$, hence $(\phi,\psi)\in\mathcal N_\rho$. 
\end{proof}

Lemma~\ref{lem:joint-ident-nullspace} reveals two requirements for a nontrivial interaction perturbation to be jointly unidentifiable: it must satisfy $\psi\in\ker\mathcal D_\rho$, and its induced force must be representable by an admissible external potential, equivalently $\psi*\rho_{t_0}\in\mathcal H_{K_1}+\mathbb R$. The latter condition can be examined explicitly for translation-invariant kernels on $\mathbb T^d$. Suppose that $K_j(x,y)=\kappa_j(x-y)$, where $\kappa_j(z)=\sum_{k\in\mathbb Z^d}a_j(k)e^{ik\cdot z}$ with $a_j(k)>0$, $j=1,2$. Using normalized periodic convolution and the Fourier characterization of the corresponding RKHSs, we obtain
\begin{equation*}
\psi*\rho_t\in\mathcal H_{K_1}\quad\Longleftrightarrow\quad
\sum_{k\in\mathbb Z^d}\frac{|\widehat\psi(k)|^2|\widehat\rho_t(k)|^2}{a_1(k)}<\infty.
\end{equation*}
Moreover, convolution by $\rho_t$ defines a bounded operator from $\mathcal H_{K_2}$ into $\mathcal H_{K_1}$ if and only if
\begin{equation}\label{eq:joint-ident-compatibility-bound}
\sup_{k\in\mathbb Z^d}\frac{a_2(k)|\widehat\rho_t(k)|^2}{a_1(k)}<\infty.
\end{equation}

\begin{example}[Compatibility of periodic Gaussian kernels]\label{exm:Compatibility}
Consider periodic Gaussian kernels with Fourier coefficients $a_j(k)=c_j\exp(-\ell_j^2|k|^2/2)$, where $c_j,\ell_j>0$. Condition~\eqref{eq:joint-ident-compatibility-bound} becomes
\[
\sup_{k\in\mathbb Z^d}|\widehat\rho_t(k)|^2\exp\left(\frac{\ell_1^2-\ell_2^2}{2}|k|^2\right)<\infty.
\]
In particular, if $\ell_1\le\ell_2$, convolution by any probability density $\rho_t$ defines a bounded operator from $\mathcal H_{K_2}$ into $\mathcal H_{K_1}$. 
\end{example}

The preceding discussion characterizes the compatibility between the interaction and external-potential RKHSs under density convolution. However, spatial compatibility alone does not imply joint non-identifiability, since the induced interaction force must also remain unchanged along the observed trajectory. We now characterize this temporal condition in Fourier variables and obtain a criterion for joint identifiability up to additive constants.

\begin{proposition}[Fourier criterion for joint identifiability]\label{prop:joint-ident-fourier}
Under the assumptions of Lemma~\ref{lem:joint-ident-nullspace}, suppose that $K_1$ and $K_2$ are translation-invariant kernels on $\mathbb T^d$ with strictly positive Fourier coefficients. Define
\[
\mathcal Z_\rho:=\left\{k\in\mathbb Z^d\setminus\{0\}:\widehat\rho_t(k)=\widehat\rho_{t_0}(k),\ \forall t\in[0,T]\right\}.
\]
Then $\mathcal N_\rho$ defined in \eqref{eq:joint-ident-nullspace} is identifiable, that is, $\mathcal N_\rho=\operatorname{span}\{1\}\times\operatorname{span}\{1\}$, if and only if $\mathcal Z_\rho=\varnothing$.
\end{proposition}

\begin{proof}
Taking Fourier coefficients in \eqref{eq:joint-ident-temporal} yields, for every $k\neq0$,
\[
\widehat{\mathcal D_\rho\psi}(t,k)=ik\,\widehat\psi(k)\big(\widehat\rho_t(k)-\widehat\rho_{t_0}(k)\big).
\]
If $\mathcal Z_\rho=\varnothing$, then $\mathcal D_\rho\psi=0$ implies that $\widehat\psi(k)=0$ for all $k\neq0$, so $\psi$ is constant. Lemma~\ref{lem:joint-ident-nullspace} then implies that $\phi$ is also constant, establishing sufficiency.

Conversely, suppose that $k\in\mathcal Z_\rho$. Since $\rho_t$ is real-valued, $-k\in\mathcal Z_\rho$ as well. Set $\psi(x)=\cos(k\cdot x)$ and $\phi=-\psi*\rho_{t_0}$. Both functions have finite Fourier support and hence belong to their respective RKHSs by the strict positivity of the kernel Fourier coefficients. Moreover, $\psi*(\rho_t-\rho_{t_0})=0$ for every $t\in[0,T]$. Thus $(\phi,\psi)\in\mathcal N_\rho$ by Lemma~\ref{lem:joint-ident-nullspace}. Since $\psi$ is nonconstant, $\mathcal N_\rho$ strictly contains $\operatorname{span}\{1\}\times\operatorname{span}\{1\}$, completing the proof.
\end{proof}

Proposition~\ref{prop:joint-ident-fourier} shows that a stationary density trajectory cannot uniquely determine the external and interaction potentials, even up to additive constants. However, nonstationarity alone is not sufficient for joint identifiability: a density trajectory may evolve while leaving certain Fourier modes unexcited. In fact, joint identifiability holds if and only if every nonzero Fourier mode varies over the observation interval.

\begin{example}[A nonstationary but non-identifiable trajectory]\label{ex:joint-ident-nonstationary}
Consider the Wasserstein gradient flow on $\mathbb T=\mathbb R/2\pi\mathbb Z$,
\begin{equation*}
\partial_t\rho=\partial_x\left[\rho\,\partial_x\left(\kappa\rho+V+W*\rho\right)\right],\qquad \kappa>0,
\end{equation*}
with $V(x)=\cos(2x)$, $W(x)=\cos(2x)+\cos(x)$, and $\rho_0(x)=1+\varepsilon\cos(2x)$, where $0<\varepsilon<1$. Let $\{\rho_t\}_{t\in[0,T]}$ be the corresponding smooth solution. The trajectory is nonstationary, since
\[
\left.\partial_t\rho(0)\right|_{t=0}=-4(1+\varepsilon)\left(1+\frac{\varepsilon}{2}+\kappa\varepsilon\right)\neq0.
\]
Since $\rho_0$ and $V$ are $\pi$-periodic and periodic convolution is translation-equivariant, uniqueness of the smooth solution implies $\rho_t(x+\pi)=\rho_t(x)$ for all $t\in[0,T]$. Consequently, $\widehat\rho_t(2k+1)=0$ for every $k\in\mathbb Z$ and all
$t\in[0,T]$, so
\[
\{2k+1:k\in\mathbb Z\}\subseteq\mathcal Z_\rho.
\]
Thus, the Fourier criterion in
Proposition~\ref{prop:joint-ident-fourier} fails.
In particular,
\[
(\cos x)*\rho_t=0,\qquad \text{for all }t\in[0,T],
\]
and hence $(0,\cos x)\in\mathcal N_\rho$. Thus $W_1(x)=\cos(2x)$ and $W_2(x)=\cos(2x)+\cos x$ generate the same density trajectory, showing that the interaction potential is not identifiable even up to additive constants.
\end{example}

The ambiguity in Example~\ref{ex:joint-ident-nonstationary} is intrinsic to the observed trajectory and cannot be removed by refining its temporal or spatial discretization. Identifiability may instead be improved by restricting the admissible potentials or combining observations from different initial conditions. This distinction is also relevant to the convergence analysis: if $B=A^*A$, then a source condition $h\in\operatorname{Ran}(B^\gamma)$ with $\gamma>0$ implies $h\perp\mathcal N_\rho$. Thus, such a condition excludes unobservable components of the ground truth rather than resolving their ambiguity through regularization.

\section{Error bounds for the estimator}
\label{error_analysis}
We now analyze the ability of the structure-preserving kernel estimator $\left(\widehat{V}_{\lambda,NL}, \widehat{W}_{\lambda,NL}\right)$, to recover the unknown potential and
interaction functions $(V,W)$. 
To achieve this goal, we decompose the  {\it reconstruction error} $\left(\widehat{V}_{\lambda,NL},\widehat{W}_{\lambda,NL}\right)-(V,W)$ as the sum of what we shall be calling the {\it estimation} and {\it approximation errors}:

\begin{align} \label{decom-error} 
\left(\widehat{V}_{\lambda,NL},\widehat{W}_{\lambda,NL}\right)-(V,W)=\underbrace{\left(\widehat{V}_{\lambda,NL},\widehat{W}_{\lambda,NL}\right)-\left(V_{\lambda}^*,W_{\lambda}^*\right)}_{\text{Estimation error}}\quad+\quad\underbrace{\left(V_{\lambda}^*,W_{\lambda}^*\right)-(V,W)}_{\text{Approximation error}}.
\end{align}
The estimation error in \eqref{decom-error} comes from two sources: the discretization approximation of the integral and the numerical approximation (Euler scheme) of derivatives of the density flow. Therefore, we further decompose the estimation error into two components, which we shall explicitly refer to as the {\it scheme-dependent  error} and the {\it mesh-induced error}: 
\begin{align}\label{decom-estimation} 
\left(\widehat{V}_{\lambda, NL}, \widehat{W}_{\lambda, NL}\right) - \left(V_{\lambda}^*, W_{\lambda}^*\right) 
= \underbrace{\left(\widehat{V}_{\lambda, NL}, \widehat{W}_{\lambda, NL}\right) - \left(V_{\lambda, NL}, W_{\lambda, NL}\right)}_{\text{Scheme-dependent error}}  +  \underbrace{\left(V_{\lambda, NL}, W_{\lambda, NL}\right) - \left(V_{\lambda}^*, W_{\lambda}^*\right)}_{\text{Mesh-induced error}}.
\end{align}
Here, \((V_{\lambda,NL}, W_{\lambda,NL})\) is defined as the minimizer of the following discretized minimization problem:
\begin{align}
({V}_{\lambda,NL},{W}_{\lambda,NL}) &:= \mathop{\arg\min}\limits_{(\phi,\psi)\in \mathcal{H}_{K_1}\times \mathcal{H}_{K_2}} \ {R}_{\lambda,NL}(\phi,\psi),\label{discretized-pro}  \\
R_{\lambda,NL}(\phi,\psi) &:= \frac{T|\Omega|}{NL} \sum_{n,l=1}^{N,L} \left| A(\phi,\psi)(t_l,x_n) - f(t_l,x_n) \right|^2 \rho_{t_l}(x_n) 
+ \lambda_1 \|\phi\|_{\mathcal{H}_{K_1}}^2 + \lambda_2 \|\psi\|_{\mathcal{H}_{K_2}}^2, \label{discretized-fun}
\end{align}
where the function \(f\) is given explicitly as follows:
\begin{itemize}
\item For the gradient flow~\eqref{gradiet-flow}, we have
\[
f(t_l,x_n) := \partial_{t}\rho_{t_l}(x_n) - \Delta_{\rho_{t_l}}U'(\rho_{t_l}(x_n)).
\]
\item For the Wasserstein Hamiltonian flow~\eqref{Hamiltonian-flow}, we have
\[
f(t_l,x_n) := \partial_{tt}\rho_{t_l}(x_n) + \Gamma_W(\partial_{t}\rho_{t_l},\partial_{t}\rho_{t_l})(x_n) - \Delta_{\rho_{t_l}}U'(\rho_{t_l}(x_n)).
\]
\end{itemize}

\begin{remark}[{\bf Operator representation of  \((V_{\lambda,NL}, W_{\lambda,NL})\)}]\normalfont
We denote by \( A_{NL}: \mathcal{H}_{K_1}\times \mathcal{H}_{K_2}\to\mathbb{R}^{NL} \) the discretized version of the operator \( A \), defined in~\eqref{operatorA}, as follows:
\begin{equation*}
\begin{aligned}
A_{NL}(\phi, \psi) &:= \sqrt{\frac{T|\Omega|}{NL}}\left(A(\phi,\psi)(t_1,x_1), A(\phi,\psi)(t_1,x_2), \dots, A(\phi,\psi)(t_L,x_N)\right)^\top \\
&= \sqrt{\frac{T|\Omega|}{NL}}\,\Delta_{\rho_L}(\phi + \psi * \rho_L)(X_N).
\end{aligned}
\end{equation*}
Then, analogously to Corollary~\ref{Bound-ANL}, the operator \(A_{NL}\) is bounded and linear. Its dual operator \(A_{NL}^{*}:\mathbb{R}^{NL}\to\mathcal{H}_{K_1}\times\mathcal{H}_{K_2}\) is explicitly given by:
\begin{equation*}
A_{NL}^{*}\mathbf{u} = \sqrt{\frac{T|\Omega|}{NL}}\,\mathbf{u}^\top C_{\rho}\left(\Delta_{\rho_L}^{(1,0)}K_1(X_N,\cdot),\,\Delta_{\rho_L}^{(1,0)}(K_2*\rho_L)(X_N,\cdot)\right).
\end{equation*}
Furthermore, the operator \( B_{NL}:= A_{NL}^{*}A_{NL}: \mathcal{H}_{K_1}\times\mathcal{H}_{K_2}\to\mathcal{H}_{K_1}\times\mathcal{H}_{K_2}\) defined by
\begin{equation*}
B_{NL}(\phi,\psi) = \frac{T|\Omega|}{NL}\Delta_{\rho_L}^{\top}(\phi+\psi*\rho_L)(X_N)\,C_{\rho}\left(\Delta_{\rho_L}^{(1,0)}K_1(X_N,\cdot),\,\Delta_{\rho_L}^{(1,0)}(K_2*\rho_L)(X_N,\cdot)\right)
\end{equation*}
is trace class, with trace satisfying $
\operatorname{Tr}(B_{NL}) \leq 4T C_T(\kappa_1^2+\kappa_2^2)$.
Consequently, the discretized minimization problem~\eqref{discretized-pro}–\eqref{discretized-fun} admits a unique minimizer, which has the following explicit operator representation:
\begin{equation}\label{kernel-estimator1}
\left(V_{\lambda,NL},\,W_{\lambda,NL}\right) = \sqrt{\frac{T|\Omega|}{NL}}(B_{NL}+\lambda I)^{-1}A_{NL}^{*}f_{NL}=(B_{NL}+\lambda I)^{-1}B_{NL}(V,W),
\end{equation}
where the vectorized data \(f_{NL}\in\mathbb{R}^{NL}\) is defined explicitly by:
\begin{equation*}
f_{NL} := \left(f(t_1,x_1),\,f(t_1,x_2),\,\dots,\,f(t_L,x_N)\right)^\top.
\end{equation*}   
\end{remark}

\subsection{The mesh-induced error}
We demonstrate that the mesh-induced error, as defined in \eqref{decom-estimation}, converges with respect to the RKHS norm $\mathcal{H}_{K_1} \times \mathcal{H}_{K_2}$ as $M \to \infty$ and $L \to \infty$. For simplicity, let $\lambda := \lambda_1 = \lambda_2$ in the following discussion. For each $(\phi, \psi) \in \mathcal{H}_{K_1} \times \mathcal{H}_{K_2}$, define:  
\begin{align}\label{fun-u}
u(t,x) := \Delta_{\rho_t}(\phi + \psi * \rho_t)(x),
\end{align}  
and
\begin{align}\label{fun-w}
w(t, x; s, y) := \Delta_{\rho_t, \rho_s}^{(1,1)}\big(K_1 + K_2 ** (\rho_t, \rho_s)\big)(x, y).
\end{align} 
We now establish properties of these two functions in the following results, whose detailed proofs are contained in Appendix \ref{Proof of Lemma 3.3}.
\begin{lemma}\label{analyze-u-w}
Let $K_1,K_2\in C^6(\Omega\times\Omega)$ be two Mercer kernels. Suppose that $\rho \in C^{1,2}([0,T] \times \Omega)$. 
Then we have that $u \in C^{1,1}([0,T] \times \Omega)$ and $w\in C^{1,1,1,1}([0,T] \times \Omega\times [0,T] \times \Omega)$. Furthermore, 
\begin{equation*}
\left\{\begin{aligned}
&\|u\|_{C^{1,1}} \leq (4 + |\Omega|\|\nabla_t \rho\|_{\infty})(\kappa_1 + \kappa_2)\|\rho\|_{C^{1,2}} (\|\phi\|_{\mathcal{H}_{K_1}} + \|\psi\|_{\mathcal{H}_{K_2}}),\\
&\|w\|_{C^{1,1,1,1}}  \leq (4 + |\Omega|\|\nabla_t \rho\|_{\infty})(\kappa_1^2 + \kappa_2^2)\|\rho\|_{C^{1,2}}^2,
\end{aligned}\right.
\end{equation*}
where $\kappa_1 = \sqrt{2\|K_1\|_{C^6}}$ and $\kappa_2 = \sqrt{2\|K_2\|_{C^6}}$.
\end{lemma}

The following lemma is a straightforward consequence of \cite[Lemma A.1]{lang2022learning}.

\begin{lemma}\label{numerical-integrations}
Suppose that $f,g,\rho\in C^{1,1}([0,T]\times \Omega)$. Let $\{x_n,t_l\}_{n,l=1}^{N,L}$ be defined in \eqref{mesh}. 
Denote 
\begin{equation*}
\begin{aligned}
I(f,g,\rho):=\int_0^T\int_{\Omega}f(t,x)g(t,x)\rho_t(x)\mathrm{d}x\mathrm d t,\quad
I_{N,L}(f,g,\rho):=\sum_{n,l=1}^{NL}f(t_l,x_n)g(t_l,x_n)\rho_{t_l}(x_n)\Delta x \Delta t.
\end{aligned}
\end{equation*}
Then we have 
\begin{align*}
\left|I(f,g,\rho)-I_{NL}(f,g,\rho) \right| \leq 2T|\Omega| \widetilde{C}(f,g,\rho)(\Delta x + \Delta t),  
\end{align*}
where the constant 
\begin{align}\label{cofficient-C}
\widetilde{C}(f,g,\rho):=\|f\|_{C^{1,1}}\|g\|_{C^{1,1}}\|\rho\|_{C^{1,1}}.
\end{align}
\end{lemma}

Before analyzing the mesh-induced error, we first state the following lemma, whose detailed proof is provided in Appendix \ref{Proof of Lemma 3.5}.
\begin{lemma} \label{pre-estimator}
Let $K_1, K_2 \in C^6(\Omega \times \Omega)$ be two Mercer kernels. Suppose that $\rho \in C^{1,2}([0, T] \times \Omega)$. Then, for any $(\phi, \psi) \in \mathcal{H}_{K_1} \times \mathcal{H}_{K_2}$, we have the following bound:  
\begin{align*}
\|B(\phi, \psi) - B_{NL}(\phi, \psi)\|_{\mathcal{H}_{K_1} \times \mathcal{H}_{K_2}} 
\leq C_1\|(\phi,\psi)\|_{\mathcal{H}_{K_1}\times\mathcal{H}_{K_2}}(\Delta x + \Delta t)^{\frac{1}{2}},
\end{align*}
where 
\begin{align}\label{coffiential-C_1}
C_1=2T|\Omega|(\kappa_1+\kappa_2)^{\frac{3}{2}}(4+|\Omega|\|\nabla_t\rho\|_{\infty})|\rho\|^3_{C^{1,2}}.    
\end{align}
\end{lemma}

We now present an upper bound for the mesh-induced error.
\begin{theorem}[{\bf Mesh-induced error}]\label{discretization-error}
Let $K_1, K_2 \in C^6(\Omega \times \Omega)$ be two Mercer kernels. Suppose that $\rho \in C^{1,2}([0, T] \times \Omega)$. Then, for any $(V,W) \in \mathcal{H}_{K_1} \times \mathcal{H}_{K_2}$, we have the following bound for the mesh-induced error:  
\begin{multline*}
\|\left(V_{\lambda, NL}, W_{\lambda, NL}\right) - \left(V_{\lambda}^*, W_{\lambda}^*\right)\|_{\mathcal{H}_{K_1} \times \mathcal{H}_{K_2}} 
\\
\leq~  \left(\sqrt{\frac{2TC_T}{\lambda}}(\kappa_1+\kappa_2)+1\right)C_1\lambda^{-1} \|\left(V,W\right)\|_{\mathcal{H}_{K_1}\times\mathcal{H}_{K_2}}(\Delta x + \Delta t)^{\frac{1}{2}},
\end{multline*}
where the coefficient $C_1$ is defined in \eqref{coffiential-C_1}.
\end{theorem}

\begin{proof} 
By the operator representations \eqref{ope-rep-emp-esti} and \eqref{kernel-estimator1}, we decompose 
\begin{multline*}
\left(V_{\lambda, NL}, W_{\lambda, NL}\right) - \left(V_{\lambda}^*, W_{\lambda}^*\right)=(B_{NL}+\lambda)^{-1}B_{NL}(V,W)-(B+\lambda)^{-1}B(V,W) \\
=(B_{NL}+\lambda)^{-1}B_{NL}(V,W)-(B_{NL}+\lambda)^{-1}B(V,W)+(B_{NL}+\lambda)^{-1}B(V,W)- (B+\lambda)^{-1}B(V,W).
\end{multline*}
Since the operator norm satisfies $\|(B_{NL} + \lambda)^{-1}\| \leq \frac{1}{\lambda}$, applying Lemma \ref{pre-estimator}, we obtain:
\begin{multline}\label{bounds1}
\|(B_{NL} + \lambda)^{-1}B_{NL}(V, W) - (B_{NL} + \lambda)^{-1}B(V, W)\|_{\mathcal{H}_{K_1} \times \mathcal{H}_{K_2}} \\
\leq ~ \frac{1}{\lambda} \|B_{NL}(V, W) - B(V, W)\|_{\mathcal{H}_{K_1} \times \mathcal{H}_{K_2}} \\
\leq ~ \frac{C_1}{\lambda}\|(V,W)\|_{\mathcal{H}_{K_1}\times\mathcal{H}_{K_2}} (\Delta x + \Delta t)^{\frac{1}{2}},
\end{multline}
where $C_1$ is given in \eqref{coffiential-C_1}. On the other hand, we have that
\begin{multline*}
\|(B_{NL}+\lambda)^{-1}B(V,W)- (B+\lambda)^{-1}B(V,W)\|_{\mathcal{H}_{K_1} \times \mathcal{H}_{K_2}} \\
=~ \|(B_{NL}+\lambda)^{-1}(B-B_{NL})(B+\lambda)^{-1}B(V,W)\|_{\mathcal{H}_{K_1} \times \mathcal{H}_{K_2}}\\ 
\leq ~ \frac{1}{\lambda}\|(B-B_{NL})(B+\lambda)^{-1}B(V,W)\|_{\mathcal{H}_{K_1} \times \mathcal{H}_{K_2}}.
\end{multline*}

Since $\left(V_{\lambda}^*,W_{\lambda}^*\right)=(B+\lambda)^{-1}B(V,W)$ is the unique minimizer of the regularized statistical risk $R _\lambda(\phi,\psi) =\|A(\phi,\psi)-A(V,W)\|^2_{L^2}+\lambda\|(\phi,\psi)\|_{\mathcal{H}_{K_1} \times \mathcal{H}_{K_2}}^2,$ plugging $(\phi,\psi)=0$, we obtain that 
\begin{align*}
\|A\left(V_{\lambda}^*,W_{\lambda}^*\right)-A(V,W)\|^2_{L^2}+\lambda\|\left(V_{\lambda}^*,W_{\lambda}^*\right)\|_{\mathcal{H}_{K_1} \times \mathcal{H}_{K_2}}^2 <\| A(V,W)\|^2_{L^2}.
\end{align*}
Then by Proposition \ref{Bound-A}, we have
\begin{equation}\label{eq2}
\begin{aligned}
\|\left(V_{\lambda}^*,W_{\lambda}^*\right)\|_{\mathcal{H}_{K_1} \times \mathcal{H}_{K_2}} <\frac{1}{\sqrt{\lambda}}\| A(V,W)\|_{L^2}\leq \frac{\sqrt{2TC_T}(\kappa_1+\kappa_2)}{\sqrt{\lambda}}\| (V,W)\|_{\mathcal{H}_{K_1} \times \mathcal{H}_{K_2}}.
\end{aligned}
\end{equation}
Applying Lemma \ref{pre-estimator} to $\left(V_{\lambda}^*,W_{\lambda}^*\right)=(B+\lambda)^{-1}B(V,W)$ and combining it with equation \eqref{eq2}, we obtain that
\begin{equation}\label{bounds2}
\begin{aligned}
\frac{1}{\lambda}\|(B-B_{NL})(B+\lambda)^{-1}B(V,W)&\|_{\mathcal{H}_{K_1} \times \mathcal{H}_{K_2}}  \leq \frac{C_1}{\lambda}\|\left(V_{\lambda}^*,W_{\lambda}^*\right)\|_{\mathcal{H}_{K_1}\times\mathcal{H}_{K_2}}(\Delta x + \Delta t)^{\frac{1}{2}}\\
&\leq C_1\sqrt{2TC_T}(\kappa_1+\kappa_2)\lambda^{-\frac{3}{2}}\|\left(V,W\right)\|_{\mathcal{H}_{K_1}\times\mathcal{H}_{K_2}}(\Delta x + \Delta t)^{\frac{1}{2}}.
\end{aligned}
\end{equation}
Finally, by combining the bounds \eqref{bounds1} and \eqref{bounds2}, we obtain that with a probability at least $1-\delta$,
\begin{multline*}
\|\left(V_{\lambda, NL}, W_{\lambda, NL}\right) - \left(V_{\lambda}^*, W_{\lambda}^*\right)\|_{\mathcal{H}_{K_1} \times \mathcal{H}_{K_2}} 
\\
\leq~  \left(\sqrt{\frac{2TC_T}{\lambda}}(\kappa_1+\kappa_2)+1\right)C_1\lambda^{-1} \|\left(V,W\right)\|_{\mathcal{H}_{K_1}\times\mathcal{H}_{K_2}}(\Delta x + \Delta t)^{\frac{1}{2}}.
\end{multline*} 
\end{proof}

\subsection{The scheme-dependent  error}

We now analyze scheme-dependent  error. Let $u,w$ be defined in \eqref{fun-u}
and \eqref{fun-w}. Denote
\begin{equation*}
\left\{\begin{aligned}
&u_l^n=u(t_l,x_n),\quad w_{l,k}^{n,m}:=w(t_l,x_n;t_k,x_m),\quad (\delta_{x}^+\rho)_l^n:=(\delta_{x}^+\rho)(t_l,x_n),\\ 
&(\delta_{x}^+u)_l^n:=(\delta_{x}^+u)(t_l,x_n)=\rho_l^n\Delta(\phi+\psi*\rho_{t_l})(x_n)+(\delta_{x}^+\rho)_l^n\nabla(\phi+\psi*\rho_{t_l})(x_n),\\ 
&(\delta_{x}^+w)_{l,k}^{n,m}:=(\delta_{x}^+w)(t_l,x_n;t_k,x_m),\quad (\delta_{y}^+w)_{l,k}^{n,m}:=(\delta_{y}^+w)(t_l,x_n;t_k,x_m),\\
&(\delta_{xy}^+w)_{l,k}^{n,m}:=(\delta_{xy}^+w)(t_l,x_n;t_k,x_m).
\end{aligned}\right.
\end{equation*}

To obtain an upper bound for the scheme-dependent error,  we state the following two lemmas, whose detailed proofs are provided in Appendices \ref{Proof of Lemma 3.7} and \ref{Proof of Lemma 3.8}.
\begin{lemma}\label{neumerical-scheme}
Suppose that $K_1, K_2 \in C^6(\Omega \times \Omega)$ and $\rho \in C^{1,2}([0, T] \times \Omega)$. Then, for all $(\phi, \psi) \in \mathcal{H}_{K_1} \times \mathcal{H}_{K_2}$, we have:
\begin{equation*}
\left\{
\begin{aligned}
&|u_l^n - (\delta_x^+ u)_l^n| 
\leq (\kappa_1 + \kappa_2) \|\rho\|_{C^{0,2}} (\|\phi\|_{\mathcal{H}_{K_1}} + \|\psi\|_{\mathcal{H}_{K_2}}) \Delta x, \\
&|(\delta_x^+ w)_{l,k}^{n,m} - w_{l,k}^{n,m}|, 
|(\delta_y^+ w)_{l,k}^{n,m} - \delta_{xy}^+ w_{l,k}^{n,m}| 
\leq (\kappa_1^2 + \kappa_2^2) \|\rho\|^2_{C^{0,2}} \Delta x.
\end{aligned}
\right.
\end{equation*}
Moreover, if $\rho \in C^{3,2}([0, T] \times \Omega)$ and $U \in C_b^2(\mathbb{R})$, then for the gradient flow \eqref{gradiet-flow}, we have:
\[
|f(t_l, x_n) - f^\delta(t_l, x_n)| 
\leq \|\rho\|_{C^{2,0}} \Delta t + \|\rho\|_{C^{0,2}} \|U\|_{C_b^2} \Delta x.
\]
And for the Wasserstein Hamiltonian flow \eqref{Hamiltonian-flow}, we have:
\[
|f(t_l, x_n) - f^\delta(t_l, x_n)| 
\leq \|\rho\|_{C^{3,0}} \Delta t + \|\rho\|_{C^{0,2}} \|U\|_{C_b^2} \Delta x.
\]
\end{lemma}

\begin{lemma} \label{pre-estimator1}
Let $K_1, K_2 \in C^6(\Omega \times \Omega)$ be two Mercer kernels. Suppose that $\rho \in C^{1,2}([0, T] \times \Omega)$. Then, for any $(\phi, \psi) \in \mathcal{H}_{K_1} \times \mathcal{H}_{K_2}$, we have the following bound:  
\begin{align*}
\|B_{NL}(\phi,\psi)-B_{NL}^\delta(\phi,\psi)\|_{\mathcal{H}_{K_1}\times\mathcal{H}_{K_2}}   \leq C_2 \|(\phi,\psi)\|_{\mathcal{H}_{K_1}\times\mathcal{H}_{K_2}}(\Delta x)^{\frac{1}{2}},  
\end{align*}
where 
\begin{align}\label{cofficient-C2}
C_2:=T|\Omega|(\kappa_1+\kappa_2)^2\|\rho\|_{C^{0,2}}^3.    
\end{align}
\end{lemma}

We now present an upper bound for the scheme-dependent error of the gradient flow \eqref{gradiet-flow}. A similar analysis could be applied to the Wasserstein Hamiltonian flow \eqref{Hamiltonian-flow}.
\begin{theorem}[{\bf Scheme-dependent  error}]\label{numrical-error}
Let $K_1, K_2 \in C^6(\Omega \times \Omega)$ be two Mercer kernels. Suppose that $\rho \in C^{2,2}([0, T] \times \Omega)$ and $U\in C_b^2(\mathbb{R})$. Then, for any $(V,W) \in \mathcal{H}_{K_1} \times \mathcal{H}_{K_2}$, we have the following bound for the numerical  error:  
\begin{align*}
{\left\|\left(\widehat{V}_{\lambda, NL}, \widehat{W}_{\lambda, NL}\right) - \left(V_{\lambda, NL}, W_{\lambda, NL}\right)\right\|_{\mathcal{H}_{K_1} \times \mathcal{H}_{K_2}}}&\leq \frac{1}{\lambda} \left( C_3 (\Delta x + \Delta t) 
+ C_4 \|(V, W)\|^2_{\mathcal{H}_{K_1} \times \mathcal{H}_{K_2}} \Delta x \right) 
\\
+~ & \left(\sqrt{\frac{2TC_T}{\lambda}}(\kappa_1+\kappa_2)+1\right)\lambda^{-1} C_2 \|(V,W)\|_{\mathcal{H}_{K_1}\times\mathcal{H}_{K_2}}(\Delta x)^{\frac{1}{2}},
\end{align*}
where the constant $C_2$ is given in \eqref{cofficient-C2} and $C_3,C_4$ are given by
\begin{equation}\label{cofficient-C3}
\left\{\begin{aligned}
&C_3:= T|\Omega| (\kappa_1+\kappa_2)\|\rho\|_{C^{0,1}}^2\|\rho\|_{C^{2,2}}(1+\|U\|_{C_b^2}),\\
&C_4:=T|\Omega| (\kappa_1+\kappa_2)^2\|\rho\|_{C^{0,2}}^3.
\end{aligned}\right.    
\end{equation}
\end{theorem}

\begin{proof}
From the operator representation~\eqref{kernel-estimator}, we have
\begin{equation*}
\begin{aligned}
\left(\widehat{V}_{\lambda, NL}, \widehat{W}_{\lambda, NL}\right)
&= \sqrt{\frac{T|\Omega|}{NL}}\left(B_{NL}^{\delta}+\lambda I\right)^{-1}A_{NL}^{\delta*}\left(f_{NL}^{\delta}-\sqrt{\frac{NL}{T|\Omega|}}A_{NL}^{\delta}(V,W)+\sqrt{\frac{NL}{T|\Omega|}}A_{NL}^{\delta}(V,W)\right)\\[6pt]
&= \left(B_{NL}^{\delta}+\lambda I\right)^{-1}B_{NL}^{\delta}(V,W) + \left(B_{NL}^{\delta}+\lambda I\right)^{-1}A_{NL}^{\delta*}\left(\sqrt{\frac{T|\Omega|}{NL}}f_{NL}^{\delta}-A_{NL}^{\delta}(V,W)\right).
\end{aligned}
\end{equation*}

Following the derivation in Theorem~\ref{discretization-error} and applying Lemma~\ref{pre-estimator1}, we obtain
\begin{equation}\label{eqn-4}
\begin{aligned}
&\left\|\left(B_{NL}^{\delta}+\lambda I\right)^{-1}B_{NL}^{\delta}(V,W)-\left(B_{NL}+\lambda I\right)^{-1}B_{NL}(V,W)\right\|_{\mathcal{H}_{K_1}\times\mathcal{H}_{K_2}}\\[6pt]
&\quad\leq\left(\sqrt{\frac{2TC_T}{\lambda}}(\kappa_1+\kappa_2)+1\right)\frac{C_2(\Delta x+\Delta t)}{\lambda}\|(V,W)\|_{\mathcal{H}_{K_1}\times\mathcal{H}_{K_2}},
\end{aligned}
\end{equation}
where the constant \( C_2 \) is given by~\eqref{cofficient-C2}.

Next, we analyze the remaining term. By Lemma~\ref{neumerical-scheme}, we have the bound
\begin{equation}\label{eqn-3}
\begin{aligned}
\left|(\delta_x^+u)_l^n - f_l^n\right|
&\leq (\kappa_1+\kappa_2)\|\rho\|_{C^{0,2}}\|(V,W)\|_{\mathcal{H}_{K_1}\times\mathcal{H}_{K_2}}(\Delta x+\Delta t) + \|\rho\|_{C^{0,2}}\|U\|_{C_b^2}(\Delta x+\Delta t).
\end{aligned}
\end{equation}

Then using the definition of the dual operator \(A_{NL}^{\delta*}\), we obtain
\begin{equation*}
\begin{aligned}
&A_{NL}^{\delta*}\left(\sqrt{\frac{T|\Omega|}{NL}} f_{NL}^{\delta}-A_{NL}^{\delta}(V,W)\right)\\[6pt]
&\quad=\frac{T|\Omega|}{NL}\left(f_{NL}^{\delta}-\sqrt{\frac{NL}{T|\Omega|}}A_{NL}^{\delta}(V,W)\right)^\top C_{\rho}\left(\Delta_{\rho_L}^{(1,0)\delta}K_1(X_N,\cdot),\;\Delta_{\rho_L}^{(1,0)\delta}(K_2*\rho_L)(X_N,\cdot)\right).
\end{aligned}
\end{equation*}

Thus, combining this expression with the inequality~\eqref{eqn-3}, we derive:
\begin{equation*}
\begin{aligned}
&\left\|A_{NL}^{\delta*}\left(\sqrt{\frac{T|\Omega|}{NL}}f_{NL}^{\delta}-A_{NL}^{\delta}(V,W)\right)\right\|_{\mathcal{H}_{K_1}\times\mathcal{H}_{K_2}}^2\\[6pt]
&\quad\leq \frac{T^2|\Omega|^2}{N^2L^2}\sum_{l,n=1}^{L,N}\sum_{k,m=1}^{L,N}\left|(\delta_x^+u)_l^n-f_l^n\right|\left|(\delta_x^+u)_k^m-f_k^m\right|\left|C_\rho\left(\Delta_{\rho_L}^{(1,0)\delta}K_1,\Delta_{\rho_L}^{(1,0)\delta}(K_2*\rho_L)\right)\right|\\[6pt]
&\quad\leq T^2|\Omega|^2\|w\|_{\infty}\|\rho\|_{\infty}^2\|\delta_x^+u-f\|_{\infty}^2\\[6pt]
&\quad\leq C_3^2(\Delta x+\Delta t)^2 + C_4^2\|(V,W)\|_{\mathcal{H}_{K_1}\times\mathcal{H}_{K_2}}^2(\Delta x+\Delta t)^2,
\end{aligned}
\end{equation*}
where the constants \(C_3, C_4\) are defined explicitly as
\[
C_3:=T|\Omega|(\kappa_1+\kappa_2)\|\rho\|_{C^{0,2}}\|U\|_{C_b^2}\|\rho\|_\infty,\quad
C_4:=T|\Omega|(\kappa_1+\kappa_2)\|\rho\|_{C^{0,2}}^2\|\rho\|_\infty.
\]

Combining this with~\eqref{eqn-4}, we conclude
\begin{equation}\label{eqn-5}
\begin{aligned}
&\left\|\left(B_{NL}^{\delta}+\lambda I\right)^{-1}A_{NL}^{\delta*}\left(\sqrt{\frac{T|\Omega|}{NL}}f_{NL}^{\delta}-A_{NL}^{\delta}(V,W)\right)\right\|_{\mathcal{H}_{K_1}\times\mathcal{H}_{K_2}}\\[6pt]
&\quad\leq\frac{1}{\lambda}\left(C_3(\Delta x+\Delta t)+C_4\|(V,W)\|_{\mathcal{H}_{K_1}\times\mathcal{H}_{K_2}}(\Delta x+\Delta t)\right).
\end{aligned}
\end{equation}

Finally, combining inequalities~\eqref{eqn-4} and~\eqref{eqn-5}, we obtain the desired result.
\end{proof}

\subsection{The total reconstruction error}
Combining the bounds of the approximation error (see later on in \eqref{app-err}), the mesh-induced error in Theorem \ref{discretization-error}, and the scheme-dependent  error in Theorem \ref{numrical-error}, we immediately obtain bounds for the total reconstruction error with respect to the RKHS norm. We shall present the result for the gradient flow \eqref{gradiet-flow}. A similar analysis could be applied to the Wasserstein Hamiltonian flow \eqref{Hamiltonian-flow}.

\noindent{\bf Approximation error.}\quad
The analysis of the approximation error defined in \eqref{decom-error} is relatively standard and relies on a so-called source condition \cite{plato2018optimal,lu2019nonparametric}. Specifically, let $\gamma \in (0, 1)$, $S > 0$, and $B = A^*A$ as defined in \eqref{operatorB}. We assume that:  
\begin{align}\label{sou-con}
(V,W) \in \Omega_S^\gamma := \{\phi \in \mathcal{H}_{K_1}\times \mathcal{H}_{K_2}\mid \phi = B^\gamma \psi,  \psi \in \mathcal{H}_{K_1}\times \mathcal{H}_{K_2}, \, \|\psi\|_{\mathcal{H}_{K_1}\times \mathcal{H}_{K_2}} < S\}.
\end{align}
Under the source condition \eqref{sou-con}, the approximation error can be bounded in the RKHS norm as:  
\begin{align}\label{app-err}
\| (V_{\lambda}^*, W_{\lambda}^*) - (V, W) \|_{\mathcal{H}_{K_1} \times \mathcal{H}_{K_2}} 
\leq \lambda^{\gamma} \| B^{-\gamma}(V, W) \|_{\mathcal{H}_{K_1} \times \mathcal{H}_{K_2}},
\end{align}  
where $B^{-\gamma}(V, W)$ represents the pre-image of $(V, W)$ via the operator's spectral decomposition.

{
\begin{remark}\normalfont
Since $B$ is bounded, self-adjoint, and positive semidefinite by Proposition~\ref{Bound-A}, its fractional powers $B^\gamma$ are well defined for $\gamma\in(0,1)$ and satisfy $\ker(B^\gamma)=\ker(B)=\mathcal N_{\rho}$.
Hence,
\[
\overline{\bigcup_{S>0}\Omega_S^\gamma}
=\overline{\operatorname{Ran}(B^\gamma)}=(\ker B^\gamma)^\perp
=\mathcal N_{\rho}^\perp.
\]
Thus, the source condition naturally restricts the target to the identifiable subspace $\mathcal N_{\rho}^\perp$,  while the classes $\Omega_S^\gamma$ characterize different levels of regularity within this subspace. In particular, components in $\mathcal N_{\rho}$ cannot be identified from the observations and are therefore not controlled by the reconstruction error analysis.
\end{remark}}

\begin{theorem}[{\bf Total reconstruction error}]\label{Total reconstruction error}
Let $K_1, K_2 \in C^6(\Omega \times \Omega)$ be two Mercer kernels. Suppose that $\rho \in C^{2,2}([0, T] \times \Omega)$. Then, if $(V,W)$ satisfies the source condition \eqref{sou-con}, we have the following bound for the total reconstruction error:
\begin{equation}\label{total-error}
\begin{aligned}
\Big\|\left(\widehat{V}_{\lambda, NL}, \widehat{W}_{\lambda, NL}\right) - &\left(V, W\right)\Big\|_{\mathcal{H}_{K_1} \times \mathcal{H}_{K_2}} 
 \leq \lambda^{\gamma} \| B^{-\gamma}(V, W) \|_{\mathcal{H}_{K_1} \times \mathcal{H}_{K_2}} \\
&\quad + \frac{1}{\lambda} \left( C_3 (\Delta x + \Delta t) 
+ C_4 \|(V, W)\|^2_{\mathcal{H}_{K_1} \times \mathcal{H}_{K_2}} \Delta x \right) \\
&\quad + \left(\sqrt{\frac{2T C_T }{\lambda}}(\kappa_1+\kappa_2) + 1\right) \lambda^{-1} C_2 \|(V,W)\|_{\mathcal{H}_{K_1} \times \mathcal{H}_{K_2}} (\Delta x)^{\frac{1}{2}}\\
&\quad+ \left(\sqrt{\frac{2T C_T }{\lambda}}(\kappa_1+\kappa_2) + 1\right) \lambda^{-1}C_1  \|(V, W)\|_{\mathcal{H}_{K_1} \times \mathcal{H}_{K_2}} (\Delta x + \Delta t)^{\frac{1}{2}},
\end{aligned}
\end{equation}
where $C_1,C_2,C_3$ and $C_4$ are given by \eqref{coffiential-C_1}, \eqref{cofficient-C2} and \eqref{cofficient-C3}, respectively. 
Suppose the regularization parameter $\lambda$, as well as the time and space discretization parameters $L$ and $N$, satisfy the following conditions:
\begin{align}\label{parameter-scale}
\lambda \propto N^{-\alpha}, \quad L \propto N^{\beta}, \quad \alpha, \beta > 0.
\end{align}
Then if $\alpha \in \left(0, \frac{1}{3}\right)$, the convergence rate of the total reconstruction error is given by
\begin{equation*}
\Big\|\big(\widehat{V}_{\lambda, NL}, \widehat{W}_{\lambda, NL}\big) - \big(V, W\big)\Big\|_{\mathcal{H}_{K_1} \times \mathcal{H}_{K_2}} 
\propto \quad
\begin{cases}
N^{-\min\{\alpha\gamma, \frac{1}{2}(\beta - 3\alpha)\}}, & \quad 3\alpha < \beta \leq 1, \\[8pt]
N^{-\min\{\alpha\gamma, \frac{1}{2}(1 - 3\alpha)\}}, & \quad \beta > 1.
\end{cases}
\end{equation*}
\end{theorem}

\begin{proof}
The upper bound~\eqref{total-error} of the total reconstruction error follows directly by combining the approximation error bound~\eqref{app-err}, the mesh-induced error bound in Theorem~\ref{discretization-error}, and the scheme-dependent  error bound in Theorem~\ref{numrical-error}.

Next, we analyze the convergence order of this upper bound. Denote the four terms on the right-hand side of inequality~\eqref{total-error} by \( O_1, O_2, O_3, O_4 \), respectively. Recall that the step sizes satisfy \(\Delta x\propto N^{-1}\) and \(\Delta t\propto L^{-1}\propto N^{-\beta}\). Thus, we have the following scaling relationships:
\begin{equation*}
\left\{
\begin{aligned}
O_1 &\propto \lambda^{\gamma} \propto N^{-\alpha\gamma},\\[6pt]
O_2 &\propto \lambda^{-1}(N^{-1}+N^{-\beta}) \propto N^{-(1-\alpha)} + N^{-(\beta-\alpha)},\\[6pt]
O_3 &\propto \lambda^{-\frac{3}{2}} N^{-\frac{1}{2}} \propto N^{-\frac{1}{2}(1-3\alpha)},\\[6pt]
O_4 &\propto \lambda^{-\frac{3}{2}}(N^{-1}+N^{-\beta})^{\frac{1}{2}} \propto N^{-\frac{1}{2}(1-3\alpha)} + N^{-\frac{1}{2}(\beta-3\alpha)}.
\end{aligned}
\right.
\end{equation*}

Therefore, if \(\beta > 1\), the convergence rate of the total reconstruction error is given by
\[
N^{-\min\left\{\alpha\gamma,\; \frac{1}{2}(1-3\alpha)\right\}}.
\]

On the other hand, if \(3\alpha < \beta \leq 1\), the convergence rate of the total reconstruction error becomes
\[
N^{-\min\left\{\alpha\gamma,\; \frac{1}{2}(\beta - 3\alpha)\right\}}.
\]
\end{proof}

\begin{remark}[{\bf Numerical considerations}]\normalfont
The condition \(\beta = 1\) in \eqref{parameter-scale} implies that the time and space discretizations have the same order. According to the Theorem \ref{Total reconstruction error}, when \(\beta\) lies in the interval \(3\alpha < \beta < 1\), the overall convergence rate is non-decreasing as \(\beta\) increases. This indicates that increasing \(\beta\) can accelerate convergence. 
However, when \(\beta\) reaches \(1\), the convergence rate no longer changes. This suggests that further refining the time discretization beyond this point becomes ineffective for the convergence speed.   
\end{remark}

\subsection{Stability estimates for the Wasserstein Hamiltonian flows}
Theorem \ref{Total reconstruction error} ensures that the structure-preserving kernel estimator yields potential and interaction functions that closely approximate the data-generating potential and interaction functions in the RKHS norm in the presence of enough data. We now demonstrate that the flow of the learned Wasserstein Hamiltonian system uniformly approximates the flow of the underlying Wasserstein Hamiltonian system, thereby justifying the use of the RKHS norm. 
A similar derivation could be extended to Wasserstein gradient flow systems; see also \cite{carrillo2025sparse}.

\begin{proposition}[{\bf Hamiltonian flow as density transition equation} {\cite[Proposition 2]{chow2020wasserstein}}] \label{hamiltonian-flow}
Let $\left(X_t\right)_{0 \leq t<T}$ be a smooth diffeomorphism in $\mathbb{T}^d$ with $X_0=I d, \dot{X}_0=\nabla \Phi (I d)$ for some smooth function $\Phi \in C ^1(\mathbb{T}^d)$. Suppose that $X_t$ satisfies
\begin{align}\label{flow}
\frac{\mathrm d^2}{\mathrm d t^2} X_t=-\nabla (V+\rho_t*W)(X_t),
\end{align}
where $\rho_t=X_t \# \mu$ is the push-forward by the flow $X_t$ of some initial density $\mu \in \mathcal{P}_{+}\left(\mathbb{T}^d\right)$. Then, the density path $\rho_t=\rho(t, \cdot)$ is a solution of the Wasserstein Hamiltonian flow equation \eqref{Hamiltonian-flow} with $U=0$. 
\end{proposition}

\begin{remark}
\normalfont 
The flow map $(X_t)_{t\in[0,T]}$ of equation \eqref{flow} on $\mathbb{T}^d$ with $X_0=I d, \dot{X}_0=\nabla \Phi(Id)$ can be represented as 
\begin{align}\label{flow-rep}
    X_t(\mathbf x)=\mathbf x +\int_0^t\nabla \Phi(\mathbf x)\mathrm ds-\int_{0}^t\int_0^s\nabla (V+W*\rho_{\tau})(X_\tau(\mathbf x)) \mathrm d\tau\mathrm ds.
\end{align}
\end{remark}

In the following two lemmas, we show some properties of the flow map of equation \eqref{flow}.
\begin{lemma}\label{Lip-flow}
Let $(X_t)_{t\in[0,T]}$ be the flow map of equation \eqref{flow} with initial conditions $X_0=I d, \dot{X}_0=\nabla \Phi(Id)$ for some smooth function $\Phi\in C^2(\mathbb{T}^d)$. If $V,W\in C^2(\mathbb{T}^d)$, then for all $t\in[0,T]$, the map $X_t$ is Lipschitz with Lipschitz constant
\begin{align*}
L_{X_t}=\left(1+td\|\Phi\|_{C^2}\right)\exp\left\{\frac{1}{2}d\left(\|V\|_{C^2}+\|W\|_{C^2}\right)t^2 \right\}.
\end{align*}
\end{lemma}

\begin{lemma}\label{dist-control}
Let $(X_t)_{t\in[0,T]}, (\widehat X_t)_{t\in[0,T]}$ be two flow maps of equation \eqref{flow} with the same initial conditions $X_0=Id, \dot{X}_0=\nabla \Phi(Id)$ for some smooth function $\Phi\in C^2(\mathbb{T}^d)$, but by different functions $V,W$ and $\widehat{V}, \widehat{W}$, respectively.  If $V,W,\widehat{V}, \widehat{W}\in C^2(\mathbb{T}^d)$, then for all $t\in[0,T]$, the distance of these two maps $X_t$ and $\widehat X_t$ can be controlled as follows
\begin{align*}
\left\|X_t(\mathbf{x})-\widehat{X}_t(\mathbf{x})\right\|^2 \leq &\left(4t^2d^2\left(\| V- \widehat{V}\|^2_{C^2}+\|W -\widehat{W}\|_{C^2}^2\right)+4td\|\widehat W\|_{C^2} \int_0^t W_2(\rho_s,\widehat \rho_s) \mathrm{d} s\right)\\
&\exp\left\{2t^2 d^2\left(\|\widehat V\|^2_{C^2}+\|\widehat W\|^2_{C^2}\right)\right\}.
\end{align*}
\end{lemma}

\begin{proposition}[{\bf From discrete data to continuous Wasserstein flows}]\label{stability estimate}
Suppose that $V,W$ satisfy the source conditions. Let $\widehat V:=\widehat{V}_{\lambda,NL},\widehat W:=\widehat{W}_{\lambda,NL}$ be their kernel estimators with Mercel kernels $K_1,K_2\in C^4(\mathbb T^d\times\mathbb T^d)$, respectively. For initial data $\rho_0,\widehat\rho_0\in\mathcal{P}_2(\mathbb{T}^d)$, let $\rho:=(\rho_t)_{t\in[0,T]},\widehat\rho:=(\widehat\rho_t)_{t\in[0,T]}\in C^2\left([0,T],\mathcal{P}_2(\mathbb{T}^d)\right)$ be solutions of equation \eqref{Hamiltonian-flow} by functions $V,W$ and $\widehat{V}, \widehat{W}$, respectively. Then we have that 
\begin{align*}
W_2^2\left(\rho_t, \widehat{\rho}_t\right) \leq C_1(t)W_2^2\left(\rho_0, \widehat{\rho}_0\right) +C_2(t)\left(\kappa_1^2\| V- \widehat{V}\|^2_{\mathcal{H}_{K_1}}+\kappa_2^2\|W -\widehat{W}\|_{\mathcal{H}_{K_2}}^2\right) ,
\end{align*}
where $L_{\widehat{X}_t}$ is defined in Lemma \ref{Lip-flow} of the flow map $\widehat{X}_t$, and $C_1(t), C_2(t)$ are defined as
\begin{equation}\label{coeff}
\begin{aligned}
C_1(t)& = L_{\widehat{X}_t} \exp\left\{\int_0^t4sd\|\widehat W\|_{C^2}\exp\left\{2s^2 d^2\left(\|\widehat V\|^2_{C^2}+\|\widehat W\|^2_{C^2}\right)\right\}\mathrm d s\right\}, \\
C_2(t) &=4t^2d^2\exp\left\{2t^2 d^2\left(\|\widehat V\|^2_{C^2}+\|\widehat W\|^2_{C^2}\right)+\int_0^t4sd\|\widehat W\|_{C^2}\exp\left\{2s^2 d^2\left(\|\widehat V\|^2_{C^2}+\|\widehat W\|^2_{C^2}\right)\right\}\mathrm d s\right\}.
\end{aligned}
\end{equation}
\end{proposition}

\section{Numerical experiments}\label{Numerical experiments}

In this section, we investigate the numerical performance of the proposed kernel framework for recovering the external potential $V$ and the interaction potential $W$ from trajectory observations. In Subsection~\ref{Wasserstein gradient flow}, we consider one-dimensional Wasserstein gradient flows, while Subsection~\ref{Wasserstein_Hamiltonian_flow} concerns two-dimensional Wasserstein Hamiltonian flows. For each class of dynamics, we examine both a successful joint reconstruction and a nonstationary trajectory with an explicitly unobservable interaction component. We also compare the kernel method with sparse learning in a separate Hamiltonian example with a highly nonconvex external potential and known $W=0$. 

\subsection{Numerical implementation and error measures}\label{subsec:num-setting}

Throughout Examples~1--4, we consider the periodic domain $\Omega_d=[-\pi,\pi)^d$. The periodic convolution is defined by
\[
(W*\rho)(x)=\int_{\Omega_d}W(x-y)\rho(y)\,dy.
\]
The internal energy is assumed known and is chosen as
\begin{equation}\label{eq:num-internal}
U(\rho)=\frac{\kappa}{m}\rho^m,\qquad U'(\rho)=\kappa\rho^{m-1},\qquad m=2,\quad \kappa=0.2.
\end{equation}

\noindent{\bf Kernel approximation.}\quad
For the periodic experiments, we choose $K_1=K_{\ell_V}^0$ and $K_2=K_{\ell_W}^0$, where
\begin{equation*}
K_\ell^0(x,z)=\sum_{n\in\mathbb Z^d}\exp\left(-\frac{|x-z+2\pi n|^2}{2\ell^2}\right)-\left(\frac{\ell}{\sqrt{2\pi}}\right)^d.
\end{equation*}
The periodic sum is evaluated to numerical precision. The subtracted term removes the constant Fourier mode, fixing the additive ambiguity by imposing zero spatial mean. We use the finite kernel expansions
\[
V_h(x)=\sum_j a_jK_{\ell_V}^0(x,z_j),\qquad W_h(x)=\sum_j b_jK_{\ell_W}^0(x,s_j),
\]
with uniformly distributed centers, and minimize the regularized empirical loss introduced in Section~\ref{Structure-preserving kernel ridge regression and numerical schemes} over these approximation spaces. The kernel basis is orthonormalized in the RKHS inner product, so that the regularization penalty becomes the squared Euclidean norm of the transformed coefficients. We evaluate potential derivatives analytically through the kernel expansions and compute periodic convolutions by fast Fourier transform (FFT). The numerical parameters are listed in Table~\ref{tab:num-settings}. Here, $C_t$ denotes the number of forward time steps between consecutive saved observations. The bandwidths and regularization parameters were chosen empirically by manual tuning. The periodic Gaussian bandwidths satisfy $\ell_V\leq\ell_W$, consistent with the RKHS compatibility condition in Example~\ref{exm:Compatibility}.

\begin{table}[htbp]
\centering\small
\caption{Numerical settings for Examples~1--4.}\label{tab:num-settings}
\begin{tabular}{c|c|c|c|c|c|c|c}
\hline
Example & Spatial grid & $\delta t$ & $T$ & $C_t$ & Kernel centers & $(\ell_V,\ell_W)$ & $(\lambda_V,\lambda_W)$\\ \hline
1 & $1200$ & $5\times10^{-5}$ & $0.5$ & $1$ & $80$ & $(0.5,0.7)$ & $(10^{-8},10^{-7})$\\
2 & $1200$ & $5\times10^{-5}$ & $0.5$ & $1$ & $80$ & $(0.5,0.7)$ & $(10^{-8},10^{-7})$\\
3 & $96\times96$ & $10^{-4}$ & $0.2$ & $10$ & $12\times12$ & $(0.8,0.8)$ & $(3\times10^{-13},3\times10^{-13})$\\
4 & $96\times96$ & $10^{-4}$ & $0.2$ & $10$ & $8\times8$ & $(0.8,0.8)$ & $(10^{-10},10^{-10})$\\ \hline
\end{tabular}
\end{table}

\noindent{\bf Error measures.}\quad
To quantify reconstruction accuracy, we compare the zero-mean estimated potentials with their zero-mean ground truths using
\begin{equation*}
E_V=\frac{\|\widehat V-V\|_{L^2(\Omega_d)}}{\|V\|_{L^2(\Omega_d)}},\qquad E_W=\frac{\|\widehat W-W\|_{L^2(\Omega_d)}}{\|W\|_{L^2(\Omega_d)}}.
\end{equation*}
These norms are approximated by quadrature on the full spatial evaluation grid. We also report the relative observation-equation residual
\begin{equation*}
R_{\mathrm{fit}}=\left(\frac{\sum_{q\in\mathcal I}\omega_q\rho_q|A_\rho^\delta(\widehat V,\widehat W)_q-f_q^\delta|^2}{\sum_{q\in\mathcal I}\omega_q\rho_q|f_q^\delta|^2}\right)^{1/2},
\end{equation*}
where $\mathcal I$ is the set of observations used for training and $\omega_q$ denotes the corresponding quadrature weight. For the Hamiltonian experiments, $\mathcal I$ consists of sampled space--time points. The residual $R_{\mathrm{true}}$ is evaluated at the same points with the same weights by substituting the true potentials for the estimates. The fitted residual may be smaller than the true-potential residual because the estimator can partially compensate for discretization errors. Potential errors measure reconstruction accuracy, whereas these training residuals measure agreement with the discrete observation equation.

We also examine the temporal excitation of Fourier modes through
\begin{equation*}
\max_l|\widehat\rho_{t_l}(k)-\widehat\rho_0(k)|,\qquad \widehat\rho_t(k)=\int_{\Omega_d}\rho_t(x)e^{-ik\cdot x}\,dx,\qquad k\in\mathbb Z^d.
\end{equation*}
The zero mode is constant by mass conservation. The computed excitation describes temporal variation in the observed Fourier modes, but does not by itself establish the infinite-dimensional condition in Proposition~\ref{prop:joint-ident-fourier}.

\subsection{Wasserstein gradient flows}\label{Wasserstein gradient flow}
We first consider the one-dimensional Wasserstein gradient flow
\begin{equation*}
\partial_t\rho=\partial_x\left[\rho\,\partial_x\bigl(U'(\rho)+V+W*\rho\bigr)\right]
\end{equation*}
on $\Omega_1=[-\pi,\pi)$, with $U'(\rho)=\kappa\rho^{m-1}$, $m=2$, and $\kappa=0.2$. The following two examples illustrate how the observed density trajectory affects joint reconstruction.

\noindent{\bf Data generation.}\quad
The reference trajectories are computed using conservative finite-volume fluxes with minmod piecewise-linear reconstruction and third-order strong-stability-preserving Runge--Kutta time stepping; see \cite{carrillo2015finite,gottlieb2001strong}. Periodic convolution is evaluated by FFT. Density positivity and the velocity-based Courant–Friedrichs–Lewy condition are checked at each Runge--Kutta stage. The numerical settings are listed in Table~\ref{tab:num-settings}. For both examples, all generated spatial observations and time intervals are used for training, and the maximum mass errors are of order $10^{-16}$.

\subsubsection{Example 1: Joint recovery of external and interaction potentials}
We choose
\[
V(x)=\sin(2x)+0.3\cos x,\qquad \rho_0(x)=\frac{1+0.3\cos x+0.2\sin x+0.4\cos(2x)}{2\pi}.
\]
The interaction potential is the zero-mean periodization of a sum of Gaussian functions:
\begin{equation*}
W(r)=\sum_{n\in\mathbb Z}w_{\mathbb R}(r+2\pi n)+\frac2\pi,\qquad w_{\mathbb R}(r)=-\frac{2e^{-r^2}}{\sqrt\pi}-\frac{2e^{-r^2/2}}{\sqrt{2\pi}}.
\end{equation*}
The constant term fixes the spatial mean without changing the dynamics.

Figure~\ref{fig:num-example1-data} displays representative density profiles and the corresponding Fourier excitation. Both odd and even modes exhibit temporal variation. The first-frequency terms in the initial density break the half-period symmetry considered in the next example.
\begin{figure}[htbp]
\centering
\includegraphics[width=0.40\linewidth]{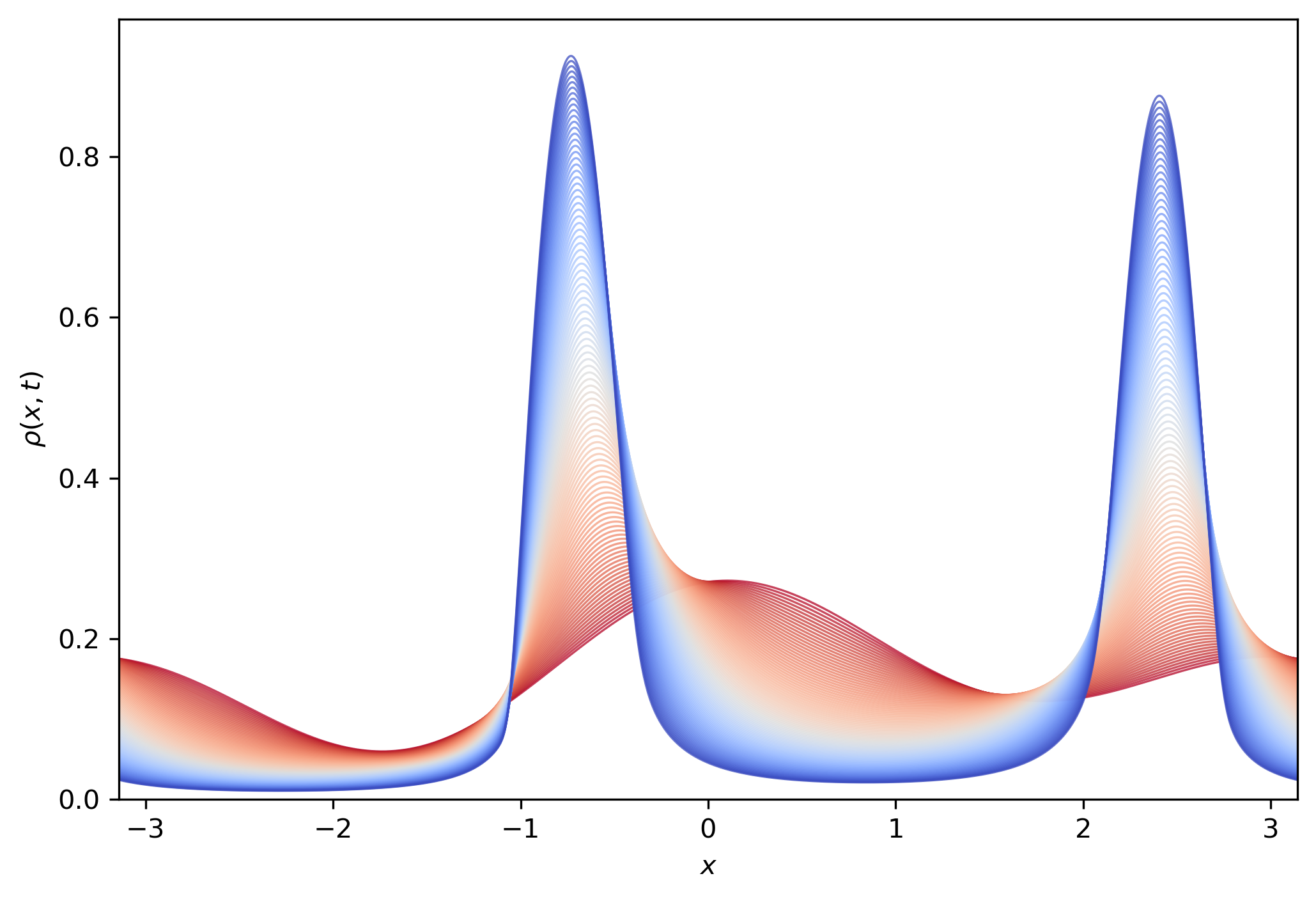}
\includegraphics[width=0.48\linewidth]{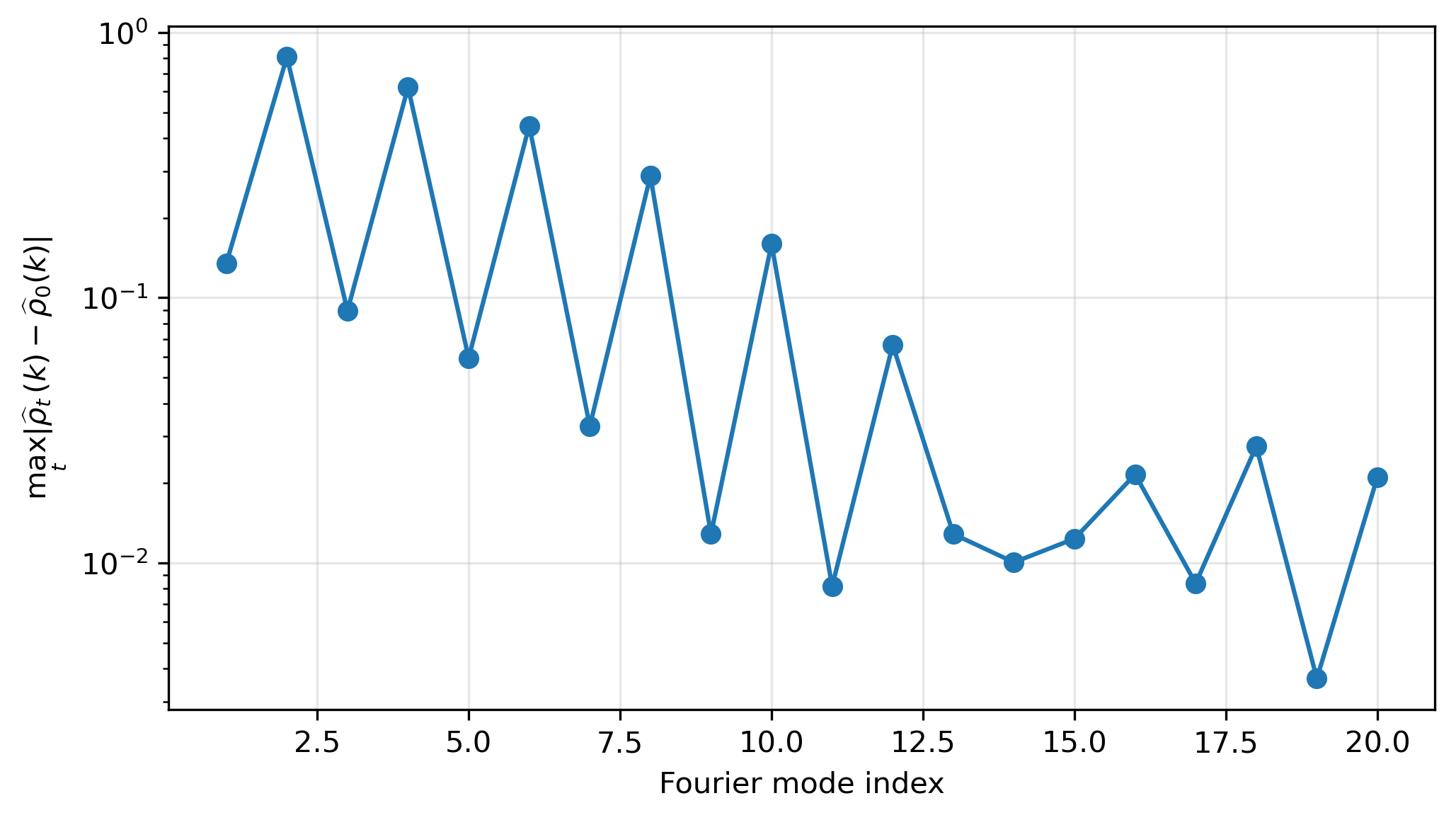}
\caption{Example~1: density evolution (left) and temporal Fourier excitation (right). Selected density snapshots are shown with colors changing from red to blue as time increases.}\label{fig:num-example1-data}
\end{figure}

Figure~\ref{fig:num-example1-recovery} shows that the estimator accurately reproduces the oscillatory external potential and the smooth interaction potential. The relative potential errors and equation residuals are reported in Table~\ref{tab:num-results}. These results demonstrate accurate simultaneous reconstruction from a single observed density trajectory.
\begin{figure}[htbp]
\centering
\includegraphics[width=0.96\linewidth]{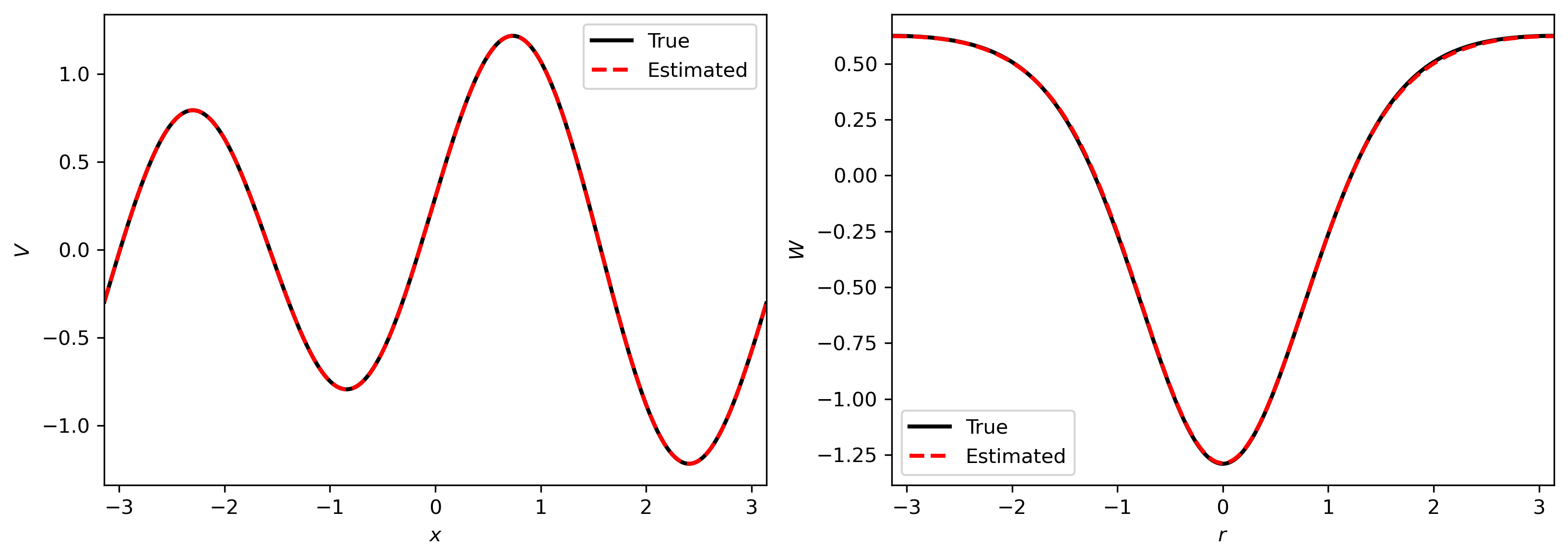}
\caption{Example~1: simultaneous reconstruction of the external potential (left) and the zero-mean periodized interaction potential (right).}\label{fig:num-example1-recovery}
\end{figure}

\subsubsection{Example 2: A nonstationary but non-identifiable trajectory}
We consider the same gradient-flow equation with
\[
V(x)=\sin(2x),\qquad W(x)=\cos(2x)+\cos x,\qquad \rho_0(x)=\frac{1+0.5\cos(2x)}{2\pi}.
\]
The initial density and external potential are $\pi$-periodic. As shown in Example~\ref{ex:joint-ident-nonstationary}, translation equivariance of the equation and uniqueness of the smooth solution preserve this symmetry, giving
\[
\widehat\rho_t(2j+1)=0,\qquad (\cos x)*\rho_t=0,\qquad j\in\mathbb Z.
\]
Consequently, $W_{\mathrm{obs}}(x)=\cos(2x)$ and $W=W_{\mathrm{obs}}+\cos x$ generate the same density trajectory with the same external potential. Thus, $(0,\cos x)$ is a nonconstant null-space direction of the observation operator.

Figure~\ref{fig:num-example2-data} shows that the trajectory is nonstationary. Nevertheless, odd Fourier modes remain at roundoff level, whereas even modes exhibit temporal variation. Nonstationarity alone therefore does not ensure joint identifiability.
\begin{figure}[htbp]
\centering
\includegraphics[width=0.40\linewidth]{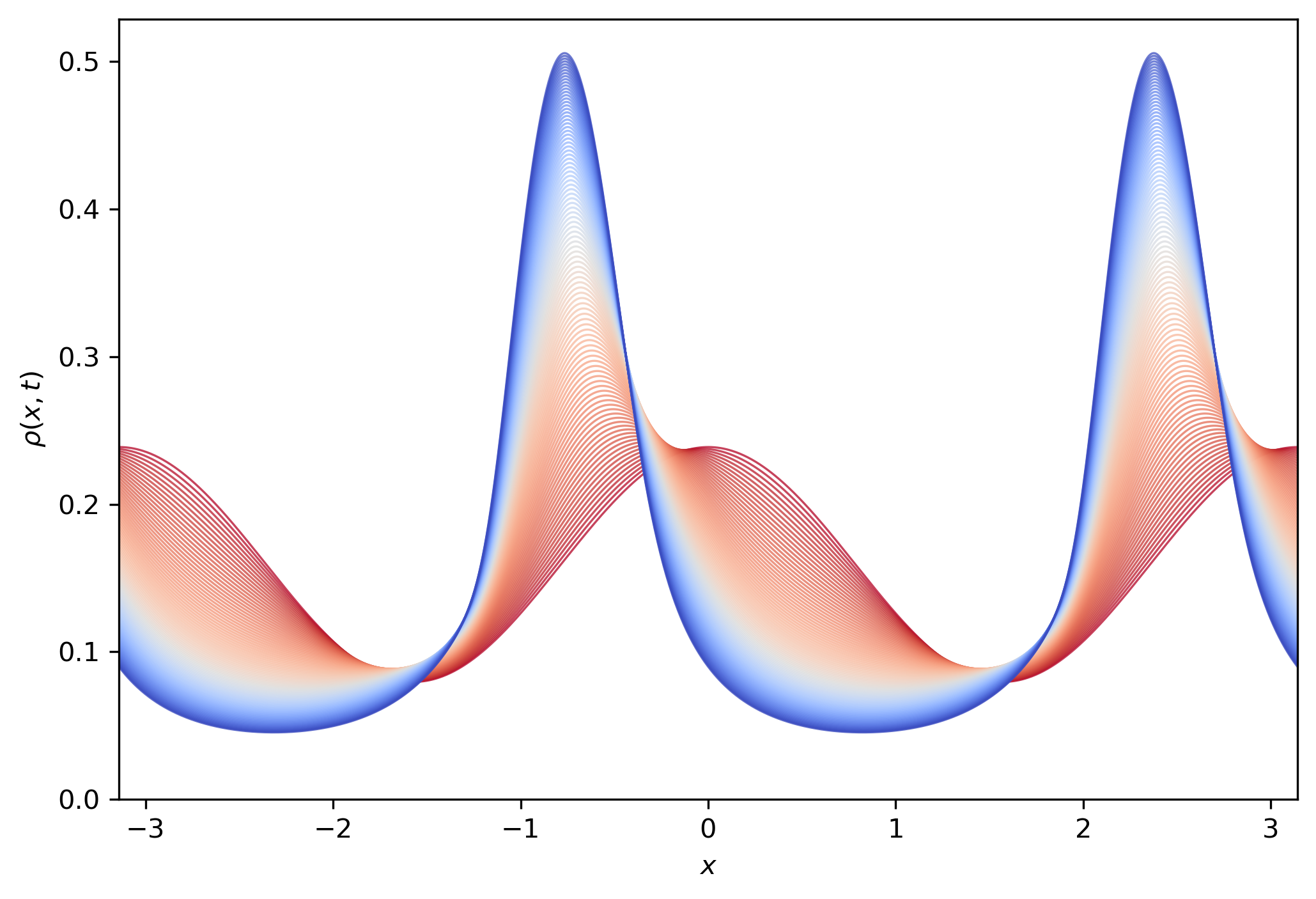}
\includegraphics[width=0.48\linewidth]{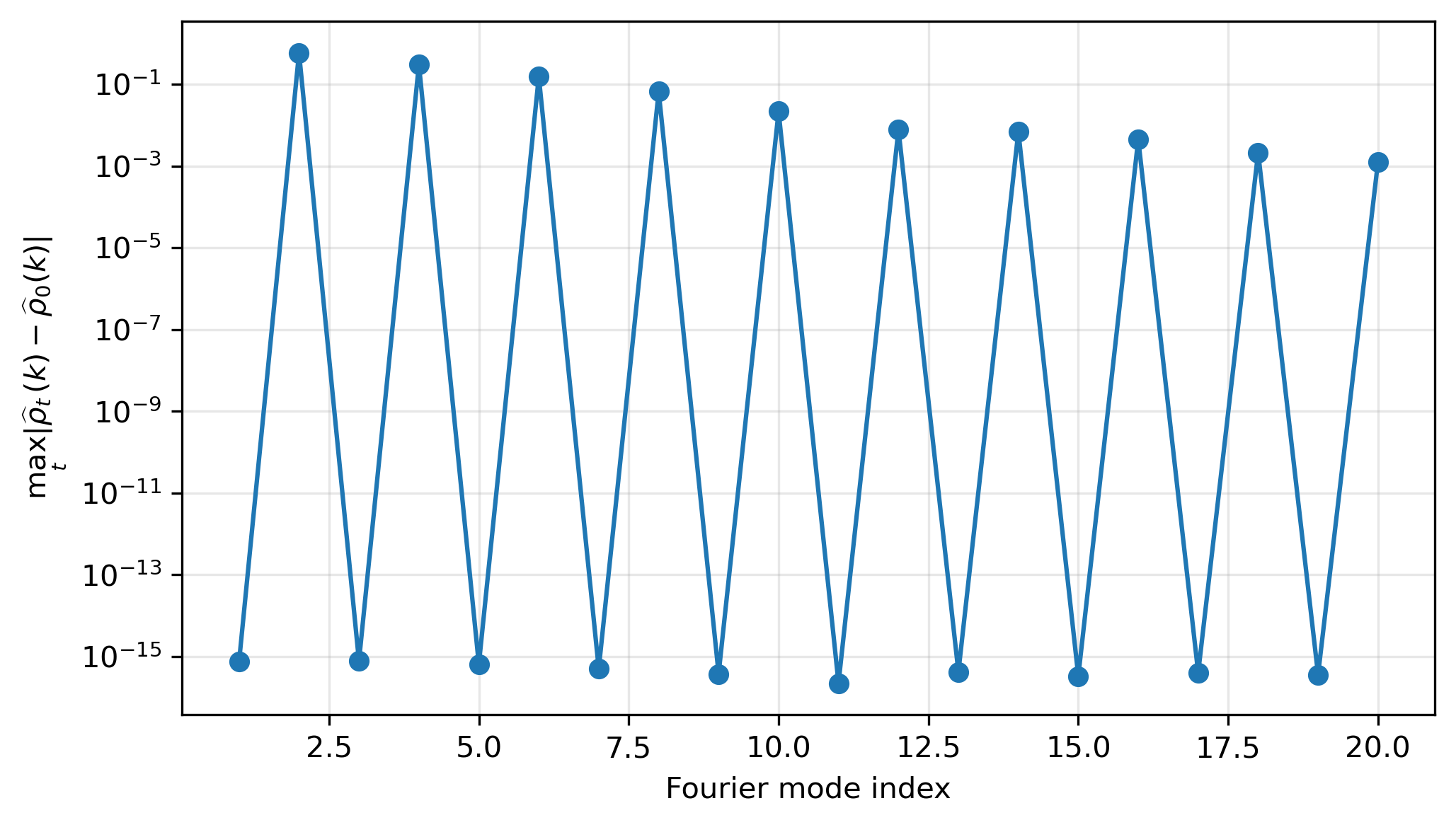}
\caption{Example~2: nonstationary density evolution and Fourier excitation. Odd modes remain at roundoff level.}\label{fig:num-example2-data}
\end{figure}

The reconstructed external potential is accurate, while the interaction estimator approximates $W_{\mathrm{obs}}$ with relative error $4.0583\times10^{-4}$ and misses the $\cos x$ component; see Figure~\ref{fig:num-example2-recovery} and Table \ref{tab:num-results}. 
This behavior is consistent with minimum-norm regularization: the invisible component does not improve the data fit but increases the periodic Gaussian RKHS penalty.
\begin{figure}[htbp]
\centering
\includegraphics[width=0.98\linewidth]{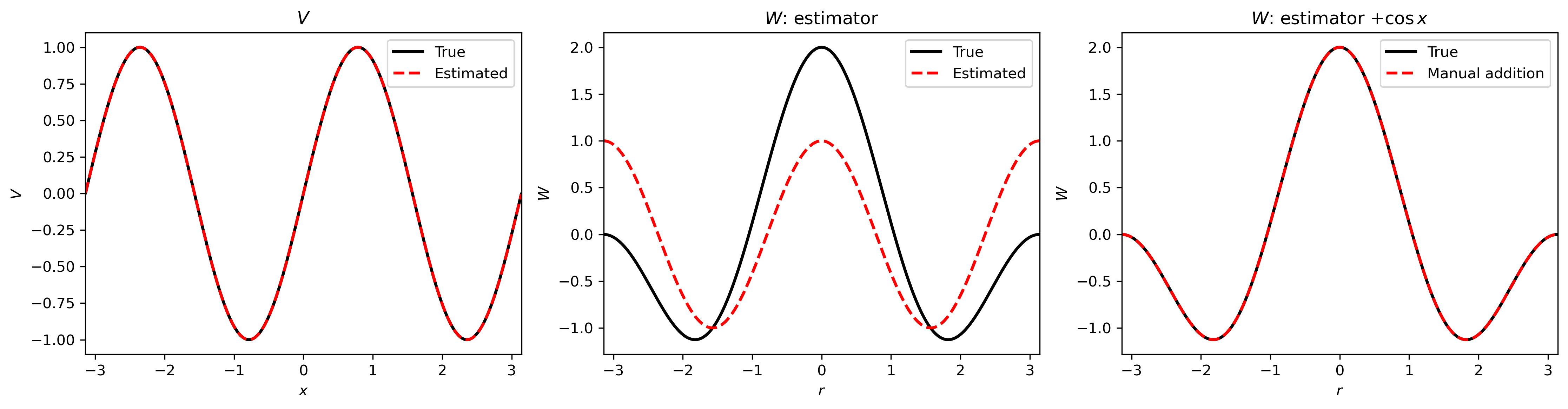}
\caption{Example~2: reconstruction of $V$, the interaction estimator, and the interaction estimator after manually adding $\cos x$.}\label{fig:num-example2-recovery}
\end{figure}

For illustration, manually adding $\cos x$ reduces the error against the complete truth to $2.8697\times10^{-4}$ without changing the equation residual to the reported precision. Thus, accurate fitting of the observations does not imply recovery of the unobservable interaction component. This illustrates the failure of the identifiability condition in Proposition~\ref{prop:joint-ident-fourier}: the odd Fourier modes of the density vanish throughout the observation interval, leaving the corresponding interaction components unobservable.

\begin{table}[htbp]
\centering\small
\caption{Relative errors and observation-equation residuals for Examples~1--4.}\label{tab:num-results}
\begin{tabular}{c|c|c|c|c}
\hline
Example & $E_V$ & $E_W$ & $R_{\mathrm{fit}}$ & $R_{\mathrm{true}}$\\ \hline
1 & $5.3095\times10^{-4}$ & $5.1651\times10^{-3}$ & $3.737642\times10^{-4}$ & $3.954938\times10^{-4}$\\
2 & $1.6578\times10^{-4}$ & $7.0711\times10^{-1}$ & $2.551356\times10^{-4}$ & $2.732671\times10^{-4}$\\
3 & $2.5204\times10^{-3}$ & $3.5056\times10^{-2}$ & $1.999942\times10^{-4}$ & $1.418599\times10^{-3}$\\
4 & $1.7005\times10^{-3}$ & $7.8455\times10^{-1}$ & $3.242360\times10^{-4}$ & $1.396757\times10^{-3}$\\ \hline
\end{tabular}
\end{table}

\subsection{Wasserstein Hamiltonian flows}\label{Wasserstein_Hamiltonian_flow}
We next consider Wasserstein Hamiltonian flows on $\Omega_2=[-\pi,\pi)^2$. To avoid explicitly evaluating the Wasserstein Christoffel operator, we use the equivalent first-order system
\begin{equation}\label{eq:first_order_whf}
\partial_t\rho+\nabla\cdot(\rho\nabla\Phi)=0,\qquad \partial_t\Phi+\frac12|\nabla\Phi|^2=-\bigl(U'(\rho)+V+W*\rho\bigr).
\end{equation}
Both $\rho$ and the velocity potential $\Phi$ are observed. Applying the weighted Laplacian $\Delta_\rho$ to the Hamilton--Jacobi equation gives the observation equation
\[
\Delta_\rho(V+W*\rho)=-\Delta_\rho\left(\partial_t\Phi+\frac12|\nabla\Phi|^2+U'(\rho)\right).
\]
We discretize this identity to construct the forcing $f^\delta$ used in the empirical loss. The internal energy is specified in \eqref{eq:num-internal}.

\noindent{\bf Data generation.}\quad
The Hamilton--Jacobi equation is discretized using a Godunov scheme, while the continuity equation is advanced by an upwind finite-volume method; see \cite{osher1988fronts,leveque2002finite}. The numerical settings are listed in Table~\ref{tab:num-settings}. For each example, we use $8000$ uniformly sampled space--time collocation points from the interior observation times, with random seed $123$.

\subsubsection{Example 3: Joint recovery in a two-dimensional Hamiltonian system}
We choose
\begin{align*}
V(x_1,x_2)&=\cos(2x_1)+0.6\cos(2x_2)+0.3\sin(2x_1)\sin(2x_2)+0.5\sin x_1+0.4\cos x_2,\\
W(z_1,z_2)&=\cos(2z_1)+0.5\cos(2z_2)+\cos z_1+\cos z_2,
\end{align*}
with initial conditions
\[
\rho_0(x_1,x_2)=\frac{1+0.2\cos(2x_1)+0.2\cos(2x_2)+0.15\cos x_1+0.1\sin x_2}{(2\pi)^2},\qquad \Phi_0=0.
\]
The external potential contains a mixed-coordinate term, while the first-frequency terms in $V$ and $\rho_0$ break the half-period symmetries considered in the next example.

\begin{figure}[htbp]
\centering
\includegraphics[width=\linewidth]{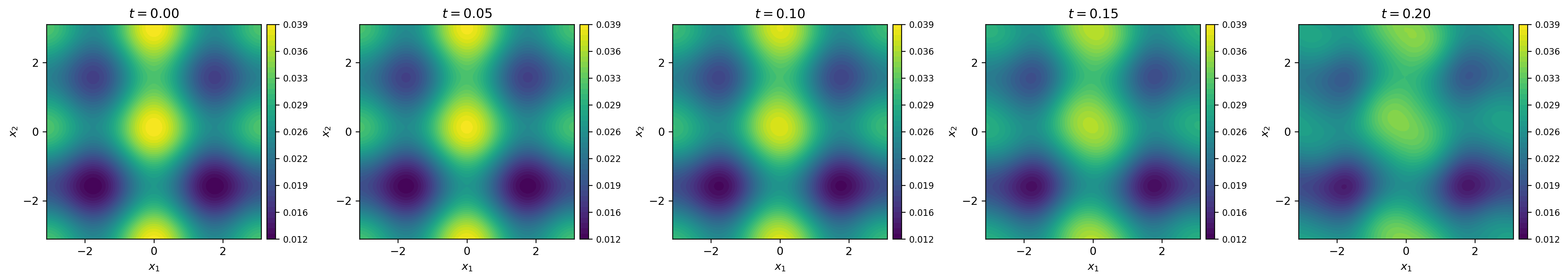}
\caption{Example~3: density snapshots at $t=0,0.05,0.10,0.15,0.20$.}\label{fig:num-example3-data}
\end{figure}

Figure~\ref{fig:num-example3-data} displays five density snapshots. The density is transported and deformed by the combined forces. The Fourier excitation in the left panel of Figure~\ref{fig:num-ham-excitation} shows temporal variation of the inspected first- and higher-frequency modes.

\begin{figure}[htbp]
\centering
\includegraphics[width=0.48\linewidth]{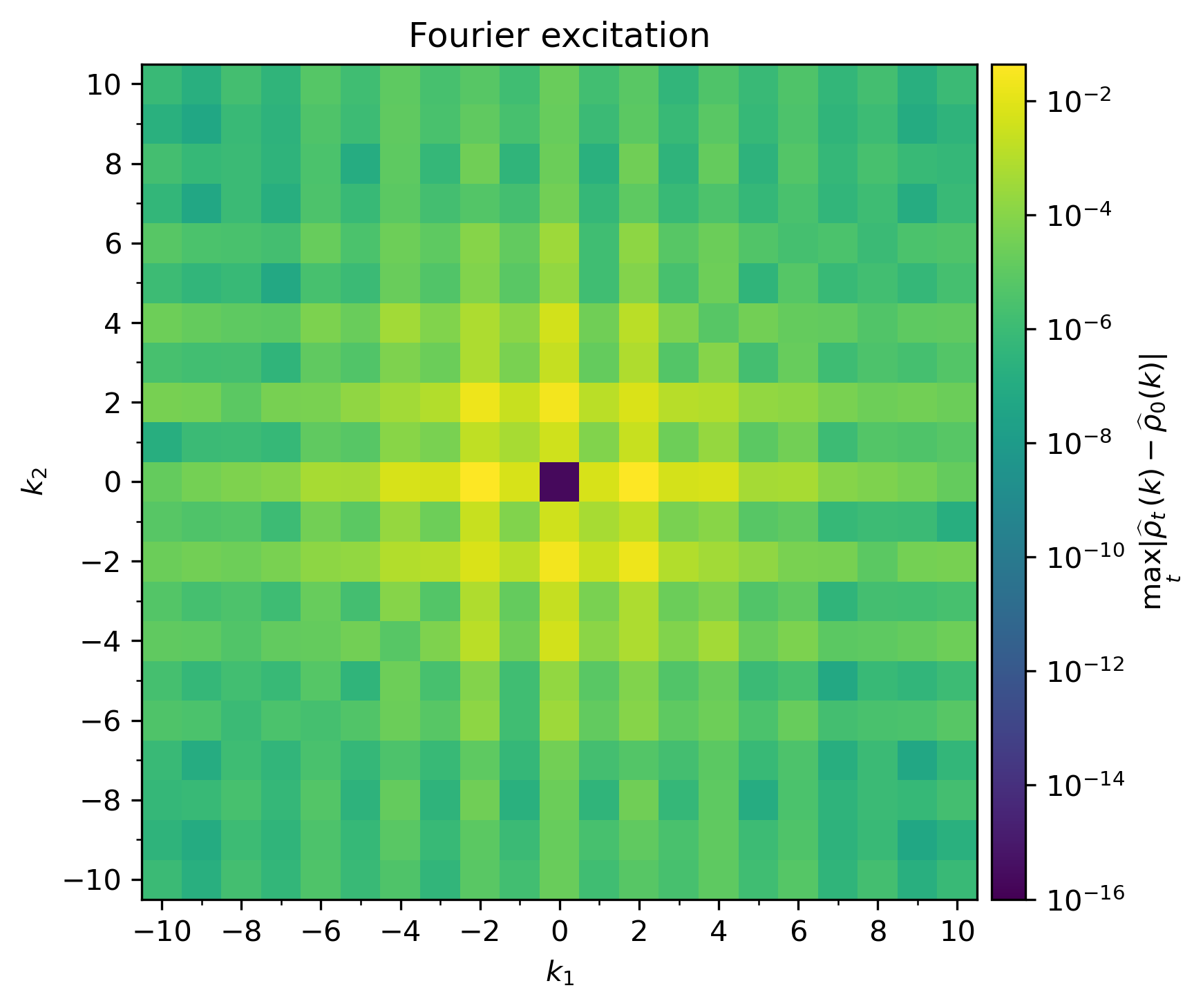}
\hfill
\includegraphics[width=0.48\linewidth]{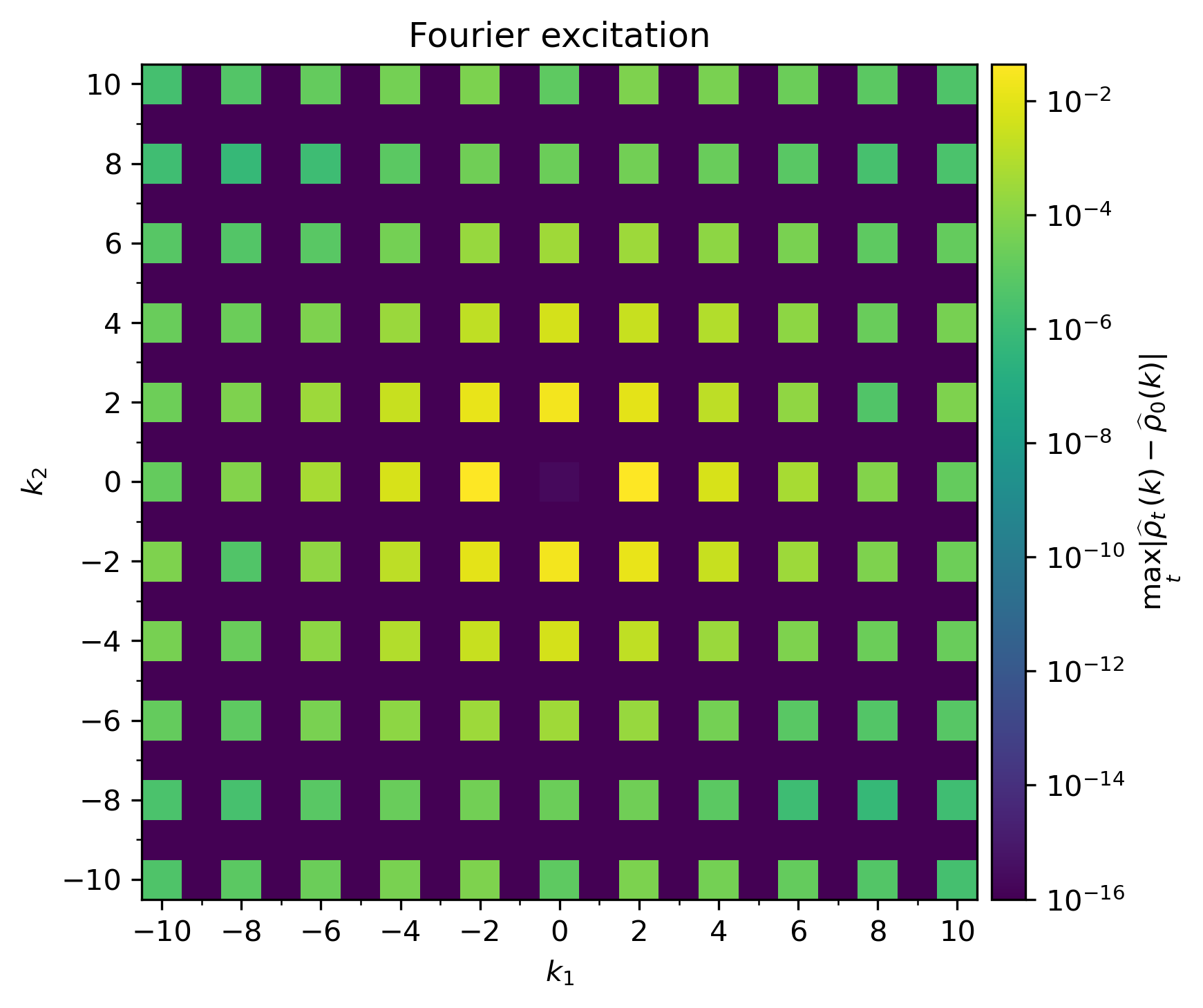}
\caption{Temporal Fourier excitation for Example~3 (left) and Example~4 (right). In Example~4, modes with at least one odd index remain at roundoff level.}\label{fig:num-ham-excitation}
\end{figure}

Figure~\ref{fig:num-example3-recovery} shows the joint reconstructions. The estimator reproduces the oscillatory and mixed-coordinate structure of $V$ and the complete interaction potential, with the errors reported in Table~\ref{tab:num-results}. In particular, the first-frequency interaction components, which are invisible in Example~4, are recovered here.
\begin{figure}[htbp]
\centering
\includegraphics[width=\linewidth]{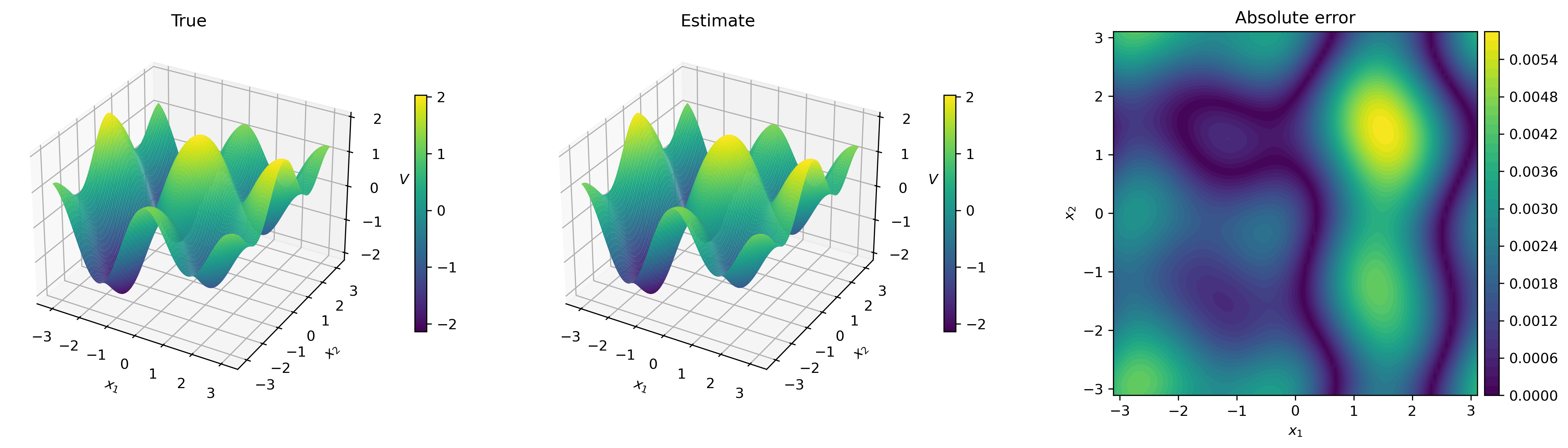}\\
\includegraphics[width=\linewidth]{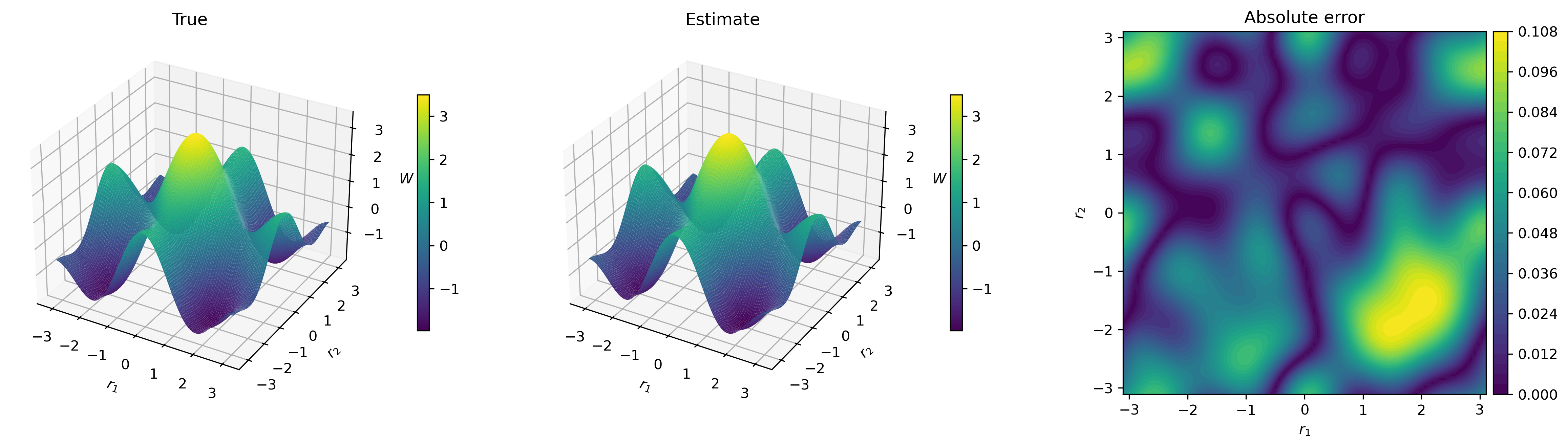}
\caption{Example~3: true potential, joint reconstruction, and absolute error for $V$ (top) and $W$ (bottom).}\label{fig:num-example3-recovery}
\end{figure}

\subsubsection{Example 4: A nonstationary Hamiltonian trajectory with invisible interaction modes}
We retain the interaction potential of Example~3 and choose
\[
V(x_1,x_2)=\cos(2x_1)+0.6\cos(2x_2)+0.3\sin(2x_1)\sin(2x_2),
\]
with initial conditions
\[
\rho_0(x_1,x_2)=\frac{1+0.2\cos(2x_1)+0.2\sin(2x_2)+0.1\cos(2x_1)\cos(2x_2)}{(2\pi)^2},\qquad \Phi_0=0.
\]
The initial conditions and external potential are invariant under a shift by $\pi$ in either coordinate. Uniqueness of the smooth solution preserves these symmetries, implying
\[
\widehat\rho_t(k_1,k_2)=0\qquad\text{whenever $k_1$ or $k_2$ is odd}.
\]
Consequently, setting
\[
\psi(z_1,z_2)=\cos z_1+\cos z_2,\qquad W_{\mathrm{obs}}(z_1,z_2)=\cos(2z_1)+0.5\cos(2z_2),
\]
we have $\psi*\rho_t=0$. The pairs $(V,W_{\mathrm{obs}})$ and $(V,W_{\mathrm{obs}}+\psi)$ generate the same density and velocity-potential trajectory with a common gauge. Additional observations of $\Phi$ therefore do not remove this ambiguity.

Figure~\ref{fig:num-example4-data} shows the nonstationary density evolution. Although $\|\rho_T-\rho_0\|_{L^2}=1.2179\times10^{-2}$, half-period symmetry errors remain below $5\times10^{-16}$. The right panel of Figure~\ref{fig:num-ham-excitation} shows that modes with at least one odd index remain at roundoff level, with a maximum inspected excitation of $2.655\times10^{-17}$.

\begin{figure}[htbp]
\centering
\includegraphics[width=\linewidth]{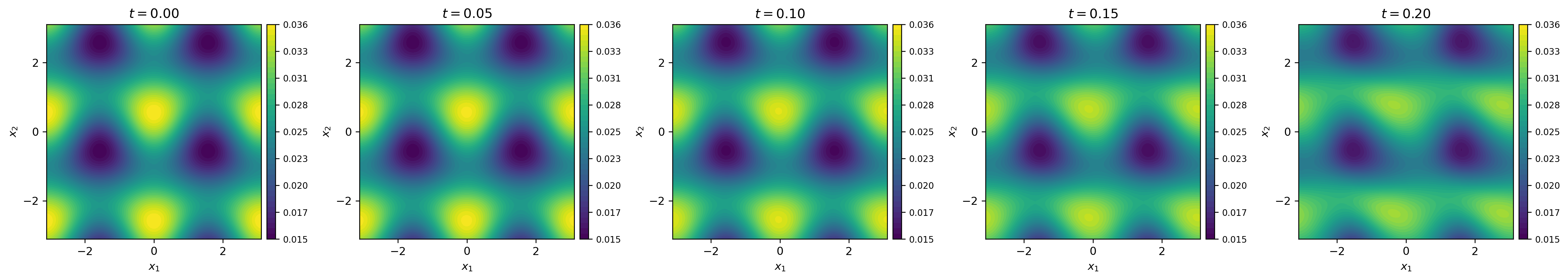}
\caption{Example~4: density snapshots at $t=0,0.05,0.10,0.15,0.20$. The nonstationary trajectory preserves the half-period symmetry in each coordinate.}\label{fig:num-example4-data}
\end{figure}

Figure~\ref{fig:num-example4-recovery} shows that the external potential is accurately recovered, while the interaction estimator approximates $W_{\mathrm{obs}}$ with relative error $1.8543\times10^{-2}$ but misses $\psi$. As reported in Table~\ref{tab:num-results}, the relative error against the complete interaction potential is $E_W=7.8455\times10^{-1}$, despite the small observation residual $R_{\mathrm{fit}}=3.2424\times10^{-4}$. Manually adding $\psi$ reduces the potential error to $1.1500\times10^{-2}$ without changing the equation residual to the reported precision; see the last row of Figure~\ref{fig:num-example4-recovery}. The maximum computed value of $|\psi*\rho_t|$ is $2.532\times10^{-16}$, confirming that this component is invisible to numerical precision. As in Example~2, the small residual indicates accurate fitting of the observations but does not imply recovery of the complete interaction potential.
\begin{figure}[htbp]
\centering
\includegraphics[width=\linewidth]{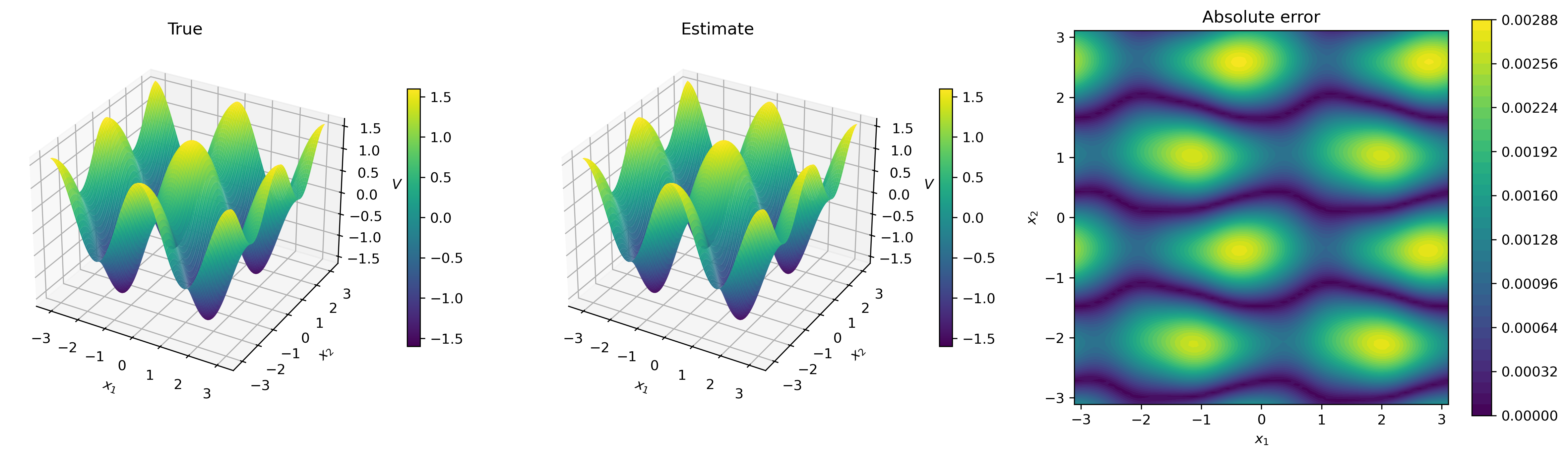}\\
\includegraphics[width=\linewidth]{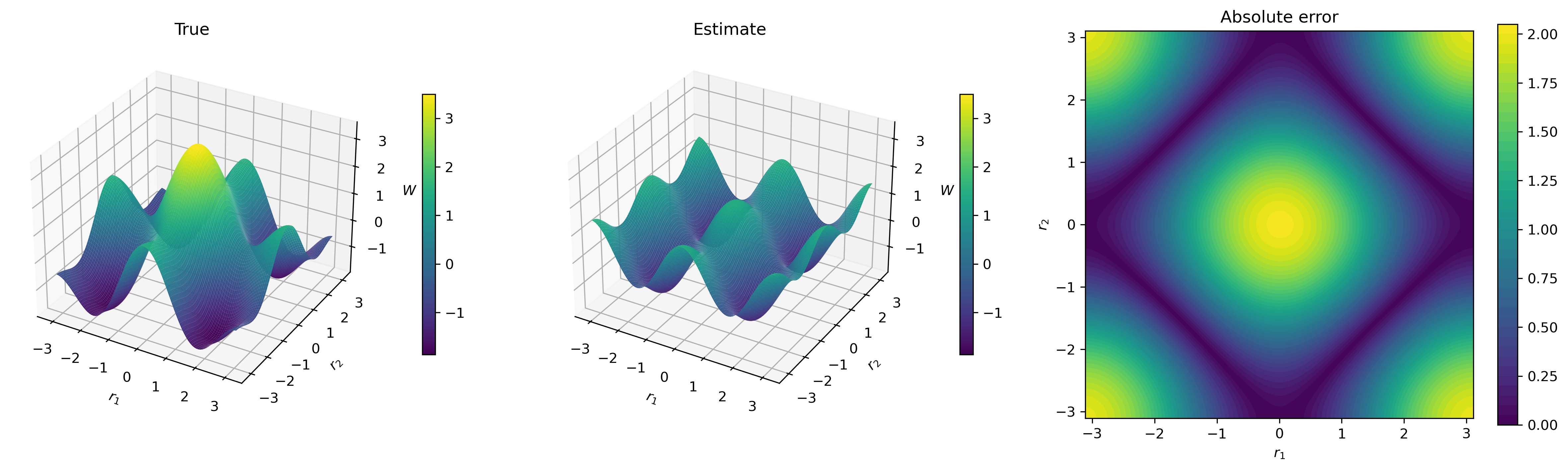}\\
\includegraphics[width=\linewidth]{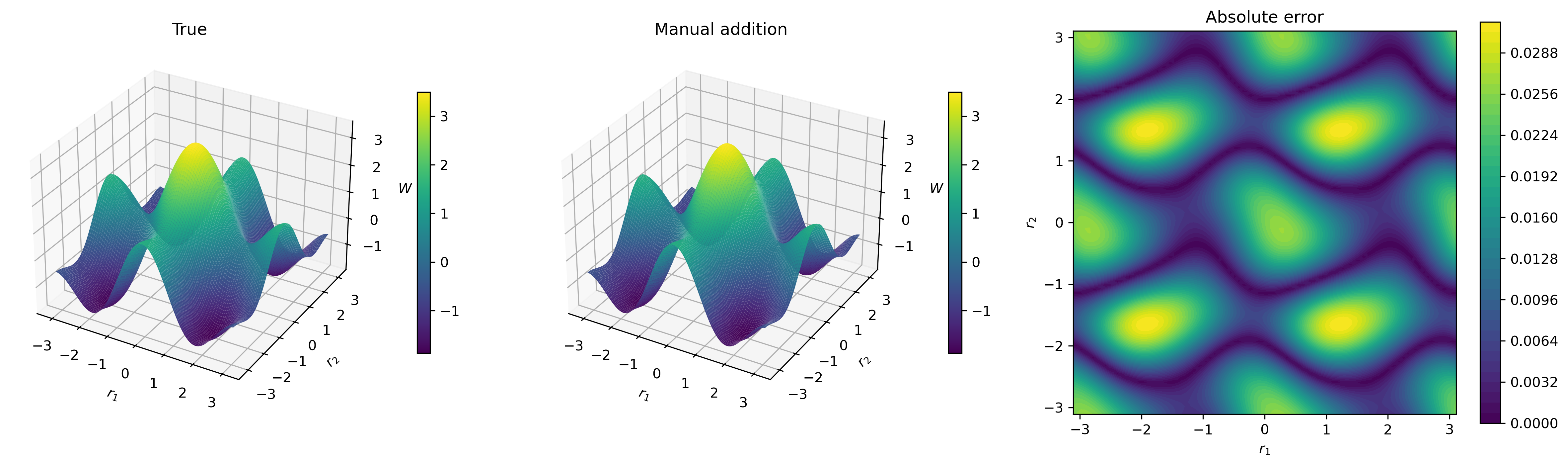}
\caption{Example~4: reconstruction of $V$ and $W$, followed by the interaction estimator with the invisible component $\psi$ manually added.}\label{fig:num-example4-recovery}
\end{figure}

\subsection{Comparison with sparse learning}
We compare the proposed kernel method with the sparse learning method in \cite{carrillo2025sparse} for recovering a two-dimensional highly nonconvex external potential. This benchmark considers the Hamiltonian equations \eqref{eq:first_order_whf} on the nonperiodic domain $[-5,5]^2$, with zero normal density flux at the boundary, known $W=0$, $m=3$, and $\kappa=1$. In this comparison, the internal energy is $U(\rho)=\kappa\rho^m/(m-1)$, so that $U'(\rho)=3\rho^2/2$. We choose
\[
V(x_1,x_2)=\sin\left(\frac{2\pi x_1}{3}\right)\cos\left(\frac{2\pi x_2}{3}\right)+\frac{\sin r}{r},\qquad r=\sqrt{x_1^2+x_2^2},
\]
where $\sin r/r$ is defined as $1$ at $r=0$. The potential combines an oscillatory landscape with a radial component. The initial conditions are
\[
\rho_0(x_1,x_2)\propto\exp\left(-\frac{x_1^2+x_2^2}{8}\right),\qquad \Phi_0(x_1,x_2)=-\frac12x_1^2,
\]
with the density normalized to unit mass on the computational domain $[-5,5]^2$.

\noindent{\bf Observations and training.}\quad
Both $\rho$ and $\Phi$ are observed over $[0,0.2]$, with spatial spacing $\Delta x=5\times10^{-2}$ and temporal spacing $\Delta t=5\times10^{-4}$. Figure~\ref{fig:sparse-comparison-density} displays representative density snapshots.
\begin{figure}[htbp]
\centering
\includegraphics[width=0.19\linewidth]{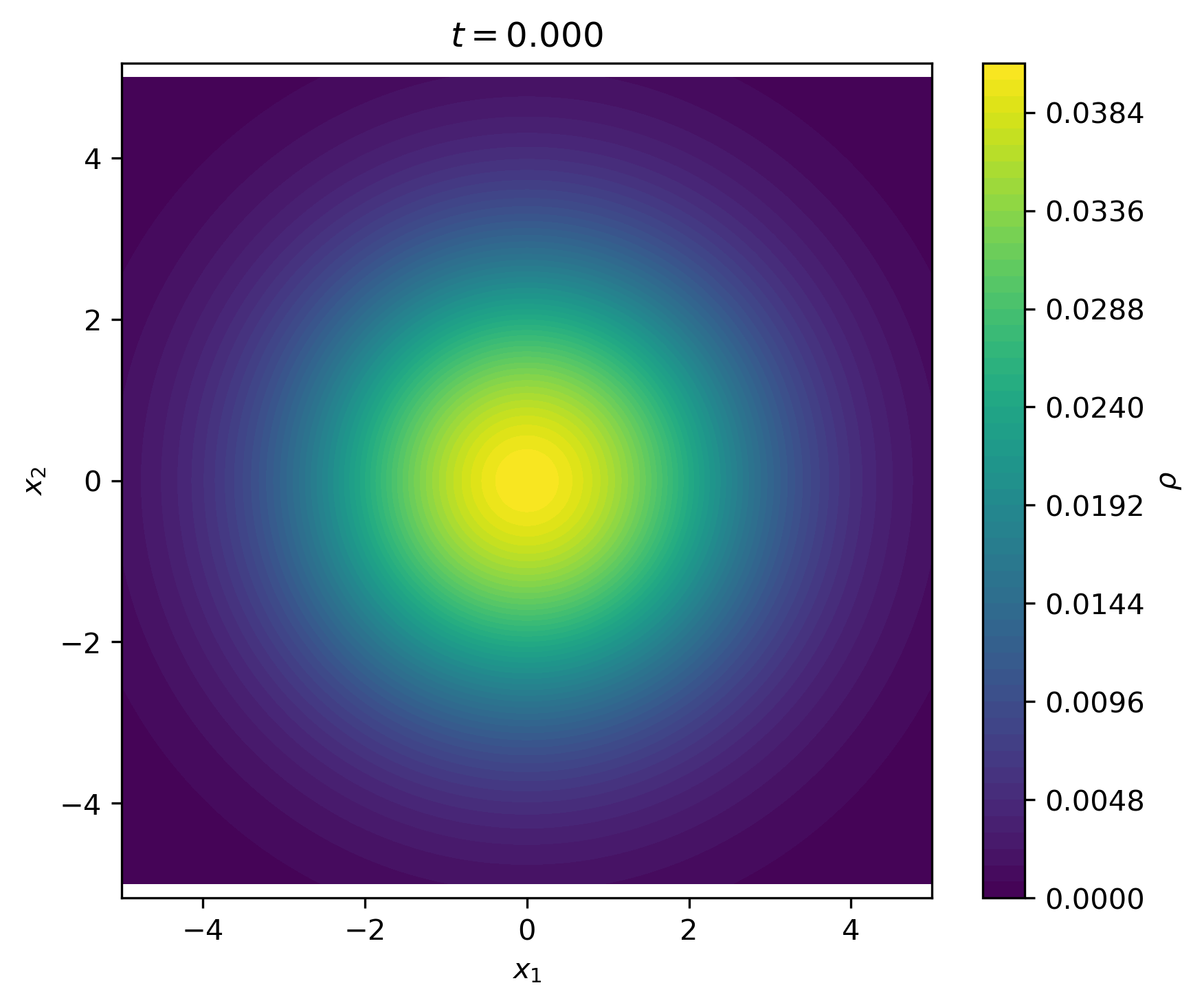}
\includegraphics[width=0.19\linewidth]{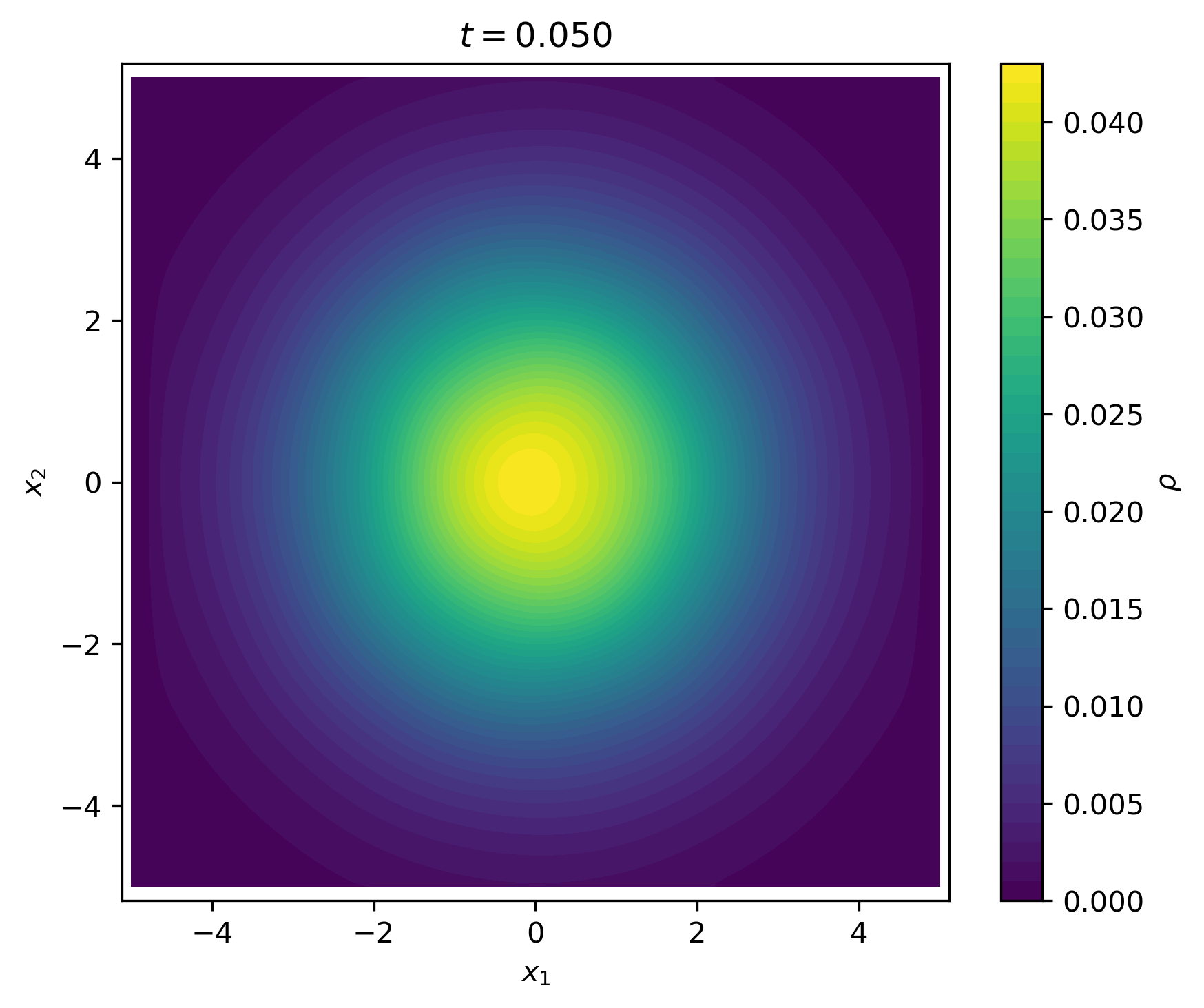}
\includegraphics[width=0.19\linewidth]{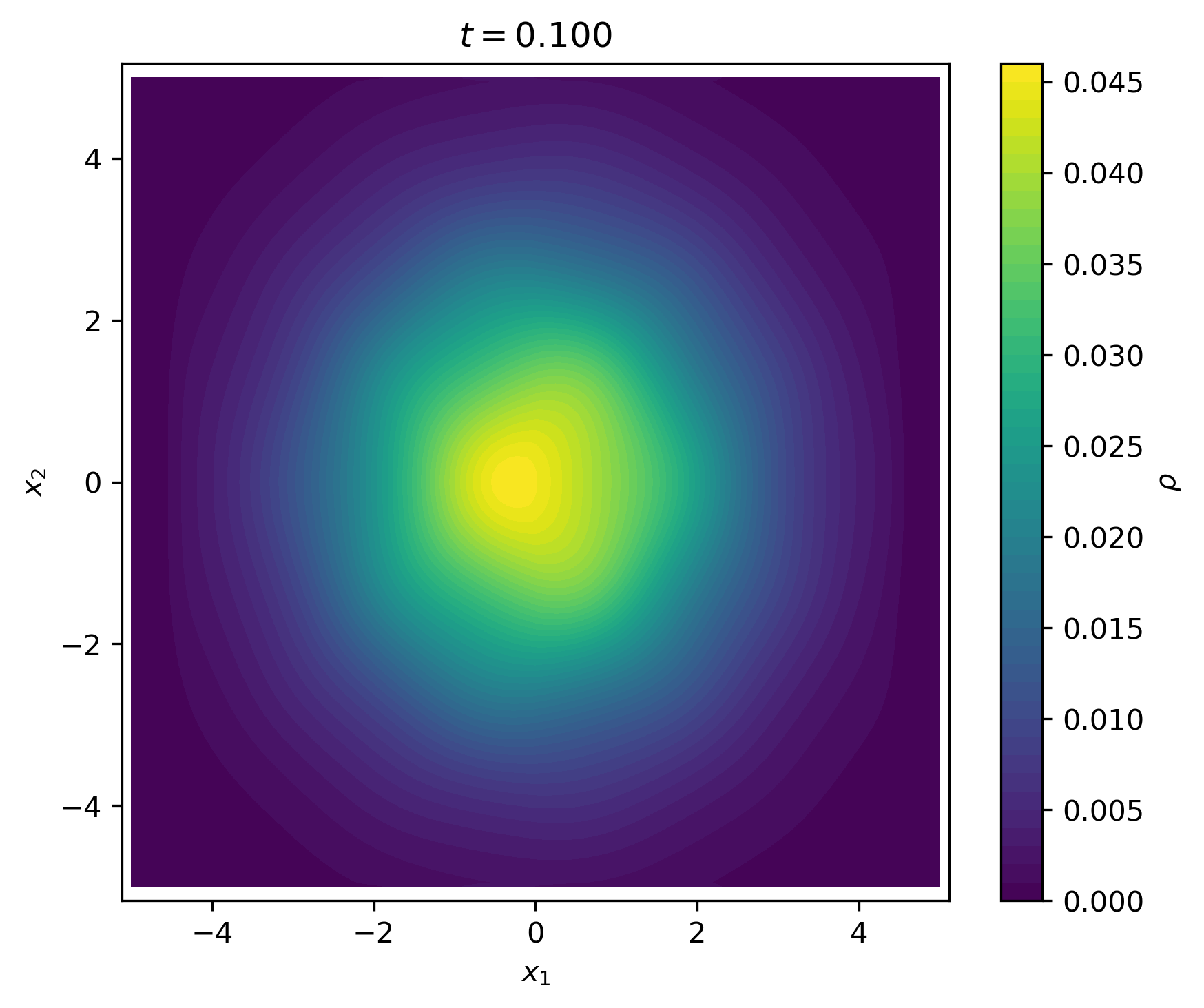}
\includegraphics[width=0.19\linewidth]{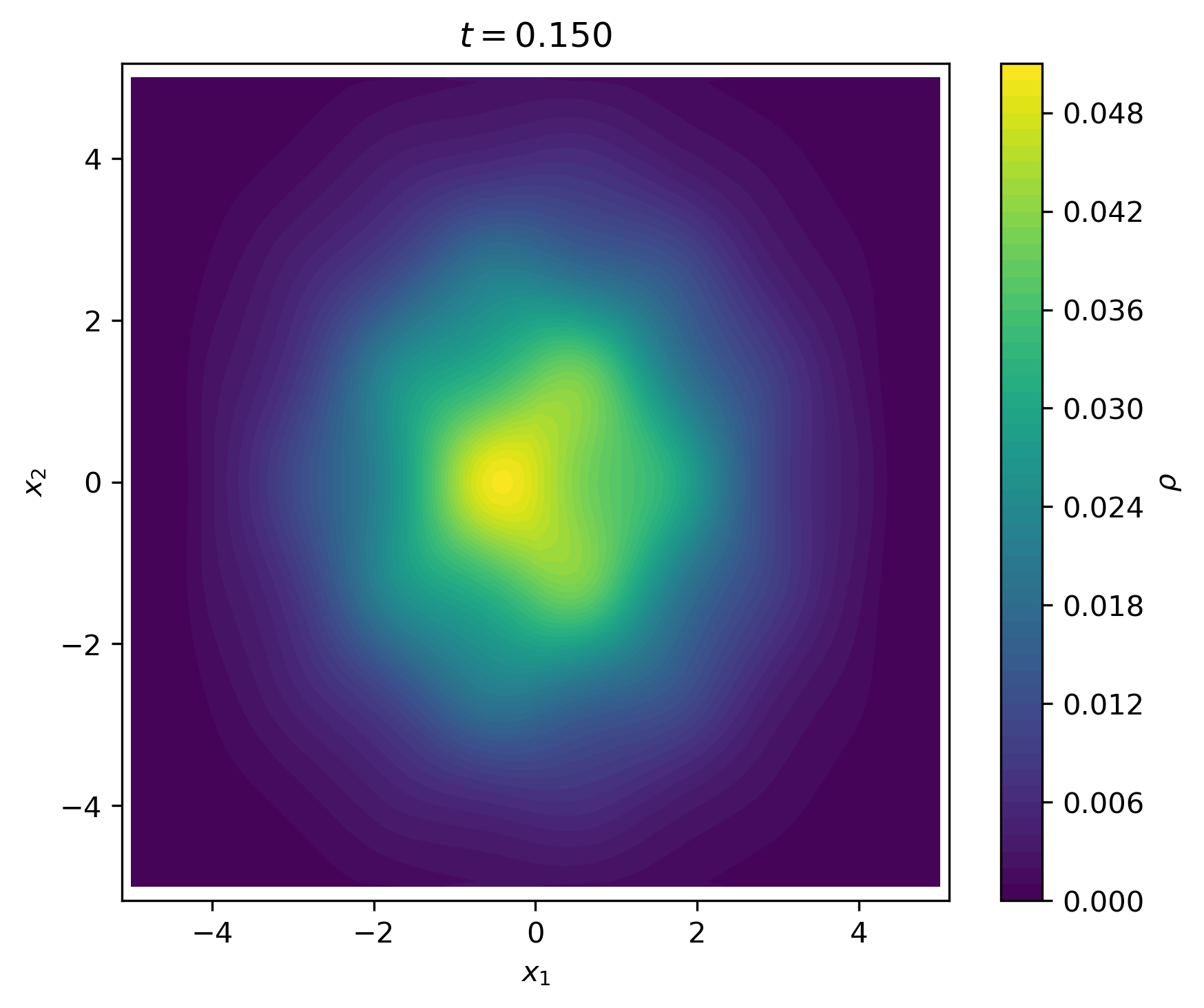}
\includegraphics[width=0.19\linewidth]{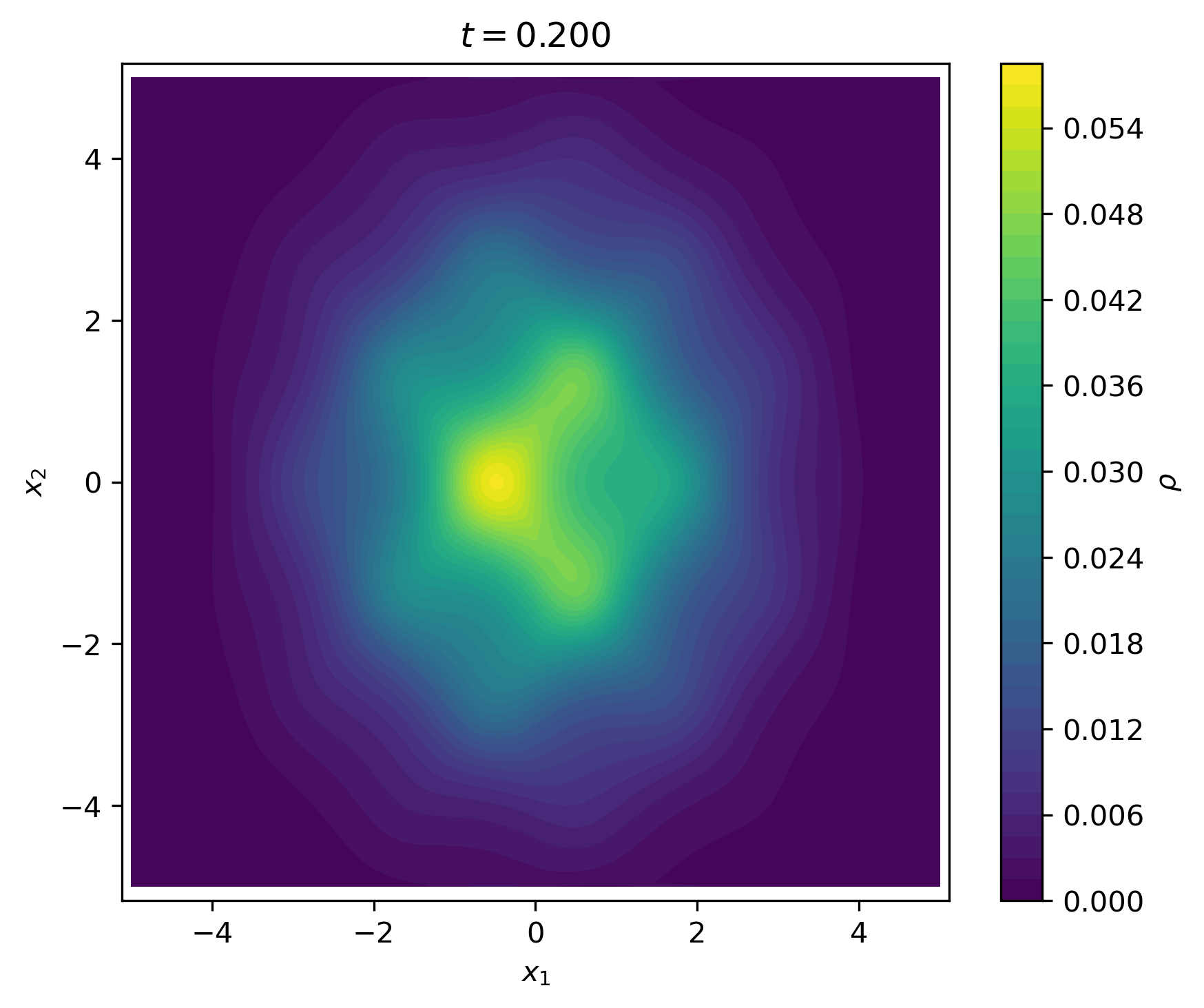}
\caption{Density snapshots for the sparse-learning benchmark at $t=0,0.05,0.10,0.15,0.20$.}\label{fig:sparse-comparison-density}
\end{figure}

The kernel estimator uses $8000$ space--time collocation points, a Gaussian kernel with bandwidth $\eta=0.4$, and regularization parameter $\lambda=10^{-7}$. Observations with density below $10^{-5}$ are excluded. For comparison, we adapt the sparse-learning approach of \cite{carrillo2025sparse} to the same discrete observation equation, using the same observations and forcing $f^\delta$. We consider polynomial functions of total degree at most five, a Fourier dictionary with eight modes per coordinate, and a Gaussian dictionary with $25$ centers uniformly distributed on $[-4,4]^2$. Sparse coefficients are obtained using PartInv support selection and residual-based pruning, followed by restricted least-squares refinement, with four active functions retained. For all methods, the additive ambiguity is handled by density-weighted mean matching.

\noindent{\bf Results.}\quad
Figure~\ref{fig:sparse-comparison-kernel} shows the kernel reconstruction. The estimator captures both the oscillatory and radial components of the external potential. Reconstruction is more accurate in the central region, where the observed density is larger.
\begin{figure}[htbp]
\centering
\includegraphics[width=0.32\linewidth]{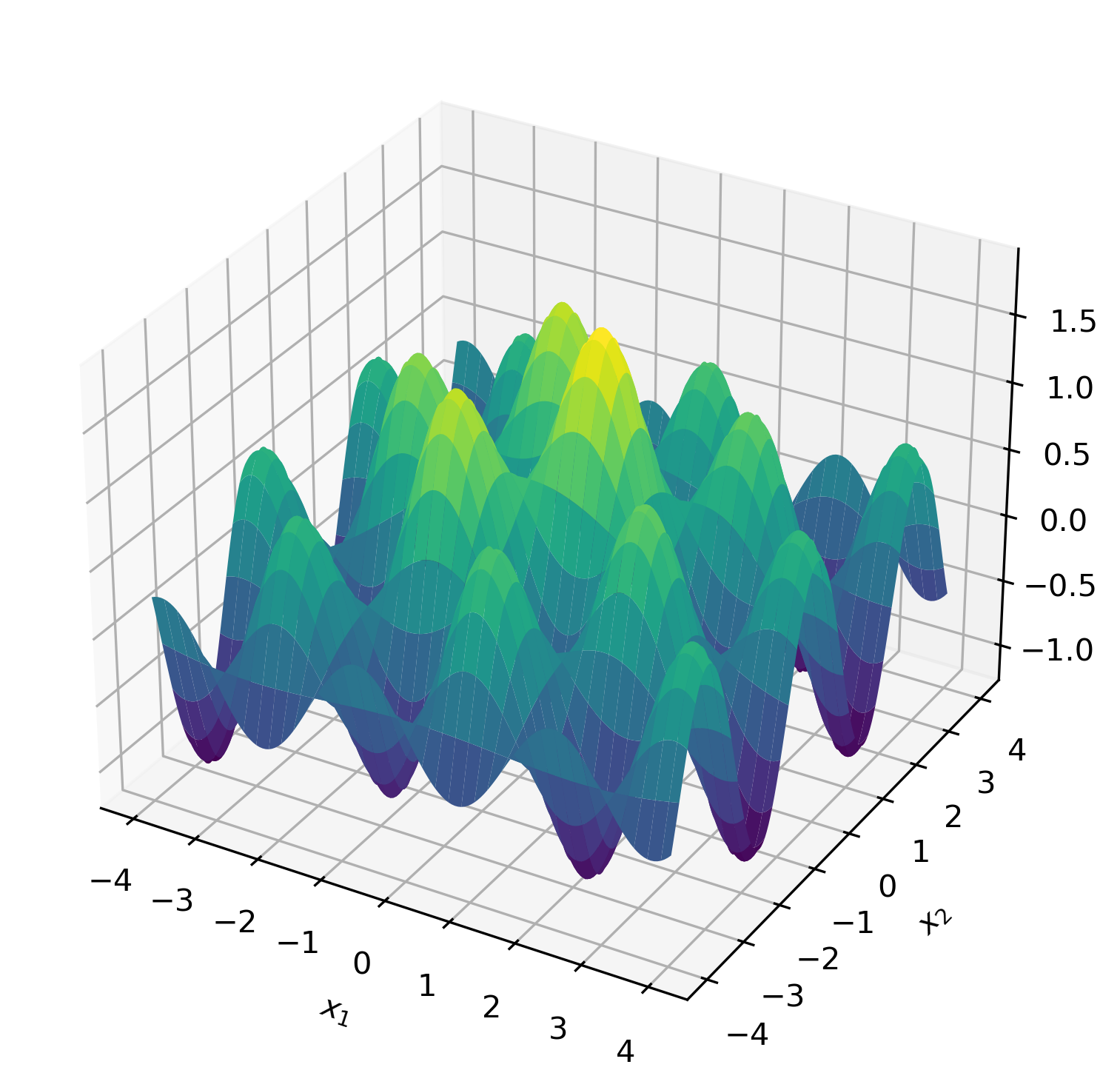}
\includegraphics[width=0.32\linewidth]{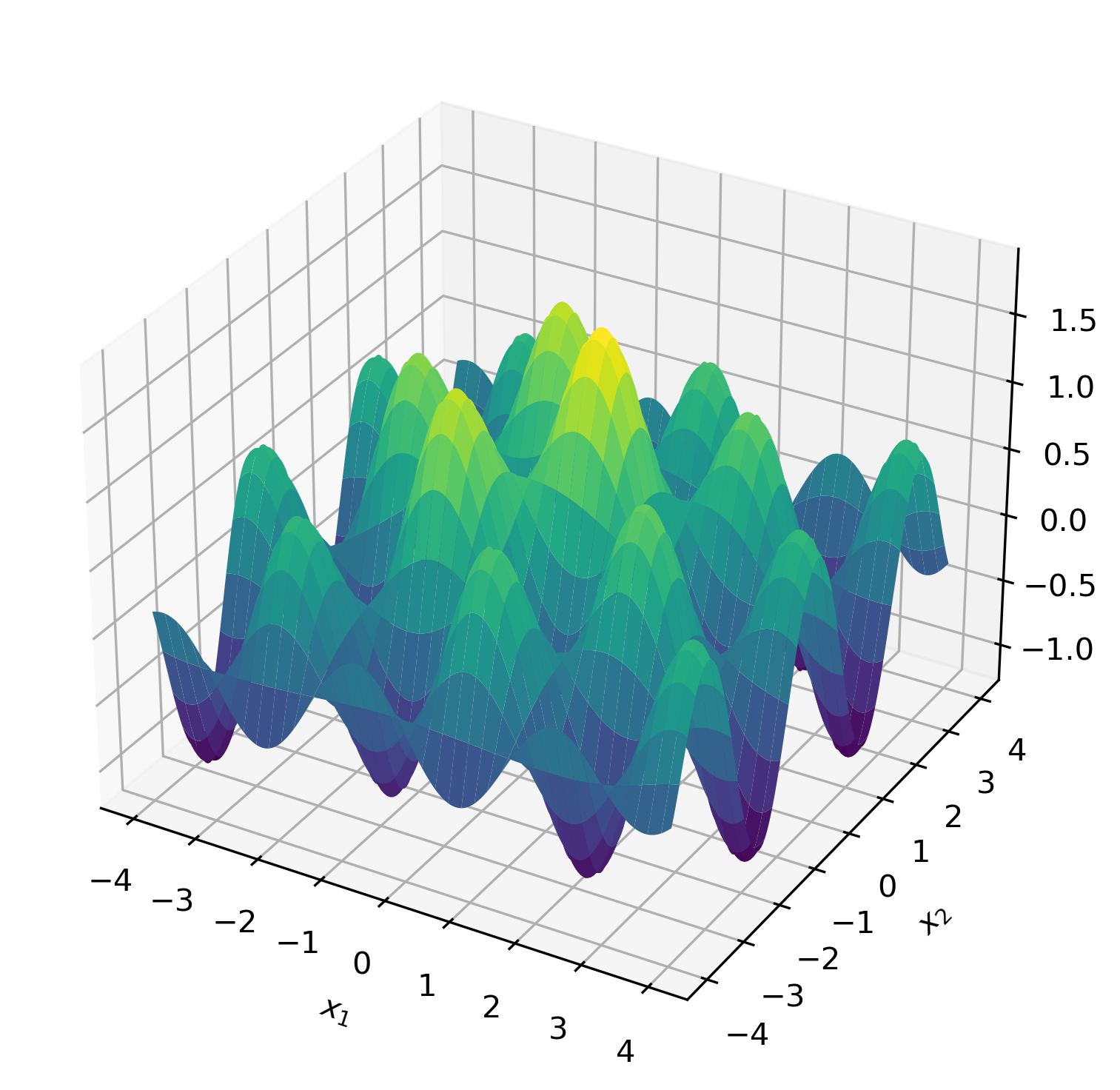}
\includegraphics[width=0.32\linewidth]{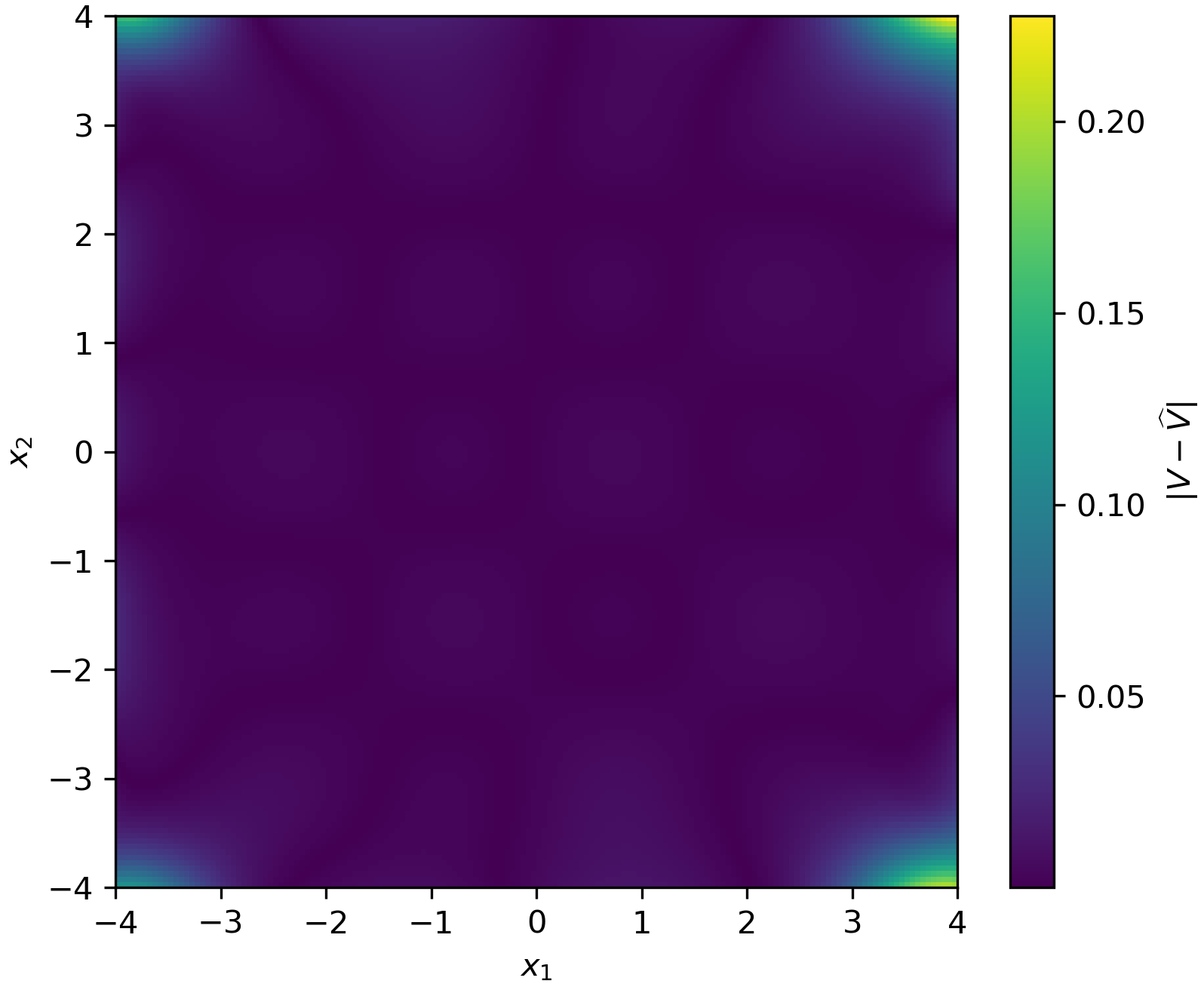}
\caption{Kernel reconstruction for the sparse-learning benchmark. From left to right: the true potential, the estimated potential, and the pointwise absolute error after additive-constant alignment.}\label{fig:sparse-comparison-kernel}
\end{figure}

Figure~\ref{fig:sparse-comparison-errors} shows the sparse-learning errors, and Table~\ref{tab:sparse-comparison} summarizes the relative $L^2$ errors on two evaluation domains. The Fourier dictionary gives the most accurate sparse reconstruction, consistent with the oscillatory structure of the target potential. The polynomial and Gaussian dictionaries produce substantially larger errors. The kernel method achieves the smallest error on both domains for the tested dictionaries and sparsity level. These results illustrate the sensitivity of sparse reconstruction to the choice of dictionary and the accuracy of RKHS regularization for this benchmark.
\begin{figure}[htbp]
\centering
\includegraphics[width=0.32\linewidth]{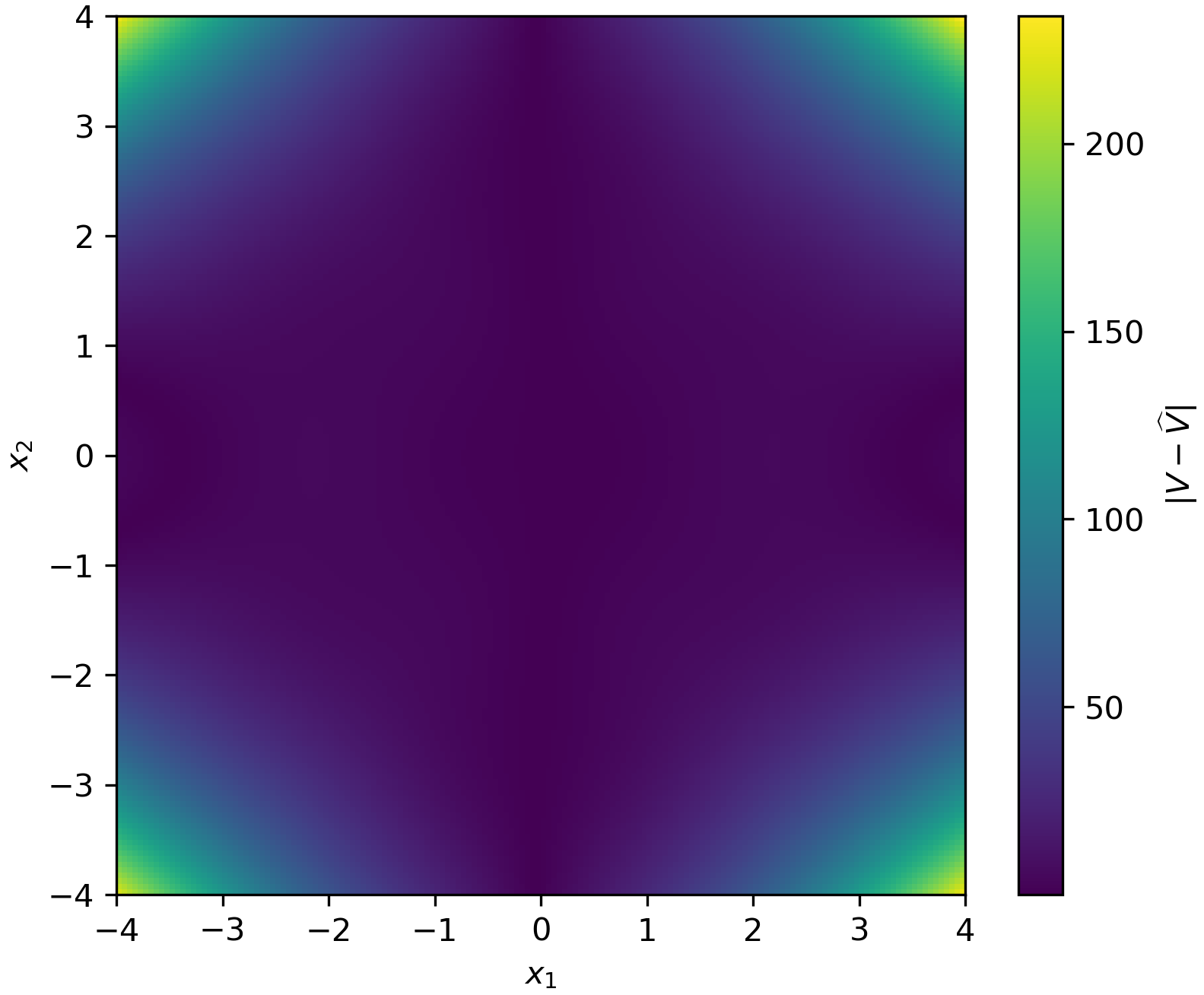}
\includegraphics[width=0.32\linewidth]{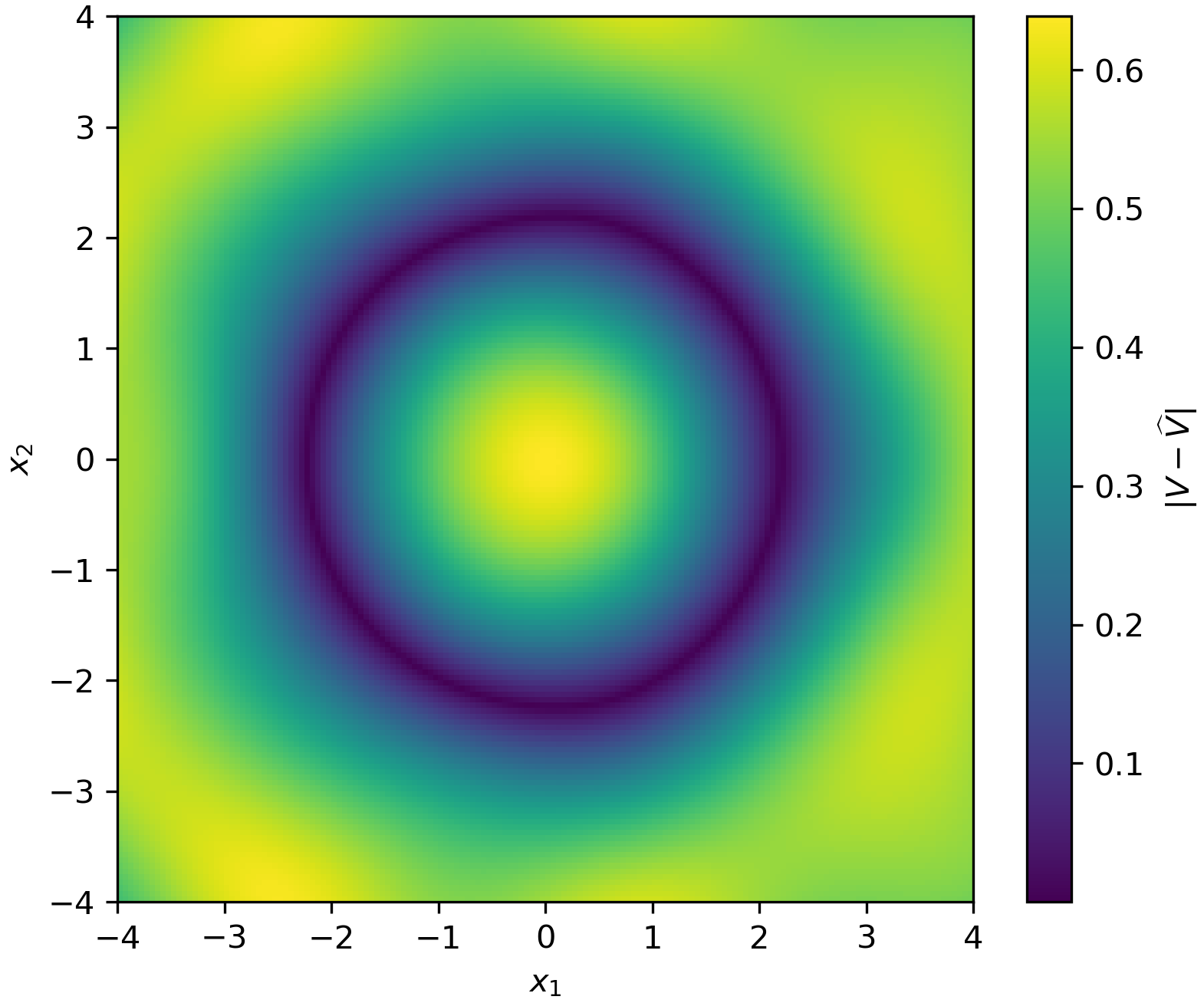}
\includegraphics[width=0.32\linewidth]{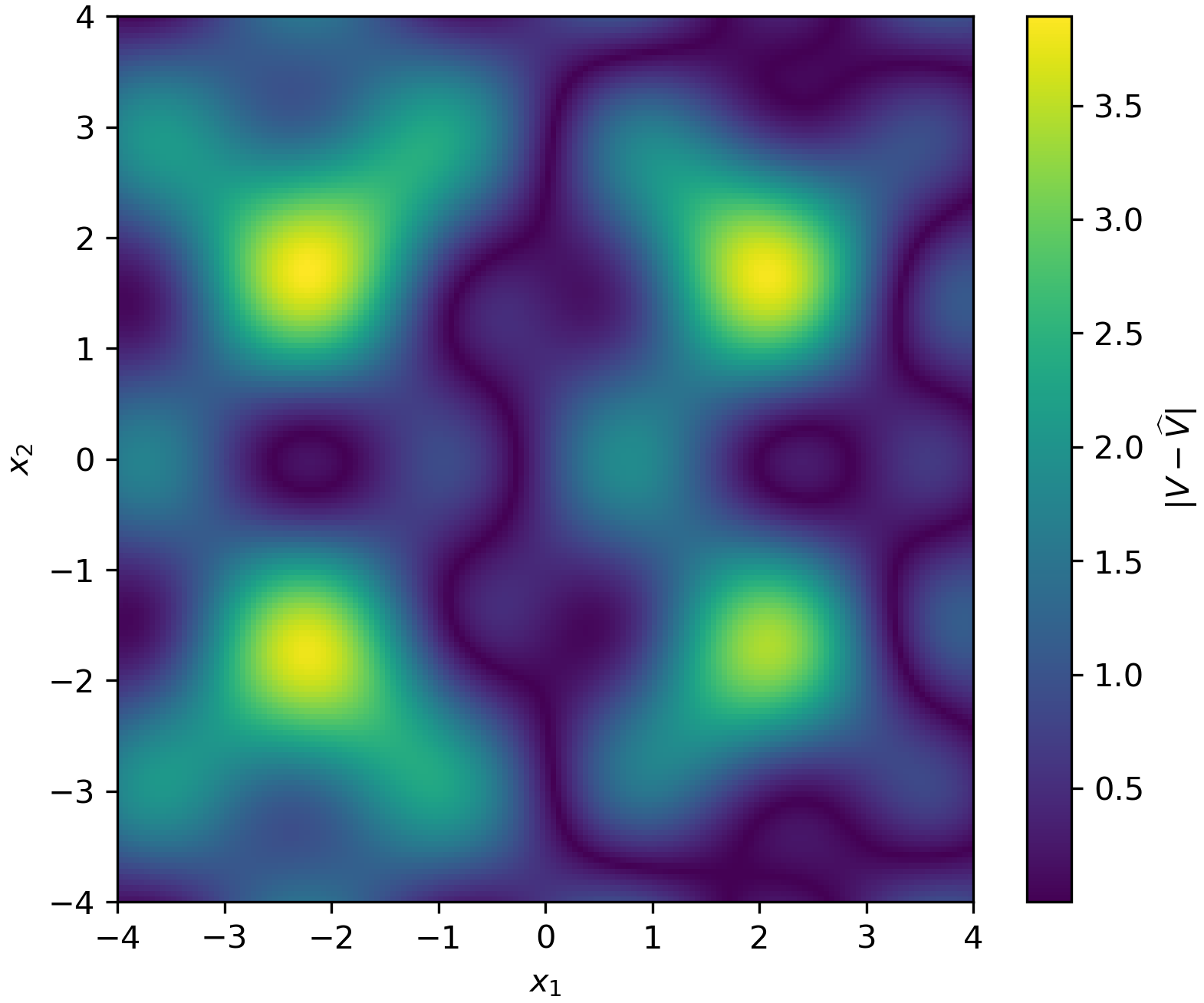}
\caption{Pointwise absolute errors of the sparse reconstructions after additive-constant alignment. From left to right: polynomial, Fourier, and Gaussian dictionaries.}\label{fig:sparse-comparison-errors}
\end{figure}

\begin{table}[htbp]
\centering\small
\caption{Relative $L^2$ potential errors for the sparse-learning benchmark.}\label{tab:sparse-comparison}
\begin{tabular}{c|c|c}
\hline
Method & $[-2,2]^2$ & $[-4,4]^2$\\ \hline
Kernel method & $2.5832\times10^{-3}$ & $2.7633\times10^{-2}$\\
Sparse polynomial & $4.9971$ & $6.7241\times10^{1}$\\
Sparse Gaussian & $1.6895$ & $2.3943$\\
Sparse Fourier & $4.0946\times10^{-1}$ & $6.8272\times10^{-1}$\\ \hline
\end{tabular}
\end{table}

\section{Conclusion}

In this work, we have introduced a structure-preserving, kernel-based learning framework for the inverse identification of energy functionals governing Wasserstein geometric flows. Specifically, our method addresses the challenging problem of recovering both the potential function and the interaction kernel underlying Wasserstein gradient and Hamiltonian systems from discretized observations of density trajectories. By embedding the problem within the setting of reproducing kernel Hilbert space (RKHS), we derived a closed-form representation of the estimators that inherently preserves the geometric structure of the underlying flow equations. This structure-preserving property is critical for ensuring that the learned models remain consistent with the intrinsic variational and symplectic formulations that characterize Wasserstein dynamics.

A key contribution of our analysis lies in establishing rigorous convergence and stability guarantees. Through a detailed decomposition of the reconstruction error into approximation, mesh-induced, and scheme-dependent components, we derived explicit convergence rates under adaptive regularization and discretization schemes. Furthermore, we demonstrated that the Wasserstein Hamiltonian flow induced by the learned energy functional converges uniformly, in the Wasserstein metric, to the true continuous flow. This result provides a solid theoretical foundation for applying our approach to realistic data-driven modeling scenarios where only discrete trajectory data are available.

Our operator-theoretic formulation also offers new insights into the identifiability and regularization of infinite-dimensional inverse problems on density manifolds. {The null space of the observation operator reveals which combinations of external and interaction potentials are indistinguishable from the data. On periodic domains, an explicit Fourier criterion shows that nonstationarity alone does not imply joint identifiability up to additive constants: every nonzero density mode must vary over time. Regularization ensures uniqueness of the estimator, not recovery of components invisible to the observations.} Unlike previous approaches that depend on basis truncation or pre-specified function dictionaries, our method exploits the differential reproducing property of RKHSs to construct estimators in closed form, avoiding ad hoc truncation and ensuring consistency with underlying infinite-dimensional geometric structures. This makes the framework both flexible and computationally tractable, bridging the gap between modern kernel-based learning and geometric analysis on probability spaces.

{The numerical experiments of Section~\ref{Numerical experiments} include successful joint reconstructions in gradient and Hamiltonian Wasserstein systems and nonstationary examples with unobservable interaction modes. The latter yield small observation residuals without recovering the full interaction potential, illustrating the identifiability limitations established by the theory. A separate benchmark with sparse-learning methods underscores the sensitivity of dictionary-based reconstructions to the chosen basis and the accuracy of the kernel method on the tested example.}

{Altogether, this study establishes a principled connection between kernel learning theory and Wasserstein geometry, providing a robust foundation for the next generation of structure-preserving machine learning methods for complex dynamical systems.} Looking ahead, several promising research directions emerge from this work. Extending the current framework to stochastic systems in both finite \cite{lazaro2008stochastic, lazaro:ortega3} and infinite-dimensional \cite{CruzeiroHolmRatiu2018, ChenCruzeiroRatiu2023} frameworks, incorporating higher-order numerical schemes, and exploring adaptive kernel selection strategies could further enhance performance and applicability. Additionally, integrating our method with generative modeling techniques or geometric deep learning architectures may open new possibilities for data-driven discovery of dynamical laws in physics, biology, and beyond. {Another important direction is the development of identifiability-aware observation strategies, including the selection of informative initial conditions and the use of multiple trajectories to reduce the null space of the observation operator and improve the recovery of the underlying energy components.}

\appendix

\section{Detailed proofs of some of the results}

\subsection{Proof of Theorem \ref{thm:symplectic form}}
\begin{proof}
The bilinearity and skew-symmetry of $\omega$ are immediate from its definition. The closedness follows since
\begin{align*}
    d\omega = \int_M d(d\phi \wedge d\rho) = 0.
\end{align*}
To establish non-degeneracy, suppose $\omega\big((\sigma_1, \psi_1), (\sigma_2, \psi_2)\big) = 0$ for all $(\sigma_2, \psi_2) \in T_{\rho}\mathcal{P}_{+}(M) \times C^\infty(M)/\mathbb{R}$. We will show that this implies $\sigma_1 = 0$ and $\psi_1 = 0$.

First, setting $\psi_2 = 0$, we have
\begin{align*}
    \omega(\rho, \phi)\big((\sigma_1, \psi_1), (\sigma_2, 0)\big) = \int_M \psi_1 \sigma_2 \, d\mathrm{vol}_M = 0, \quad \text{for all } \sigma_2 \in T_{\rho}\mathcal{P}_{+}(M).
\end{align*}
Observe that $\langle \psi_1, \sigma_2 \rangle_{L^2} = g_W(\rho)(-\Delta_\rho \psi_1, \sigma_2)$. By the non-degeneracy of Otto’s Riemannian metric $g_W$ and the isomorphism property of the weighted Laplacian $\Delta_\rho$, we conclude that $\psi_1 = 0$. Similarly, by considering $\sigma_2 = 0$ and varying $\psi_2$, we deduce $\sigma_1 = 0$.
This proves non-degeneracy and completes the proof.
\end{proof}

\subsection{Proof of Corollary \ref{Poissons structure}}
\label{Proof of Corollary Poisson}
\begin{proof}
The bilinearity and skew-symmetry follow directly from the definition. To show the Leibniz rule, let  $\mathcal{H},\mathcal{G},\mathcal{F}\in C^\infty\left(\mathcal{P}_+(M)\times C^\infty(M)/\mathbb{R}\right)$. Then we have that 
\begin{align*}
\{\mathcal{F}\mathcal{G},\mathcal{H}\}&= \int_M \left(\frac{\delta (\mathcal{FG})}{\delta\rho} \frac{\delta \mathcal{H}}{\delta\phi} - \frac{\delta \mathcal{H}}{\delta\rho} \frac{\delta (\mathcal{FG})}{\delta\phi} \right)d\mathrm{vol}_M\\
&=\int_M \left(\frac{\delta \mathcal{F}}{\delta\rho}\mathcal{G}+\frac{\delta \mathcal{G}}{\delta\rho}\mathcal{F}\right) \frac{\delta \mathcal{H}}{\delta\phi} - \frac{\delta \mathcal{H}}{\delta\rho} \left(\frac{\delta \mathcal{F}}{\delta\phi}\mathcal{G}+\frac{\delta \mathcal{G}}{\delta\phi}\mathcal{F}\right) d\mathrm{vol}_M\\
&=\int_M \left(\mathcal{F}\left(\frac{\delta \mathcal{G}}{\delta\rho}\frac{\delta \mathcal{H}}{\delta\phi}-\frac{\delta \mathcal{H}}{\delta\rho}\frac{\delta \mathcal{G}}{\delta\phi}\right)  +  \left(\frac{\delta \mathcal{F}}{\delta\rho}\frac{\delta \mathcal{H}}{\delta\phi}-\frac{\delta \mathcal{H}}{\delta\rho}\frac{\delta \mathcal{F}}{\delta\phi}\right)\mathcal{G}\right) d\mathrm{vol}_M\\
&=\mathcal{F}\{\mathcal{G},\mathcal{H}\}+\{\mathcal{F},\mathcal{H}\}\mathcal{G}.
\end{align*}
Finally, we prove the Jacobi identity. 
Denote $\mathcal{H}_{\rho}=\frac{\delta \mathcal{H}}{\delta\rho}$ and $\mathcal{H}_{\phi}=\frac{\delta \mathcal{H}}{\delta\phi}$. Let $\mathcal{E}=\{\mathcal{H},\mathcal{G}\}$. Then for any $\psi\in C^\infty(M)/\mathbb{R}$, we have 
\begin{align*}
\langle\mathcal{E}_{\phi},\psi\rangle_{L^2} :&= \lim_{\varepsilon\to 0}\frac{1}{\varepsilon} \left(\mathcal{E}(\rho,\phi+\varepsilon\psi)-\mathcal{E}(\rho,\phi)\right) \\
&=\lim_{\varepsilon\to 0}\frac{1}{\varepsilon} \int_M  \mathcal{H}_{\rho}(\rho,\phi+\varepsilon\psi)  \mathcal{G}_{\phi}(\rho,\phi+\varepsilon\psi) - \mathcal{H}_{\rho}(\rho,\phi)  \mathcal{G}_{\phi}(\rho,\phi)  d\mathrm{vol}_M\\
&\quad - \lim_{\varepsilon\to 0}\frac{1}{\varepsilon} \int_M  \mathcal{H}_{\phi}(\rho,\phi+\varepsilon\psi)  \mathcal{G}_{\rho}(\rho,\phi+\varepsilon\psi) - \mathcal{H}_{\phi}(\rho,\phi)  \mathcal{G}_{\rho}(\rho,\phi) d\mathrm{vol}_M\\
&=\langle\mathcal{H}_{\rho\phi} \mathcal{G}_{\phi}+\mathcal{H}_{\rho}\mathcal{G}_{\phi\phi}- \mathcal{H}_{\phi\phi} \mathcal{G}_{\rho}-\mathcal{H}_{\phi}\mathcal{G}_{\rho\phi},\psi\rangle_{L^2},
\end{align*}
which implies that $\mathcal{E}_{\phi}=\mathcal{H}_{\rho\phi} \mathcal{G}_{\phi}+\mathcal{H}_{\rho}\mathcal{G}_{\phi\phi}- \mathcal{H}_{\phi\phi} \mathcal{G}_{\rho}-\mathcal{H}_{\phi}\mathcal{G}_{\rho\phi}$. Similarly, we have $\mathcal{E}_{\rho}=\mathcal{H}_{\rho\rho} \mathcal{G}_{\phi}+\mathcal{H}_{\rho}\mathcal{G}_{\phi\rho}- \mathcal{H}_{\phi\rho} \mathcal{G}_{\rho}-\mathcal{H}_{\phi}\mathcal{G}_{\rho\rho}$. Therefore, we obtain 
\begin{align*}
&\{\mathcal{F},\{\mathcal{H},\mathcal{G}\}\}+ \{\mathcal{H},\{\mathcal{G},\mathcal{F}\}\}+\{\mathcal{G},\{\mathcal{F},\mathcal{H}\}\}\\
=\ &\int_M \mathcal{F}_\rho(\mathcal{H}_{\rho\phi} \mathcal{G}_{\phi}+\mathcal{H}_{\rho}\mathcal{G}_{\phi\phi}- \mathcal{H}_{\phi\phi} \mathcal{G}_{\rho}-\mathcal{H}_{\phi}\mathcal{G}_{\rho\phi})-\mathcal{F}_\phi(\mathcal{H}_{\rho\rho} \mathcal{G}_{\phi}+\mathcal{H}_{\rho}\mathcal{G}_{\phi\rho}- \mathcal{H}_{\phi\rho} \mathcal{G}_{\rho}-\mathcal{H}_{\phi}\mathcal{G}_{\rho\rho})d\mathrm{vol}_M\\
+& \int_M \mathcal{H}_\rho(\mathcal{G}_{\rho\phi} \mathcal{F}_{\phi}+\mathcal{G}_{\rho}\mathcal{F}_{\phi\phi}- \mathcal{G}_{\phi\phi} \mathcal{F}_{\rho}-\mathcal{G}_{\phi}\mathcal{F}_{\rho\phi})-\mathcal{H}_\phi(\mathcal{G}_{\rho\rho} \mathcal{F}_{\phi}+\mathcal{G}_{\rho}\mathcal{F}_{\phi\rho}- \mathcal{G}_{\phi\rho} \mathcal{F}_{\rho}-\mathcal{G}_{\phi}\mathcal{F}_{\rho\rho})d\mathrm{vol}_M\\
+& \int_M \mathcal{G}_\rho(\mathcal{F}_{\rho\phi} \mathcal{H}_{\phi}+\mathcal{F}_{\rho}\mathcal{H}_{\phi\phi}- \mathcal{F}_{\phi\phi} \mathcal{H}_{\rho}-\mathcal{F}_{\phi}\mathcal{H}_{\rho\phi})-\mathcal{G}_\phi(\mathcal{F}_{\rho\rho} \mathcal{H}_{\phi}+\mathcal{F}_{\rho}\mathcal{H}_{\phi\rho}- \mathcal{F}_{\phi\rho} \mathcal{H}_{\rho}-\mathcal{F}_{\phi}\mathcal{H}_{\rho\rho})d\mathrm{vol}_M\\
=\ &\int_M(\mathcal{F}_{\rho}\mathcal{G}_{\phi}-\mathcal{F}_{\phi}\mathcal{G}_{\rho})(\mathcal{H}_{\rho\phi}-\mathcal{H}_{\phi\rho})+(\mathcal{G}_{\rho}\mathcal{H}_{\phi}-\mathcal{G}_{\phi}\mathcal{H}_{\rho})(\mathcal{F}_{\rho\phi}-\mathcal{F}_{\phi\rho})+(\mathcal{H}_{\rho}\mathcal{F}_{\phi}-\mathcal{H}_{\phi}\mathcal{F}_{\rho})(\mathcal{G}_{\rho\phi}-\mathcal{G}_{\phi\rho})d\mathrm{vol}_M\\
=\ &0.
\end{align*}
This completes the proof.
\end{proof}

\subsection{Proof of Proposition \ref{Bound-A}} \label{Proof of Proposition Bound-A}
\begin{proof}
For each $t\in [0,T]$, we have that
\begin{align*}
|\Delta_{\rho_t}\phi(x)| = |\nabla\cdot (\rho_t(x)\nabla \phi(x))|\leq \|\nabla\rho_t(x)\|\|\nabla \phi(x)\|+\rho_t(x)|\Delta \phi(x)|,
\end{align*}
which leads to 
\begin{align}\label{ineq-hkn}
\|\Delta_{\rho_t}\phi\|_{\infty}\leq \|\rho_t\|_{C_b^1} \|\phi\|_{C_b^2} \leq  \kappa_1\|\rho_t\|_{C_b^1} \|\phi\|_{\mathcal{H}_{K_1}}
\end{align}
where $\kappa_1^2=2\|K_1\|_{C_b^4}$ and the last equality is due to Theorem 2.7 in \cite{RCSP2}. Similarly, we have that 
\begin{align}\label{con-ineq}
\|\Delta_{\rho_t}\psi*\rho_t\|_{\infty} \leq  \kappa_2\|\rho_t\|_{C_b^1} \|\psi\|_{\mathcal{H}_{K_2}},
\end{align}
where $\kappa_2^2=2\|K_2\|_{C_b^4}$. Hence combining the inequalies \eqref{ineq-hkn} and \eqref{con-ineq}, the operator $A$ is bounded linear because
\begin{align*}
\|A(\phi,\psi)\|^2_{L^2}&= \int_0^T\int_{\mathbb{R}^d}|\Delta_{\rho_t}(\phi+\psi\ast \rho_t)(x)|^2\rho_t(x)\mathrm dx \mathrm dt\\
&\leq \int_0^T\int_{\mathbb{R}^d}2\left(\kappa_1^2 \|\phi\|^2_{\mathcal{H}_{K_1}}+\kappa^2_2\|\psi\|^2_{\mathcal{H}_{K_2}}\right)\|\rho_t\|^2_{C_b^1}\rho_t(x)\mathrm dx \mathrm dt\\
&\leq 2T C_T \left(\kappa_1^2 \|\phi\|^2_{\mathcal{H}_{K_1}}+\kappa^2_2\|\psi\|^2_{\mathcal{H}_{K_2}}\right)<\infty.
\end{align*}
Then, by Proposition \ref{rep-pro}, for any $g\in L^2\left([0,T]\times\mathbb{R}^d,\{\rho_t\}_{t\in[0,T]}\right)$ we have
\begin{align*}
\langle A^*g, (\phi,\psi)\rangle_{\mathcal{H}_{K_1}\times \mathcal{H}_{K_2}}&= \langle g, A(\phi,\psi)\rangle_{L^2}= \int_0^T\int_{\mathbb{R}^d}g(t,x)\Delta_{\rho_t}(\phi+\psi\ast \rho_t)(x)\rho_t(x)\mathrm dx \mathrm dt \\
&=\int_0^T\int_{\mathbb{R}^d}g(t,x)\left(\langle \phi, \Delta_{\rho_t}K_1(x,\cdot)\rangle_{\mathcal{H}_{K_1}}+\langle \psi, \Delta_{\rho_t}(K*\rho_t)(x,\cdot)\rangle_{\mathcal{H}_{K_2}}\right)\rho_t(x)\mathrm dx \mathrm dt\\
&=\left\langle\int_0^T\int_{\mathbb{R}^d} g(t,x)\begin{bmatrix}
    \Delta_{\rho_t}K_1(x,\cdot)\\
    \Delta_{\rho_t}(K_2\ast \rho_t)(x,\cdot)
    \end{bmatrix}^\top\rho_t(x)\mathrm dx \mathrm dt,(\phi,\psi)\right\rangle_{\mathcal{H}_{K_1}\times \mathcal{H}_{K_2}},
\end{align*}
which implies that $A^{\ast}:L^2\left([0,T]\times\mathbb{R}^d,\{\rho_t\}_{t\in[0,T]}\right)\rightarrow\mathcal{H}_{K_1}\times \mathcal{H}_{K_2}$ is
\begin{equation}\label{adjoint}
    A^{\ast}g = \int_0^T\int_{\mathbb{R}^d} g(t,x)\begin{bmatrix}
    \Delta_{\rho_t}K_1(x,\cdot)\\
    \Delta_{\rho_t}(K_2\ast \rho_t)(x,\cdot)
    \end{bmatrix}^\top\rho_t(x)\mathrm dx \mathrm dt.
\end{equation}

Since $B=A^{\ast}A$, $B$ is clearly a bounded linear operator. Equation \eqref{operatorB} follows from \eqref{adjoint} by direct calculation and the fact that the integral commutes with the scalar product. Indeed, for any $(\phi,\psi)\in\mathcal{H}_{K_1}\times\mathcal{H}_{K_2}$,  
\begin{align*}
B(\phi,\psi) &= A^{\ast}A (\phi,\psi)= A^{\ast}(\Delta_{\rho_t}(\phi+\psi\ast \rho_t))\\
&=\int_0^T\int_{\mathbb{R}^d} \Delta_{\rho_t}(\phi+\psi\ast\rho_t)(x)\begin{bmatrix}
    \Delta_{\rho_t} K_1(x,\cdot)\\
    \Delta_{\rho_t}(K_2\ast \rho_t)(x,\cdot)
    \end{bmatrix}^\top\rho_t(x)\mathrm dx \mathrm dt.
\end{align*}

We now prove that $B$ is a trace class operator, that is, we show that $\operatorname{Tr}(|B|)<\infty$, where $|B|=\sqrt{B^* B}$. Since $B$ is positive semidefinite, we have that $|B|=B$. Therefore, it is equivalent to show that $\operatorname{Tr}(B)<\infty$. In order to do that, we choose a spanning orthonormal set $\left\{e _n\right\} _{n \in \mathbb{N}}$ and $\left\{f _n\right\} _{n \in \mathbb{N}}$ for ${\mathcal H} _{K_1}$ and ${\mathcal H} _{K_2}$ respectively whose existence are guaranteed by \cite[Lemma A.3]{RCSP2} and \cite[Theorem 2.4]{owhadi2017separability}. Then,
$$
\begin{aligned}
\operatorname{Tr}(B)&=\operatorname{Tr}\left(A^* A\right) =\sum_n\left\langle A^* A (e_n,f_n), (e_n,f_n)\right\rangle_{{\mathcal{H}_{K_1}}\times{\mathcal{H}_{K_2}}}\\
&=\sum_n\left\langle A (e_n,f_n), A (e_n,f_n)\right\rangle_{L^2\left(\mu_{\mathbf{X}}\right)} \\
& =\sum_n\int_0^T\int_{\mathbb{R}^{d}} \left |\Delta_{\rho_t}(e_n+f_n\ast\rho_t)(x)\right |^2\rho_t(x)\mathrm{d}x\mathrm{d}t\\
& \leq 2\int_0^T\int_{\mathbb{R}^{d}} \Delta_{\rho_t}^{(1,2)}\left(K_1+K_2*\rho_t\right)(x,x)\rho_t(x)\mathrm{d}x\mathrm{d}t\\
&\leq 4T C_T(\kappa_1^2+\kappa_2^2),
\end{aligned}
$$
where the fifth inequality is due to the fact that
\begin{align*}
&\sum_n \left |\Delta_{\rho_t}(e_n+f_n\ast\rho_t)(x)\right |^2\leq
2\sum_n \left(|\Delta_{\rho_t}e_n(x)|^2+|\Delta_{\rho_t}(f_n\ast\rho_t)(x)|^2\right)\\
&= 2\sum_n \left(|\langle e_n, \Delta_{\rho_t}K_1(x,\cdot)\rangle_{\mathcal{H}_{K_1}}|^2+|\langle f_n, \Delta_{\rho_t}(K_2*\rho_t)(x,\cdot)\rangle_{\mathcal{H}_{K_2}}|^2\right)\\
&= 2 \sum_n\left(\langle \langle e_n, \Delta_{\rho_t}K_1(x,\cdot)\rangle_{\mathcal{H}_{K_1}}e_n, \Delta_{\rho_t}K_1(x,\cdot)\right\rangle_{\mathcal{H}_{K_1}}+\left\langle \langle f_n, \Delta_{\rho_t}(K_2*\rho_t)(x,\cdot)\rangle_{\mathcal{H}_{K_2}}f_n, \Delta_{\rho_t}(K_2*\rho_t)(x,\cdot)\rangle_{\mathcal{H}_{K_2}}\right)\\
&=2 \left\langle  \Delta_{\rho_t}K_1(x,\cdot), \Delta_{\rho_t}K_1(x,\cdot)\right\rangle_{\mathcal{H}_{K_1}}+2\left\langle \Delta_{\rho_t}(K_2*\rho_t)(x,\cdot), \Delta_{\rho_t}(K_2*\rho_t)(x,\cdot)\right\rangle_{\mathcal{H}_{K_2}} \\
&=2\Delta_{\rho_t}^{(1,2)}(K_1+K_2*\rho_t)(x,x).
\end{align*}
Finally, the form of the operator $B=A^{*}A$ automatically guarantees that it is positive semidefinite. 
\end{proof}

\subsection{Proof of Proposition \ref{Operator_representation}}
\label{proof prop with the solution}
\begin{proof}
To derive these results, we first compute the G\^ateaux derivative of the functional \( R_{\lambda} \) in the space \(\mathcal{H}_{K_1}\times \mathcal{H}_{K_2}\). For any \((\phi,\psi), (f,g)\in \mathcal{H}_{K_1}\times \mathcal{H}_{K_2}\), we have
\begin{align*}
    \mathrm{d}R_{\lambda}(\phi,\psi)\cdot (f,g)
    &= \lim_{t\to 0}\frac{R_{\lambda}(\phi + t f,\psi + t g)-R_{\lambda}(\phi,\psi)}{t}\\[6pt]
    &= 2\big\langle A(f,g),\,A(\phi,\psi)-A(V,W)\big\rangle_{L^2} 
    + 2\lambda_1\langle \phi,f\rangle_{\mathcal{H}_{K_1}}
    + 2\lambda_2\langle \psi,g\rangle_{\mathcal{H}_{K_2}} \\[6pt]
    &= 2\big\langle (f,g),\,(B+\lambda I)(\phi,\psi)-B(V,W)\big\rangle_{\mathcal{H}_{K_1}\times \mathcal{H}_{K_2}}.
\end{align*}
Hence, the critical points of the functional \( R_{\lambda} \) are characterized by the equation
\begin{equation}\label{Cri-eqn}
    \big\langle (f,g),\,(B+\lambda I)(\phi,\psi)-B(V,W)\big\rangle_{\mathcal{H}_{K_1}\times \mathcal{H}_{K_2}}=0,\quad\text{for all }(f,g)\in \mathcal{H}_{K_1}\times \mathcal{H}_{K_2}.
\end{equation}
Since the operator \( B \) is positive semidefinite and of trace class, it follows that the critical equation~\eqref{Cri-eqn} admits a unique solution given explicitly by
\begin{equation*}
    (V_{\lambda}^*, W_{\lambda}^*)=(B+\lambda I)^{-1}B(V,W).
\end{equation*}  
Similarly, $
\left(\widehat{V}_{\lambda,NL},\,\widehat{W}_{\lambda,NL}\right)= \sqrt{\frac{T|\Omega|}{NL}}\left(B_{NL}^{\delta} + \lambda I\right)^{-1} A_{NL}^{\delta*}\,f_{NL}^{\delta}$ is the unique minimizer of the regularized empirical loss functional \eqref{emp-fun1}.
\end{proof}

\subsection{Proof of Lemma \ref{analyze-u-w}}\label{Proof of Lemma 3.3}
\begin{proof}
By the differential reproducing property \cite{zhou2008derivative} or \cite[Theorem 2.7]{RCSP2}, the smoothness condition $K_1, K_2 \in C^6(\Omega \times \Omega)$ implies that $\mathcal{H}_{K_1}, \mathcal{H}_{K_2} \subset C^3(\Omega)$.  
Then, the regularity condition $\rho \in C^{1,2}([0,T] \times \Omega)$ yields that $u \in C^{1,1}([0,T] \times \Omega)$ and $w\in C^{1,1,1,1}([0,T] \times \Omega\times [0,T] \times \Omega)$.
From the definition \eqref{fun-u}, it follows that:
\begin{equation*}
\left\{
\begin{aligned}
&\|u\|_{\infty} 
\leq \|\rho\|_{C^{0,1}}(\|\phi\|_{C^2} + \|\psi\|_{C^2}), \\
&\|\nabla_t u\|_{\infty} \leq \|\rho\|_{C^{1,1}}\big(2\|\phi\|_{C^2} + (1 + |\Omega|\|\nabla_t \rho\|_{\infty})\|\psi\|_{C^2}\big), \\
&\|\nabla_x u\|_{\infty} \leq \|\rho\|_{C^{0,2}}(\|\phi\|_{C^3} + \|\psi\|_{C^3}).
\end{aligned}
\right.
\end{equation*}

Hence, we obtain:
\begin{align*}
\|u\|_{C^{1,1}} &\leq \|\rho\|_{C^{1,2}} \big(4\|\phi\|_{C^3} + (3 + |\Omega|\|\nabla_t \rho\|_{\infty})\|\psi\|_{C^3}\big) \\
&\leq (4 + |\Omega|\|\nabla_t \rho\|_{\infty})(\kappa_1 + \kappa_2)\|\rho\|_{C^{1,2}} (\|\phi\|_{\mathcal{H}_{K_1}} + \|\psi\|_{\mathcal{H}_{K_2}}),
\end{align*}
where $\kappa_1 = \sqrt{2\|K_1\|_{C^6}}$ and $\kappa_2 = \sqrt{2\|K_2\|_{C^6}}$.

Similarly, from the definition \eqref{fun-w}, we have:
\begin{equation*}
\left\{
\begin{aligned}
&\|w\|_{\infty} 
\leq \|\rho\|^2_{C^{0,1}}(\|K_1\|_{C^4} + \|K_2\|_{C^4}), \\
&\|\nabla_t w\|_{\infty}, \|\nabla_s w\|_{\infty} 
\leq \|\rho\|^2_{C^{1,1}}\big(2\|K_1\|_{C^4} + (1 + |\Omega|\|\nabla_t \rho\|_{\infty})\|K_2\|_{C^4}\big), \\
&\|\nabla_x w\|_{\infty}, \|\nabla_y w\|_{\infty} 
\leq \|\rho\|^2_{C^{0,2}}\big(\|K_1\|_{C^6} + \|K_2\|_{C^6}\big).
\end{aligned}
\right.
\end{equation*}

Thus, we obtain:
\begin{align*}
\|w\|_{C^{1,1,1,1}} &\leq \|\rho\|_{C^{1,2}}^2 \big(7\|K_1\|_{C^6} + (5 + 2|\Omega|\|\nabla_t \rho\|_{\infty})\|K_2\|_{C^6}\big) \\
&\leq (4 + |\Omega|\|\nabla_t \rho\|_{\infty})(\kappa_1^2 + \kappa_2^2)\|\rho\|_{C^{1,2}}^2.
\end{align*}
\end{proof}

\subsection{Proof of Lemma \ref{pre-estimator}}\label{Proof of Lemma 3.5}
\begin{proof}
Firstly, by \ref{numerical-integrations}, we have for each $(s, y) \in [0,T] \times \Omega$:
\begin{align*}
\left|I(u,w(\cdot,\cdot;s,y),\rho)-I_{NL}(u,w(\cdot,\cdot;s,y),\rho) \right| \leq 2T|\Omega|\widetilde{C}(u,w(\cdot,\cdot;s,y),\rho)(\Delta x + \Delta t),    
\end{align*}
where $\widetilde{C}$ is defined in \eqref{cofficient-C}.
Hence, by Lemma \ref{analyze-u-w}, we obtain
\begin{align*}
J_1:&= \left\langle B(\phi,\psi),B(\phi,\psi)-B_{NL}(\phi,\psi)\right\rangle_{\mathcal{H}_{K_1}\times\mathcal{H}_{K_2}}\\
&=\int_0^T\int_{\Omega} u(s,y)\rho_s(y)\left(I(u,w(\cdot,\cdot;s,y),\rho)-I_{NL}(u,w(\cdot,\cdot;s,y),\rho)\right)\mathrm{d}y\mathrm{d}s\\
&\leq \frac{1}{2}C^2_1\|(\phi,\psi)\|^2_{\mathcal{H}_{K_1}\times\mathcal{H}_{K_2}}(\Delta x + \Delta t),
\end{align*}
where $C_1^2=4T^2|\Omega|^2(\kappa_1+\kappa_2)^3(4+|\Omega|\|\nabla_t\rho\|_{\infty})^2\|\rho\|^6_{C^{1,2}}$.
Similarly,
\begin{align*}
J_2:&= \left\langle B_{NL}(\phi,\psi),B(\phi,\psi)-B_{NL}(\phi,\psi)\right\rangle_{\mathcal{H}_{K_1}\times\mathcal{H}_{K_2}}\\
&=\sum_{n,l=1}^{NL} u(t_l,x_n)\rho_{t_l}(x_n)\left(I(u,w(\cdot,\cdot;t_l,x_n),\rho)-I_{NL}(u,w(\cdot,\cdot;t_l,x_n),\rho)\right)\Delta x\Delta t\\
&\leq \frac{1}{2}C^2_1\|(\phi,\psi)\|^2_{\mathcal{H}_{K_1}\times\mathcal{H}_{K_2}}(\Delta x + \Delta t).
\end{align*}
Therefore,
\begin{align*}
\|B(\phi,\psi)-B_{NL}(\phi,\psi)\|^2_{\mathcal{H}_{K_1}\times\mathcal{H}_{K_2}} =J_1-J_2\leq  C_1^2\|(\phi,\psi)\|^2_{\mathcal{H}_{K_1}\times\mathcal{H}_{K_2}}(\Delta x + \Delta t).  
\end{align*}
\end{proof}

\subsection{The proof of Lemma \ref{neumerical-scheme}} \label{Proof of Lemma 3.7}
\begin{proof}
\begin{align*}
|u_l^n-(\delta_{x}^+u)_l^n|&= \left|((\delta_{x}^+\rho)_l^n-(\nabla_x\rho)(t_l,x_n))\nabla(\phi+\psi*\rho_{t_l})(x_n) \right| \\
&\leq \|\rho\|_{C^{0,2}}(\|\phi\|_{C^1}+\|\psi\|_{C^1})\Delta x\\
&\leq (\kappa_1+\kappa_2)\|\rho\|_{C^{0,2}}(\|\phi\|_{\mathcal{H}_{K_1}}+\|\psi\|_{\mathcal{H}_{K_2}})\Delta x
\end{align*}
Similarly,
\begin{align*}
|(\delta_{x}^+w)_{l,k}^{n,m}-w_{l,k}^{n,m}|&= \left|((\delta_{x}^+\rho)_l^n-(\nabla_x\rho)(t_l,x_n))\left(\rho_k^m\nabla^{(2,1)}+(\nabla_x\rho)_k^m\nabla^{(1,1)}\right)(K_1+K_2**(\rho_{t_l},\rho_{t_k}))(x_n,x_m) \right| \\
&\leq \|\rho\|_{C^{0,2}}\|\rho\|_{C^{0,1}}(\|K_1\|_{C^3}+\|K_2\|_{C^3})\Delta x\\
&\leq (\kappa_1^2+\kappa_2^2)\|\rho\|^2_{C^{0,2}} \Delta x. 
\end{align*}
Similary,
\begin{align*}
|(\delta_{y}^+w)_{l,k}^{n,m}-\delta_{xy}^+w_{l,k}^{n,m}|\leq (\kappa_1^2+\kappa_2^2)\|\rho\|^2_{C^{0,2}} \Delta x. 
\end{align*}
{For the gradient flow \eqref{gradiet-flow},} we have:
\begin{equation*}
\begin{aligned}
|f(t_l, x_n) - f^\delta(t_l, x_n)| 
&\leq |(\partial_t \rho)_l^n - (\delta_t^+ \rho)_l^n| 
+ |((\partial_x \rho)_l^n - (\delta_x^+ \rho)_l^n) U''(\rho_l^n)|\\
&\leq \|\rho\|_{C^{2,0}} \Delta t + \|\rho\|_{C^{0,2}} \|U\|_{C_b^2} \Delta x.
\end{aligned}
\end{equation*}
And for the Wasserstein Hamiltonian flow \eqref{Hamiltonian-flow}, we have:
\begin{align*}
|f(t_l, x_n) - f^\delta(t_l, x_n)| 
&\leq |\left(\partial_{tt}\rho\right)_l^n-\left(\delta_{tt}^{+} \rho\right)_l^n| + |\Gamma_W\left(\left(\delta_t^{+} \rho\right)_l^n, \left(\delta_t^{+} \rho\right)_l^n\right) -\Gamma_W\left(\left(\delta_t^{+} \rho\right)_l^n, \left(\delta_t^{+} \rho\right)_l^n\right)|\\
&\ \ \ +  |((\partial_x \rho)_l^n - (\delta_x^+ \rho)_l^n) U''(\rho_l^n)|\\
&\leq \|\rho\|_{C^{3,0}} \Delta t + \|\rho\|_{C^{0,2}} \|U\|_{C_b^2} \Delta x.
\end{align*}
The results follow.
\end{proof}

\subsection{Proof of Lemma \ref{pre-estimator1}}\label{Proof of Lemma 3.8}

\begin{proof}
Firstly, by Lemma \ref{analyze-u-w} and \ref{neumerical-scheme}, we have that
\begin{align*}
J_3:&= \left\langle B_{NL}(\phi,\psi),B_{NL}(\phi,\psi)-B_{NL}^\delta(\phi,\psi)\right\rangle_{\mathcal{H}_{K_1}\times\mathcal{H}_{K_2}}\\
&=\sum_{n,l=1}^{NL} \sum_{m,k=1}^{NL} u_l^n\rho_l^n\rho_k^m\left(w_{l,k}^{n,m}u_k^m-(\delta_{x}^+w)_{l,k}^{n,m}(\delta_{x}^+u)_k^m\right)(\Delta x)^2(\Delta t)^2\\
&\leq T^2|\Omega|^2 \|u\rho^2\|_{\infty} \left(\|w\|_{\infty} |u_k^m-(\delta_{x}^+u)_k^m|+  \|u\|_{\infty}|w_{l,k}^{n,m}-(\delta_{x}^+w)_{l,k}^{n,m}| \right) \Delta x\\
&\leq \frac{C_2^2}{2} \|(\phi,\psi)\|^2_{\mathcal{H}_{K_1}\times\mathcal{H}_{K_2}}\Delta x,
\end{align*}
where $C_2:=T|\Omega|(\kappa_1+\kappa_2)^2\|\rho\|_{C^{0,2}}^3.$
Similarly,
\begin{align*}
J_4:&= \left\langle B_{NL}^\delta(\phi,\psi),B_{NL}(\phi,\psi)-B_{NL}^\delta(\phi,\psi)\right\rangle_{\mathcal{H}_{K_1}\times\mathcal{H}_{K_2}}\\
&=\sum_{n,l=1}^{NL} \sum_{m,k=1}^{NL} (\delta_{x}^+u)u_l^n\rho_l^n\rho_k^m\left((\delta_{y}^+w)_{l,k}^{n,m}u_{k}^m-(\delta_{xy}^+w)_{l,k}^{n,m}(\delta_{x}^+u)_k^m\right)(\Delta x)^2(\Delta t)^2\\
&\leq T^2|\Omega|^2~ \|u\rho^2\|_{\infty} \left(\|w\|_{\infty} |u_k^m-(\delta_{x}^+u)_k^m|+  \|u\|_{\infty}|\delta_{y}^+w_{l,k}^{n,m}-(\delta_{xy}^+w)_{l,k}^{n,m}|\right) \Delta x\\
&\leq \frac{C_2^2}{2} \|(\phi,\psi)\|^2_{\mathcal{H}_{K_1}\times\mathcal{H}_{K_2}}\Delta x.
\end{align*}  
Therefore,
\begin{align*}
\|B_{NL}(\phi,\psi)-B_{NL}^\delta(\phi,\psi)\|^2_{\mathcal{H}_{K_1}\times\mathcal{H}_{K_2}} =J_1-J_2  \leq C_2^2 \|(\phi,\psi)\|^2_{\mathcal{H}_{K_1}\times\mathcal{H}_{K_2}}\Delta x.  
\end{align*}
\end{proof}

\subsection{Proof of Lemma \ref{Lip-flow}}
\begin{proof}
By the representation of the flow map $(X_t)_{t\in[0,T]}$ in \eqref{flow-rep}, one can compute that for all $\mathbf x,\mathbf y\in \mathbb{T}^d$,
\begin{align*}
\|X_t(\mathbf x)-X_t(\mathbf y)\| &\leq \|\mathbf{x}-\mathbf{y}\| + \int_0^t\|\nabla \Phi(\mathbf x)-\nabla \Phi(\mathbf y)\|\mathrm ds+\int_{0}^t\int_0^s\|\nabla V(X_\tau(\mathbf x))-\nabla V(X_\tau(\mathbf y)) \|\mathrm d\tau\mathrm ds\\
&\ \ \ +\int_{0}^t\int_0^s\int_{\mathbb{R}^d}\|\nabla W(X_\tau(\mathbf x)-\mathbf{z})-\nabla W(X_\tau(\mathbf y)-\mathbf{z})\|\rho_{\tau}(\mathbf z) \mathrm d\mathbf z\mathrm d\tau\mathrm ds\\
&\leq \|\mathbf{x}-\mathbf{y}\|+  td\|\Phi\|_{C^2}\|\mathbf{x}-\mathbf{y}\|+t\int_0^t\|\nabla V(X_\tau(\mathbf x))-\nabla V(X_\tau(\mathbf y)) \|\mathrm d\tau\\
&\ \ \ +t\int_0^t\int_{\mathbb{R}^d}\|\nabla W(X_\tau(\mathbf x)-\mathbf{z})-\nabla W(X_\tau(\mathbf y)-\mathbf{z})\|\rho_{\tau}(\mathbf z) \mathrm d\mathbf z\mathrm d\tau\\
& \leq \left(1+td\|\Phi\|_{C^2}\right)\|\mathbf{x}-\mathbf{y}\|+td\left(\|V\|_{C^2}+\|W\|_{C^2}\right)\int_0^t\|X_\tau(\mathbf x)-X_\tau(\mathbf y) \|\mathrm d\tau.
\end{align*}
Then by Gr\"{o}nwall's inequality \cite[ Lemma 1.1]{barbu2016differential}, we obtain that
\begin{align*}
\|X_t(\mathbf x)-X_t(\mathbf y)\|\leq    \left(1+td\|\Phi\|_{C^2}\right)\exp\left\{\frac{1}{2}d\left(\|V\|_{C^2}+\|W\|_{C^2}\right)t^2 \right\}\|\mathbf{x}-\mathbf{y}\|,
\end{align*}
which implies the map $X_t$ is Lipschitz for all $t\in[0,T]$.
\end{proof}

\subsection{Proof of Lemma \ref{dist-control}}
\begin{proof}
By the representation of the flow map in \eqref{flow-rep}, for each $\mathbf{x}\in \mathbb{T}^d$, one can compute
\begin{equation*}
\begin{aligned}
\left\|X_t(\mathbf{x})-\widehat{X}_t(\mathbf{x})\right\|^2 &=\left\|\int_{0}^t\int_0^s\nabla (V+W*\rho_{\tau})(X_\tau(\mathbf x))- \nabla (\widehat V+\widehat W*\widehat\rho_{\tau})(\widehat X_\tau(\mathbf x))\mathrm d\tau\mathrm ds\right\|^2\\
&\leq \int_{0}^t\int_0^s\left\|\nabla (V+W*\rho_{\tau})(X_\tau(\mathbf x))- \nabla (\widehat V+\widehat W*\widehat\rho_{\tau})(\widehat X_\tau(\mathbf x))\right\|^2\mathrm d\tau\mathrm ds\\
&\leq t \int_{0}^t\left\|\nabla (V+W*\rho_{s})(X_s(\mathbf x))- \nabla (\widehat V+\widehat W*\widehat\rho_{s})(\widehat X_s(\mathbf x))\right\|^2\mathrm ds \\
&\leq  4t \int_0^t\left\|\nabla V\left(X_s(\mathbf{x})\right)-\nabla \widehat{V}\left(X_s(\mathbf{x})\right)\right\|^2 \mathrm{d} s+ 4t \int_0^t\left\|\nabla \widehat{V}\left(X_s(\mathbf{x})\right)-\nabla \widehat{V}\left(\widehat{X}_s(\mathbf{x})\right)\right\|^2 \mathrm{d} s\\
&\ \ \ + 4t \int_0^t\left\|\nabla \left(W * \rho_s\right)\left(X_s(\mathbf{x})\right)-\nabla \left(\widehat{W}* \widehat\rho_s\right)\left(X_s(\mathbf{x})\right)\right\|^2 \mathrm{d} s \\
&\ \ \ + 4t \int_0^t\left\|\nabla \left(\widehat{W} * \widehat\rho_s\right)\left(X_s(\mathbf{x})\right)-\nabla \left(\widehat{W}* \widehat{\rho}_s\right)\left(\widehat{X}_s(\mathbf{x})\right)\right\|^2 \mathrm{d} s \\
& \leq  4t\int_0^t\left\|\nabla V\left(X_s(\mathbf{x})\right)-\nabla \widehat{V}\left(X_s(\mathbf{x})\right)\right\|^2 \mathrm{d} s + 4t d^2\|\widehat V\|^2_{C^2}\int_0^t\left\|X_s(\mathbf{x})-\widehat{X}_s(\mathbf{x})\right\|^2\mathrm{d} s\\
&\ \ \ + 4t \int_0^t\int_{\mathbb{R}^d}\left\|\nabla W(X_s(\mathbf{x})-\mathbf y)- \nabla\widehat{W}(X_s(\mathbf{x})-\mathbf y)\right\|^2\rho_s(\mathbf y)\mathrm d\mathbf y \mathrm{d} s\\
&\ \ \ + 4t \int_0^t\left\|\int_{\mathbb{R}^d} \nabla\widehat{W}(X_s(\mathbf{x})-\mathbf y)\left(\rho_s(\mathbf y)-\widehat\rho_s(\mathbf y)\right)\mathrm d\mathbf y\right\|^2 \mathrm{d} s\\
&\ \ \ + 4t \int_0^t\int_{\mathbb{R}^d}\left\|\nabla \widehat W(X_s(\mathbf{x})-\mathbf y)- \nabla\widehat{W}(\widehat X_s(\mathbf{x})-\mathbf y)\right\|^2\widehat\rho_s(\mathbf y)\mathrm d\mathbf y \mathrm{d} s \\
& \leq  4t^2d^2\left(\| V- \widehat{V}\|^2_{C^2}+\|W -\widehat{W}\|_{C^2}^2\right)+  4td\|\widehat W\|_{C^2} \int_0^t W_2(\rho_s,\widehat \rho_s) \mathrm{d} s\\
&\ \ \ + 4t d^2\left(\|\widehat V\|^2_{C^2}+\|\widehat W\|^2_{C^2}\right)\int_0^t\left\|X_s(\mathbf{x})-\widehat{X}_s(\mathbf{x})\right\|^2\mathrm{d} s.
\end{aligned}
\end{equation*}
Then by Gr\"{o}nwall's inequality \cite[ Lemma 1.1]{barbu2016differential}, we obtain that
\begin{align*}
\left\|X_t(\mathbf{x})-\widehat{X}_t(\mathbf{x})\right\|^2 \leq &\left(4t^2d^2\left(\| V- \widehat{V}\|^2_{C^2}+\|W -\widehat{W}\|_{C^2}^2\right)+4td\|\widehat W\|_{C^2} \int_0^t W_2(\rho_s,\widehat \rho_s) \mathrm{d} s\right)\\
&\exp\left\{2t^2 d^2\left(\|\widehat V\|^2_{C^2}+\|\widehat W\|^2_{C^2}\right)\right\}.
\end{align*}
\end{proof}

\subsection{Proof of Proposition \ref{stability estimate}}
\begin{proof}
By Proposition \ref{hamiltonian-flow}, it is known that given our assumptions on $W,V, \widehat{W}$ and $\widehat{V}$, the solutions of \eqref{Hamiltonian-flow} are of the form $\rho_t=X_t \# \rho_{0}$, $ \widehat{\rho}_t=\widehat{X}_t \# \widehat{\rho}_0$, where $X_t$, $\widehat{X}_t$ are the flow maps of equations \eqref{flow} with the same initial conditions $X_0=I d, \dot{X}_0=\nabla \Phi (Id)$ for some smooth function $\Phi$ chosen in $C^2(\mathbb T^d)$ but by the accelerated velocity fields $\nabla W * \rho_t+\nabla V$ and $\nabla \widehat{W} * \widehat{\rho}_t+\nabla \widehat{V}$, respectively. Then we have the following
estimate
\begin{equation}\label{Wass-dis}
\begin{aligned}
W_2^2\left(\rho_t, \widehat{\rho}_t\right) & =W_2^2\left(X_t \# \rho_0, \widehat{X}_t \# \widehat{\rho}_0\right) \leq W_2^2\left(X_t \# \rho_0, \widehat{X}_t \# \rho_0\right)+W_2^2\left(\widehat{X}_t \# \rho_0, \widehat{X}_t \# \widehat{\rho}_0\right) \\
& \leq \int_{\mathbb{T}^d}\left\|X_t(\mathbf{x})-\widehat{X}_t(\mathbf{x})\right\|^2 \mathrm{d} \rho_0(\mathbf{x})+W_2^2\left(\widehat{X}_t \# \rho_0, \widehat{X}_t \# \widehat{\rho}_0\right).
\end{aligned}
\end{equation}
Denote the product measure $\Pi_t:=\left(\widehat{X}_t \times \widehat{X}_t\right) \# \Pi_0$, where $\Pi_0$ is the optimal transport plan between $\rho_0$ and $\widehat{\rho}_0$ with respect to the 2-Wasserstein metric. Then, by definition of the 2-Wasserstein metric we have that
\begin{equation}\label{eqn1}
\begin{aligned}
W_2^2\left(\widehat{X}_t \# \rho_0, \widehat{X}_t \# \widehat{\rho}_0\right) & \leq \int_{\mathbb{T}^d \times \mathbb{T}^d}|\mathbf{x}-\mathbf{y}|^2 \mathrm{d} \Pi_t=\int_{\mathbb{T}^d \times \mathbb{T}^d}\left\|\widehat{X}_t(\mathbf{x})-\widehat{X}_t(\mathbf{y})\right\|^2 \mathrm{d} \Pi_0 \\
& \leq L_{\widehat{X}_t} \int_{\mathbb{T}^d \times \mathbb{T}^d}|\mathbf{x}-\mathbf{y}|^2 \mathrm{d} \Pi_0 \leq L_{\widehat{X}_t} W_2^2\left(\rho_0, \widehat{\rho}_0\right),
\end{aligned}
\end{equation}
where $L_{\widehat{X}_t}$ is the Lipschitz constant given in Lemma \ref{Lip-flow} of the flow map $\widehat{X}_t$. Furthermore, by Lemma \ref{dist-control}, we obtain that
\begin{equation}\label{eqn2}
\begin{aligned}
\int_{\mathbb{T}^d}\left\|X_t(\mathbf{x})-\widehat{X}_t(\mathbf{x})\right\|^2 \mathrm{d}& \rho_0(\mathbf{x})\leq \exp\left\{2t^2 d^2\left(\|\widehat V\|^2_{C^2}+\|\widehat W\|^2_{C^2}\right)\right\}\\
&\left(4t^2d^2\left(\| V- \widehat{V}\|^2_{C^2}+\|W -\widehat{W}\|_{C^2}^2\right)+4td\|\widehat W\|_{C^2} \int_0^t W_2(\rho_s,\widehat \rho_s) \mathrm{d} s\right).
\end{aligned}
\end{equation}
Then combing equations \eqref{Wass-dis},\eqref{eqn1} and \eqref{eqn2}, we obtain that
\begin{align*}
W_2^2\left(\rho_t, \widehat{\rho}_t\right)&\leq  L_{\widehat{X}_t} W_2^2\left(\rho_0, \widehat{\rho}_0\right) +\exp\left\{2t^2 d^2\left(\|\widehat V\|^2_{C^2}+\|\widehat W\|^2_{C^2}\right)\right\}\\
&\left(4t^2d^2\left(\| V- \widehat{V}\|^2_{C^2}+\|W -\widehat{W}\|_{C^2}^2\right)+4td\|\widehat W\|_{C^2} \int_0^t W_2(\rho_s,\widehat \rho_s) \mathrm{d} s\right).
\end{align*}
Then the Gr\"{o}nwall's inequality \cite[ Lemma 1.1]{barbu2016differential} and the differential reproducing property \cite[Theorem 2.7]{RCSP2} yield that
\begin{align*}
W_2^2\left(\rho_t, \widehat{\rho}_t\right) \leq C_1(t)W_2^2\left(\rho_0, \widehat{\rho}_0\right) +C_2(t)\left(\kappa_1^2\| V- \widehat{V}\|^2_{\mathcal{H}_{K_1}}+\kappa_2^2\|W -\widehat{W}\|_{\mathcal{H}_{K_2}}^2\right) ,
\end{align*}
where $C_1(t)$ and $ C_2(t)$ are defined in \eqref{coeff}.
\end{proof}

\section{Calculus on the density manifold}
{
\label{apped: density manifold intro}
\subsubsection*{The tangent bundle of the density manifold $\mathcal{P}_{+}(M)$}
We now describe the tangent spaces of the density manifold $\mathcal{P}_{+}(M)$. We say that two curves $\gamma_1,\gamma_2:(-\varepsilon,\varepsilon)\to \mathcal{P}_+(M)$ with $\gamma_1(0)=\gamma_2(0)=\rho$ for some $\varepsilon>0$ are equivalent, if $(f\circ \gamma_1)'(0)=(f\circ \gamma_2)'(0)$ for every smooth real-valued function $f: U _\rho \subset \mathcal{P}_+(M) \longrightarrow \mathbb{R}$ defined in a neighborhood $U _\rho $ of $\rho$. This defines an equivalence relation on the set
of all smooth curves of the form $\gamma:(-\varepsilon,\varepsilon)\to \mathcal{P}_+(M)$ with $\gamma(0)=\rho$. The tangent space of the density manifold $\mathcal{P}_{+}(M)$ at $\rho \in \mathcal{P}_{+}(M) $ denoted by $T_{\rho}\mathcal{P}_{+}(M) $ is defined as
\begin{align*}
T_{\rho}\mathcal{P}_{+}(M):=\left\{ [\gamma]\mid \gamma:(-\varepsilon,\varepsilon)\to \mathcal{P}_+(M), \gamma(0)=\rho \right\},   
\end{align*}
where $[\gamma]$ is the equivalence class of curves associated to $\gamma:(-\varepsilon,\varepsilon)\to \mathcal{P}_+(M) $. The tangent bundle $T\mathcal{P}_{+}(M)  $ of the density manifold is defined as
\begin{align*}
T\mathcal{P}_{+}(M) := \bigcup_{\rho \in \mathcal{P}_{+}(M)} T_{\rho}\mathcal{P}_{+}(M).
\end{align*}

In the following proposition, we provide a characterization of the tangent spaces of the density manifold $\mathcal{P}_{+}(M)$ that was initially presented in \cite{otto2001geometry} and was utilized within the framework of Riemannian geometry to analyze and quantify the asymptotic behavior of the porous medium equation. 
\begin{proposition}
For each $\rho \in \mathcal{P}_{+}(M) $, the tangent space $T_{\rho}\mathcal{P}_{+}(M) $ can be written as  
\begin{align*}
T_{\rho}\mathcal{P}_{+}(M)= \left\{\sigma\in C^\infty(M)~\bigg|~\int_M \sigma\, d\operatorname{vol}_M = 0\right\}.
\end{align*}
\end{proposition}

\begin{proof}
First, note that for any \( \sigma \in C^\infty(M) \) with zero mean, the curve \( t \mapsto \rho - t \sigma \), \( t \in \mathbb{R} \), remains on the density manifold \( \mathcal{P}_{+}(M) \), which shows that \( \sigma \in T_{\rho}\mathcal{P}_{+}(M) \). Therefore, the following inclusion holds:
\begin{align*}
T_{\rho}\mathcal{P}_{+}(M) \supseteq \left\{\sigma\in C^\infty(M)~\bigg|~\int_M \sigma\, d\operatorname{vol}_M = 0\right\}.   
\end{align*}
To prove the reverse inclusion, let \( \gamma:(\varepsilon,\varepsilon)\to \mathcal{P}_+(M) \) be an arbitrary smooth curve in the density manifold \( \mathcal{P}_{+}(M) \) with \( \gamma(0) = \rho \). Then:
\begin{align*}
\int_M \gamma'(0)\, d\operatorname{vol}_M = \frac{d}{dt}\bigg|_{t=0} \int_M \gamma(t)\, d\operatorname{vol}_M = 0,
\end{align*}
which implies \( \gamma'(0) \in C^\infty(M) \) has zero mean and hence 
$
T_{\rho}\mathcal{P}_{+}(M) \subseteq \left\{\sigma\in C^\infty(M)~\bigg|~\int_M \sigma\, d\operatorname{vol}_M = 0\right\}$, necessarily.
\end{proof}

\begin{proposition}
\label{vrho map}
For each $\rho \in \mathcal{P}_{+}(M) $, the operator $h_\rho$ defined in \eqref{V-rho} is a linear isomorphism from the quotient space $C^\infty(M)/\mathbb{R}$ to the tangent space $T_{\rho}\mathcal{P}_{+}(M)$.
\end{proposition}

\begin{proof}
Since \( M \) is a manifold without boundary, integration by parts yields:
\begin{equation}\label{integration by parts}
\int_M h_\rho(\phi)\, \psi\, d\operatorname{vol}_M = \int_M \langle \nabla\phi, \nabla\psi \rangle_g\, \rho\, d\operatorname{vol}_M, \quad \text{for all } \phi, \psi \in C^\infty(M).
\end{equation}
Setting \( \psi \equiv 1 \), we have:
\begin{equation*}
\int_M h_\rho(\phi)\, d\operatorname{vol}_M = 0, \quad \text{for all } \phi \in C^\infty(M),
\end{equation*}
which shows that the image of \( h_\rho \) is contained in the tangent space $T_{\rho}\mathcal{P}_{+}(M)$. To prove the injectivity of $h_\rho$, assume that $\phi \in C^\infty(M) $ is such that  \( h_\rho(\phi) = 0 \). Using \eqref{integration by parts}, we have:
\begin{align*}
\int_M \|\nabla\phi\|_g^2 \rho\, d\operatorname{vol}_M = 0.
\end{align*}
Since \( \rho \) is strictly positive, this implies that \( \nabla\phi = 0 \), so \( \phi \) is constant. Thus, \( h_\rho \) is injective on \( C^\infty(M)/\mathbb{R} \). 

We now establish the surjectivity of \( h_\rho \). Let $H^1(\rho)/\mathbb{R}$ denote the Sobolev space $H^1(\rho)$ defined up to constant functions. This space is a Hilbert space with the inner product
\begin{align*}
\langle \phi, \psi \rangle_{H^1(\rho)} := \int_M \langle \nabla \phi, \nabla \psi \rangle_g\, \rho\, d\operatorname{vol}_M,
\end{align*}
and associated norm
\begin{align*}
\|\phi\|_{H^1(\rho)} := \left( \int_M \|\nabla \phi\|_g^2\, \rho\, d\operatorname{vol}_M \right)^{1/2}.
\end{align*}
The bilinear form $a[\phi, \psi] := \langle \phi, \psi \rangle_{H^1(\rho)}$ satisfies
\begin{align*}
|a[\phi, \psi]| \leq \|\phi\|_{H^1(\rho)} \|\psi\|_{H^1(\rho)}, \quad a[\phi, \phi] \geq \|\phi\|_{H^1(\rho)}^2,
\end{align*}
for all $\phi, \psi \in H^1(\rho)/\mathbb{R}$. By the Lax--Milgram theorem \cite[p.~297]{evans2022partial}, for any continuous linear functional $f$ on $H^1(\rho)/\mathbb{R}$, there exists a unique $\phi \in H^1(\rho)/\mathbb{R}$ such that
\begin{align*}
a[\phi, \psi] = \langle f, \psi \rangle_{L^2}, \quad \text{for all } \psi \in H^1(\rho)/\mathbb{R}.
\end{align*}
This $\phi$ is the weak solution to the {\it weighted Laplacian} equation
\begin{equation}
\label{weighted Laplacian}
\Delta_\rho \phi:=\nabla \cdot (\rho\nabla\phi) = f.
\end{equation}
If $f \in C^\infty(M)$, the Elliptic Regularity Theorem \cite[Theorem 6.30]{warner1983foundations} ensures that the weak solution $\phi$ is smooth. This establishes the surjectivity of \( h_\rho \) since by \eqref{weighted Laplacian} we have that \( h_\rho(-\phi)= f\).
\end{proof}

\subsubsection*{The cotangent space of the density manifold $\mathcal{P}_{+}(M)$}

\begin{proposition}\label{characterization}
The map $f _\rho : C^\infty(M)/\mathbb{R} \to \mathrm{Im}(\flat_\rho)\subsetneq T^*_{\rho}\mathcal{P}_{+}(M)$ defined in \eqref{iso} is a linear isomorphism.
\end{proposition}

\begin{proof}
First, $f _\rho$ is well-defined because $\Delta_\rho(\phi + C) = \Delta_\rho \phi$ for any constant $C$. Linearity and continuity of $f _\rho$ follow directly from the properties of Otto's Riemannian metric $g_W$. In order to prove injectivity, suppose that $\phi \in C^\infty(M) $  is such that  $f _\rho(\phi) = 0$. Then, for all $\sigma \in T_{\rho}\mathcal{P}_{+}(M)$,
\begin{equation*}
g_W(\rho)(\sigma, -\Delta_\rho \phi) = 0.
\end{equation*}
By the non-degeneracy of $g_W$, this implies $-\Delta_\rho \phi = 0$. Hence, $\phi$ is determined up to an additive constant, since the weighted Laplacian $\Delta_\rho$ is an isomorphism between $C^\infty(M)/\mathbb{R}$ and $T_{\rho}\mathcal{P}_{+}(M)$. Thus, $f _\rho$ is injective on $C^\infty(M)/\mathbb{R}$.

We now prove the surjectivity of $f _\rho$. {Let $\alpha \in \mathrm{Im}(\flat_\rho)$. Then there exists a unique $\psi \in T_{\rho}\mathcal{P}_{+}(M)$ such that $\alpha=g_W(\rho)(\cdot,\psi)$ at the point $\rho\in\mathcal{P}_{+}(M)$.}
Since $\Delta_\rho$ is an isomorphism, there exists a unique $\phi \in C^\infty(M)/\mathbb{R}$ such that $\psi = -\Delta_\rho \phi$. Substituting this, we obtain
\begin{equation}
\label{intermediate with gw}
\langle \alpha, \sigma \rangle = g_W(\rho)(\sigma, -\Delta_\rho \phi), \quad \text{for all } \sigma \in T_{\rho}\mathcal{P}_{+}(M),
\end{equation}
which shows that $\alpha = f _\rho(\phi)$ and therefore $f _\rho$ is surjective.
\end{proof}

\subsubsection*{A brief comparison with alternative Hamiltonian structures on Wasserstein space}

The Hamiltonian structure of the Wasserstein space $\mathcal{P}_2(\mathbb{R}^{2d})$ (see \cite[Definition 5.1]{ambrosio2008hamiltonian}) has been extensively studied in recent years. For a comprehensive overview, we refer the reader to \cite{villani2009optimal} and the works \cite{ambrosio2008hamiltonian, Gangbo2011}. Much of this research is motivated by geometric approaches developed in fluid dynamics; see, for instance, \cite{arnold1966geometrie, ebin1970groups,Marsden1982TheHS, holm1998euler}. These geometric perspectives demonstrate, for example, that one-dimensional Euler--Poisson models can be viewed as minimizing paths for the natural action functional \cite{Gangbo2009}. 
For the Hamiltonian structure on the density manifold $\mathcal{P}_+(M)$, the paper \cite{lott2008some} provides insights into the associated Poisson structure. However, both of these Hamiltonian formulations are generally induced by the symplectic or Poisson structure of the underlying space, requiring the base space itself to admit such a structure.

The  symplectic form on the characterization space $\mathcal{P}_+(M)\times C^\infty(M)/\mathbb{R}$ introduced in Section~\ref{Hamiltonian structure}, offers a geometric perspective on the Wasserstein Hamiltonian flows discussed in \cite{chow2020wasserstein, Wu2025}. As it is pointed out in \cite{khesin2019geometry,khesin2021geometric}, this symplectic structure is intrinsic to the density manifold and does not depend on any symplectic or Poisson structure of the underlying space.

\subsubsection*{Infinite dimensional Hamiltonian ODEs on the Wasserstein space} 
The formulation \eqref{Ham-for} differs from the Hamiltonian ODEs \cite{ambrosio2008hamiltonian} defined on the Wasserstein space $\mathcal{P}_2(\mathbb{R}^{2d})$, which is based on Otto's calculus. The main difference between the two is that the topology of the Wasserstein space $\mathcal{P}_2(\mathbb{R}^{2d})$ is induced by the Wasserstein distance. In contrast, the smooth structure of the density manifold arises from the model space. The tangent bundle (see \cite[Definition 8.4.1]{ambrosio2008gradient}) of the Wasserstein space $\mathcal{P}_2(\mathbb{R}^{2d})$ is naturally given by
\begin{align*}
T_{\mu}\mathcal{P}_2(\mathbb{R}^{2d}) := \overline{\{\nabla \varphi : \varphi \in C_c^\infty(\mathbb{R}^{2d})\}}^{L^2(\mu;\mathbb{R}^{2d})},
\end{align*}
where the closure is in the sense of $L^2(\mu)$-norm.
The corresponding cotangent space is given by
\begin{align*}
    T^*_{\mu}\mathcal{P}_2(\mathbb{R}^{2d}) := \{\pi_\mu(Jv) \mid v \in T_{\mu}\mathcal{P}_2(\mathbb{R}^{2d})\},
\end{align*}
where $\pi_\mu : L^2(\mu, \mathbb{R}^{2d}) \to T_{\mu}\mathcal{P}_2(\mathbb{R}^{2d})$ denotes the canonical orthogonal projection, and $J$ is the canonical symplectic matrix.
Consequently, the symplectic form $\Omega_\mu : T^*_{\mu}\mathcal{P}_2(\mathbb{R}^{2d}) \times T^*_{\mu}\mathcal{P}_2(\mathbb{R}^{2d}) \to \mathbb{R}$ is defined by
\begin{align*}
    \Omega_\mu(\bar{v}_1, \bar{v}_2) := \int_{\mathbb{R}^{2d}} \langle J v_1, v_2 \rangle \, d\mu,
\end{align*}
for all $v_1, v_2 \in T_{\mu}\mathcal{P}_2(\mathbb{R}^{2d})$, where $\bar{v}_i = \pi_\mu(J v_i)$ for $i=1,2$.
In this framework, the Hamiltonian ODE takes the form
\begin{align*}
    \partial_t \mu_t + \nabla \cdot (J \nabla \mathcal{H}(\mu_t)\, \mu_t) = 0.
\end{align*}    
This Hamiltonian ODE is, in fact, a Lie–Poisson system \cite{Gangbo2011} on $(C_c^\infty)^*$ whose symplectic leaves are the orbits of the coadjoint action of $\mathrm{Ham}_c(\mathbb{R}^{2d})$ on $(C_c^\infty)^*$.

\subsubsection*{A Poisson bracket on the density manifold}
In \cite{lott2008some}, the underlying space $M$ is assumed to be a Poisson manifold with a Poisson bracket $\{\cdot,\cdot\}$.
Then, the Poisson bracket in the density manifold $\mathcal{P}_+(M)$ is defined as follows:
\begin{align*}
\{F_{\phi_1},F_{\phi_2}\}_W(\mu):=\int_M\{\phi_1,\phi_2\} d\mu,    
\end{align*}
where $F_{\phi_1}, F_{\phi_2}\in C^\infty(\mathcal{P}_+(M))$ are defined as
\begin{align*}
F_{\phi_i}(\rho) = \int_M \phi_i\rho~d_{\mathrm{vol}_M}, i = 1,2,   
\end{align*}
for $\phi_1,\phi_2\in C^\infty(M)$.

\subsubsection*{Hamiltonian structure for ideal fluids}
The Poisson bracket \eqref{Poi-bra} closely resembles the Hamiltonian structure used to describe ideal fluids (see, for instance, \cite{holm1998euler}). Specifically, the Lie–Poisson bracket, often called the ideal fluid bracket, is defined on functions of the velocity field \( \mathbf{v} \) as follows:
\begin{align*}
\left\{F, G\right\}(\mathbf{v}) := \int_{\mathfrak{D}} \mathbf{v} \cdot \left[\frac{\delta F}{\delta \mathbf{v}}, \frac{\delta G}{\delta \mathbf{v}}\right]_L \, dx,
\end{align*}
where \( \mathfrak{D} \subset \mathbb{R}^3 \), \( [\cdot, \cdot]_L \) is the left Lie algebra bracket of vector fields, defined as:
\begin{align*}
[\mathbf{v}, \mathbf{w}]_L^i = \sum_{j=1}^3 \left( w^j \frac{\partial v^i}{\partial x^j} - v^j \frac{\partial w^i}{\partial x^j} \right),
\end{align*}
and the functional derivative \( \delta F / \delta \mathbf{v} \) is defined as the unique element that satisfies the equality:
\begin{align*}
\lim_{\varepsilon \to 0} \frac{1}{\varepsilon} \left[F(\mathbf{v} + \delta \mathbf{v}) - F(\mathbf{v})\right] = \int_{\mathfrak{D}} \delta \mathbf{v} \cdot \frac{\delta F}{\delta \mathbf{v}} \, dx.
\end{align*}
By choosing the energy functional as the kinetic energy:
\begin{align*}
H(\mathbf{v}) := \frac{1}{2} \int_{\mathfrak{D}} |\mathbf{v}|^2 \, dx,
\end{align*}
the Euler equations for ideal fluid dynamics can be derived directly using this Hamiltonian framework.
}

\section*{Acknowledgments}
The authors thank Prof. Lyudmila Grigoryeva, Dr. Qiao Huang, and Prof. Manuel de Le\'{o}n for helpful discussions and remarks, and acknowledge partial financial support from the School of Physical and Mathematical Sciences of the Nanyang Technological University. DY is funded by the Nanyang President's Graduate Scholarship of Nanyang Technological University. Financial support from Singapore's Ministry of Education Tier 1 grant RG100/24 entitled ``Kernel methods for global structure preserving machine learning" is also acknowledged.

\footnotesize
\addcontentsline{toc}{section}{References}
\bibliographystyle{wmaainf}
\bibliography{Refs}

@article{lund2014nonparametric,
  title={Nonparametric estimates of drift and diffusion profiles via Fokker--Planck algebra},
  author={Lund, Steven P and Hubbard, Joseph B and Halter, Michael},
  journal={The Journal of Physical Chemistry B},
  volume={118},
  number={44},
  pages={12743--12749},
  year={2014},
  publisher={ACS Publications}
}

@article{srivastava2025inference,
  title={Inference of weak-form partial differential equations describing migration and proliferation mechanisms in wound healing experiments on cancer cells},
  author={Srivastava, Siddhartha and Kinnunen, Patrick C and Wang, Zhenlin and Ho, Kenneth KY and Humphries, Brock A and Chen, Siyi and Linderman, Jennifer J and Luker, Gary D and Luker, Kathryn E and Garikipati, Krishna},
  journal={PLOS Computational Biology},
  volume={21},
  number={10},
  pages={e1013607},
  year={2025},
  publisher={Public Library of Science San Francisco, CA USA}
}

@article{batlle2025error,
  title={{Error analysis of kernel/GP methods for nonlinear and parametric PDEs}},
  author={Batlle, Pau and Chen, Yifan and Hosseini, Bamdad and Owhadi, Houman and Stuart, Andrew M.},
  journal={Journal of Computational Physics},
  volume={520},
  pages={113488},
  year={2025}
}

@inproceedings{doumeche2024physics,
  title={{Physics-informed machine learning as a kernel method}},
  author={Doum{\`e}che, Nathan and Bach, Francis and Biau, G{\'e}rard and Boyer, Claire},
  booktitle={Proceedings of the 37th Annual Conference on Learning Theory},
  pages={1399--1450},
  year={2024}
}

@article{schaback2016all,
  title={{All well-posed problems have uniformly stable and convergent discretizations}},
  author={Schaback, Robert},
  journal={Numerische Mathematik},
  volume={132},
  pages={597--630},
  year={2016}
}

@article{schaback2007convergence,
  title={{Convergence of unsymmetric kernel-based meshless collocation methods}},
  author={Schaback, Robert},
  journal={SIAM Journal on Numerical Analysis},
  volume={45},
  pages={333--351},
  year={2007}
}

@book{wendland2004scattered,
  title={{Scattered Data Approximation}},
  author={Wendland, Holger},
  year={2004},
  publisher={Cambridge University Press}
}

@article{franke1998solving,
  title={{Solving partial differential equations by collocation using radial basis functions}},
  author={Franke, Carsten and Schaback, Robert},
  journal={Applied Mathematics and Computation},
  volume={93},
  pages={73--82},
  year={1998}
}

@article{blanchard2018optimal,
  title={{Optimal rates for regularization of statistical inverse learning problems}},
  author={Blanchard, Gilles and M{\"u}cke, Nicole},
  journal={Foundations of Computational Mathematics},
  volume={18},
  pages={971--1013},
  year={2018}
}

@article{de2006discretization,
  title={{Discretization error analysis for Tikhonov regularization}},
  author={De Vito, Ernesto and Rosasco, Lorenzo and Caponnetto, Andrea},
  journal={Analysis and Applications},
  volume={4},
  pages={81--99},
  year={2006}
}

@article{schaback2006kernel,
  title={{Kernel techniques: From machine learning to meshless methods}},
  author={Schaback, Robert and Wendland, Holger},
  journal={Acta Numerica},
  volume={15},
  pages={543--639},
  year={2006}
}

@book{saitoh2016theory,
  title={{Theory of Reproducing Kernels and Applications}},
  author={Saitoh, Saburou and Sawano, Yoshihiro},
  year={2016},
  publisher={Springer}
}

@book{scholkopf2002learning,
  title={{Learning with Kernels: Support Vector Machines, Regularization, Optimization, and Beyond}},
  author={Sch{\"o}lkopf, Bernhard and Smola, Alexander J.},
  year={2002},
  publisher={MIT Press}
}

@book{leveque2002finite,
  title={Finite volume methods for hyperbolic problems},
  author={LeVeque, Randall J},
  volume={31},
  year={2002},
  publisher={Cambridge university press}
}

@article{osher1988fronts,
  title={Fronts propagating with curvature-dependent speed: Algorithms based on Hamilton-Jacobi formulations},
  author={Osher, Stanley and Sethian, James A},
  journal={Journal of computational physics},
  volume={79},
  number={1},
  pages={12--49},
  year={1988},
  publisher={Elsevier}
}

@book{abraham2012manifolds,
  title={Manifolds, Tensor Analysis, and Applications},
  author={Abraham, Ralph and Marsden, Jerrold E and Ratiu, Tudor},
  year={2012},
  publisher={Springer Science \& Business Media}
}

@article{khesin2019geometry,
  title={Geometry of the Madelung transform},
  author={Khesin, Boris and Misio{\l}ek, Gerard and Modin, Klas},
  journal={Archive for Rational Mechanics and Analysis},
  volume={234},
  number={2},
  pages={549--573},
  year={2019},
  publisher={Springer}
}

@article{khesin2021geometric,
  title={Geometric hydrodynamics and infinite-dimensional Newton’s equations},
  author={Khesin, Boris and Misio{\l}ek, Gerard and Modin, Klas},
  journal={Bulletin of the American Mathematical Society},
  volume={58},
  number={3},
  pages={377--442},
  year={2021}
}

@book{ambrosio2021lectures,
  title={Lectures on Optimal Transport},
  author={Ambrosio, Luigi and Bru{\'e}, Elia and Semola, Daniele},
  year={2021},
  publisher={Springer Nature}
}

@article{li2018geometry,
  title={Geometry of probability simplex via optimal transport},
  author={Li, Wuchen},
  journal={arXiv preprint arXiv:1803.06360},
  volume={2},
  number={4},
  pages={13},
  year={2018}
}

@book{Mohri:learning:2012,
author = {Mohri, Mehryar and Rostamizadeh, Afshin and Tawalkar, Ameet},
edition = {Second},
isbn = {9780262018258},
publisher = {The MIT Press},
title = {{Foundations of Machine Learning}},
year = {2018}
}

@book{Marsden1994,
address = {New York},
author = {Marsden, Jerrold E. and Ratiu, Tudor S.},
edition = {Second},
pages = {500},
publisher = {Springer-Verlag},
title = {{Introduction to Mechanics and Symmetry}},
year = {1999}
}

@book{Abraham1978,
author = {Abraham, Ralph and Marsden, Jerrold E.},
edition = {2nd},
publisher = {Addison-Wesley, Reading, MA},
title = {{Foundations of Mechanics}},
year = {1978}
}

@article{holm1998euler,
  title={The Euler--Poincar{\'e} equations and semidirect products with applications to continuum theories},
  author={Holm, Darryl D and Marsden, Jerrold E and Ratiu, Tudor S},
  journal={Advances in Mathematics},
  volume={137},
  number={1},
  pages={1--81},
  year={1998},
  publisher={Elsevier}
}

@article{Wu2025,
   author = {Hao Wu and Shu Liu and Xiaojing Ye and Haomin Zhou},
   doi = {10.1137/23M159281X},
   issn = {0036-1429},
   issue = {1},
   journal = {SIAM Journal on Numerical Analysis},
   month = {2},
   pages = {360-395},
   title = {Parameterized Wasserstein Hamiltonian flow},
   volume = {63},
   year = {2025}
}

@book{schmeding2022introduction,
  title={An Introduction to Infinite-dimensional Differential Geometry},
  author={Schmeding, Alexander},
  volume={202},
  year={2022},
  publisher={Cambridge University Press}
}

@book{kriegl1997convenient,
  title={The Convenient Setting of Global Analysis},
  author={Kriegl, Andreas and Michor, Peter W},
  volume={53},
  year={1997},
  publisher={American Mathematical Socciety}
}

@book{do:carmo:1993,
author = {do Carmo, Manfredo Perdig{\~{a}}o},
publisher = {Birkh{\"{a}}user Boston},
title = {{Riemannian Geometry}},
year = {1992}
}

@article{he2023group,
  title={Group projected subspace pursuit for identification of variable coefficient differential equations (GP-IDENT)},
  author={He, Yuchen and Kang, Sung Ha and Liao, Wenjing and Liu, Hao and Liu, Yingjie},
  journal={Journal of Computational Physics},
  volume={494},
  pages={112526},
  year={2023},
  publisher={Elsevier}
}

@article{rudy2019data,
  title={Data-driven identification of parametric partial differential equations},
  author={Rudy, Samuel and Alla, Alessandro and Brunton, Steven L and Kutz, J Nathan},
  journal={SIAM Journal on Applied Dynamical Systems},
  volume={18},
  number={2},
  pages={643--660},
  year={2019},
  publisher={SIAM}
}

@article{kang2021ident,
  title={Ident: Identifying differential equations with numerical time evolution},
  author={Kang, Sung Ha and Liao, Wenjing and Liu, Yingjie},
  journal={Journal of Scientific Computing},
  volume={87},
  pages={1--27},
  year={2021},
  publisher={Springer}
}

@article{rudy2017data,
  title={Data-driven discovery of partial differential equations},
  author={Rudy, Samuel H and Brunton, Steven L and Proctor, Joshua L and Kutz, J Nathan},
  journal={Science advances},
  volume={3},
  number={4},
  pages={e1602614},
  year={2017},
  publisher={American Association for the Advancement of Science}
}

@book{wright2022high,
  title={{High-dimensional Data Analysis with Low-dimensional Models: Principles, Computation, and Applications}},
  author={Wright, John and Ma, Yi},
  year={2022},
  publisher={Cambridge University Press}
}

@article{miller2023learning,
  title={Learning theory for inferring interaction kernels in second-order interacting agent systems},
  author={Miller, Jason and Tang, Sui and Zhong, Ming and Maggioni, Mauro},
  journal={Sampling Theory, Signal Processing, and Data Analysis},
  volume={21},
  number={1},
  pages={21},
  year={2023},
  publisher={Springer}
}

@article{benamou2000computational,
  title={A computational fluid mechanics solution to the Monge-Kantorovich mass transfer problem},
  author={Benamou, Jean-David and Brenier, Yann},
  journal={Numerische Mathematik},
  volume={84},
  number={3},
  pages={375--393},
  year={2000},
  publisher={Springer-Verlag Berlin/Heidelberg}
}

@article{lott2008some,
  title={Some geometric calculations on {W}asserstein space},
  author={Lott, John},
  journal={Communications in Mathematical Physics},
  volume={277},
  number={2},
  pages={423--437},
  year={2008}
}

@book{ambrosio2008gradient,
  title={{Gradient Flows: in Metric Spaces and in the Space of Probability Measures}},
  author={Ambrosio, Luigi and Gigli, Nicola and Savar{\'e}, Giuseppe},
  year={2008},
  publisher={Springer Science \& Business Media}
}

@book{vazquez2007porous,
  title={{The Porous Medium Equation: Mathematical Theory}},
  author={V{\'a}zquez, Juan Luis},
  year={2007},
  publisher={Oxford university press}
}

@article{otto2001geometry,
  title={The geometry of dissipative evolution equations: the porous medium equation},
  author={Otto, F},
  journal={Comm. Partial Differential Equations},
  volume={26},
  pages={101--174},
  year={2001}
}

@article{bodnar2006integro,
  title={An integro-differential equation arising as a limit of individual cell-based models},
  author={Bodnar, Marek and Vel{\'a}zquez, Juan Jos{\'e} Lop{\'e}z},
  journal={Journal of Differential Equations},
  volume={222},
  number={2},
  pages={341--380},
  year={2006},
  publisher={Elsevier}
}

@article{topaz2006nonlocal,
  title={A nonlocal continuum model for biological aggregation},
  author={Topaz, Chad M and Bertozzi, Andrea L and Lewis, Mark A},
  journal={Bulletin of mathematical biology},
  volume={68},
  pages={1601--1623},
  year={2006},
  publisher={Springer}
}

@article{carrillo2019population,
  title={A population dynamics model of cell-cell adhesion incorporating population pressure and density saturation},
  author={Carrillo, Jose A and Murakawa, Hideki and Sato, Makoto and Togashi, Hideru and Trush, Olena},
  journal={Journal of theoretical biology},
  volume={474},
  pages={14--24},
  year={2019},
  publisher={Elsevier}
}

@article{RCSP2,
    author = {Jianyu Hu and Juan-Pablo Ortega and Daiying Yin},
   doi = {10.1090/mcom/4106},
   issn = {1088-6842},
   journal = {Mathematics of Computation},
   month = {6},
   title = {A structure-preserving kernel method for learning Hamiltonian systems},
   year = {2025}
}

@article{owhadi2017separability,
  title={Separability of reproducing kernel spaces},
  author={Owhadi, Houman and Scovel, Clint},
  journal={Proceedings of the American Mathematical Society},
  volume={145},
  number={5},
  pages={2131--2138},
  year={2017}
}

@book{barbu2016differential,
  title={Differential Equations},
  author={Barbu, Viorel},
  year={2016},
  publisher={Springer}
}

@article{carrillo2025sparse,
  title={Sparse identification of nonlocal interaction kernels in nonlinear gradient flow equations via partial inversion},
  author={Carrillo, Jose A and Estrada-Rodriguez, Gissell and Mikolas, Laszlo and Tang, Sui},
  journal={Mathematical Models and Methods in Applied Sciences},
  volume={35},
  number={05},
  pages={1073--1131},
  year={2025},
  publisher={World Scientific}
}

@article{carrillo2015finite,
  title={A finite-volume method for nonlinear nonlocal equations with a gradient flow structure},
  author={Carrillo, Jos{\'e} A and Chertock, Alina and Huang, Yanghong},
  journal={Communications in Computational Physics},
  volume={17},
  number={1},
  pages={233--258},
  year={2015},
  publisher={Cambridge University Press}
}

@article{gottlieb2001strong,
  title={Strong stability-preserving high-order time discretization methods},
  author={Gottlieb, Sigal and Shu, Chi-Wang and Tadmor, Eitan},
  journal={SIAM review},
  volume={43},
  number={1},
  pages={89--112},
  year={2001},
  publisher={SIAM}
}

@book{steinwart2008support,
  title={Support Vector Machines},
  author={Steinwart, Ingo},
  year={2008},
  publisher={Springer}
}

@article{plato2018optimal,
  title={Optimal rates for Lavrentiev regularization with adjoint source conditions},
  author={Plato, Robert and Math{\'e}, Peter and Hofmann, Bernd},
  journal={Mathematics of Computation},
  volume={87},
  number={310},
  pages={785--801},
  year={2018}
}

@article{Boltzmann1964LecturesOG,
  title={Lectures on Gas Theory},
  author={Ludwig Boltzmann and Stephen G. Brush and Nandor L. Balazs},
  journal={Physics Today},
  year={1964},
  volume={17},
  pages={68-68},
  url={https://api.semanticscholar.org/CorpusID:120563737}
}

@article{Furioli2017FokkerPlanckEI,
  title={Fokker–Planck equations in the modeling of socio-economic phenomena},
  author={Giulia Furioli and Ada Pulvirenti and Elide Terraneo and Giuseppe Toscani},
  journal={Mathematical Models and Methods in Applied Sciences},
  year={2017},
  volume={27},
  pages={115-158},
  url={https://api.semanticscholar.org/CorpusID:125945921}
}

@article{Toscani2006KineticMO,
  title={Kinetic models of opinion formation},
  author={Giuseppe Toscani},
  journal={Communications in Mathematical Sciences},
  year={2006},
  volume={4},
  pages={481-496},
  url={https://api.semanticscholar.org/CorpusID:1127752}
}

@article{Jordan1996THEVF,
  title={The variational formulation of the Fokker--Planck equation},
  author={Jordan, Richard and Kinderlehrer, David and Otto, Felix},
  journal={SIAM Journal on Mathematical Analysis},
  volume={29},
  number={1},
  pages={1--17},
  year={1998},
  publisher={SIAM}
}

@article{Carrillo2000AsymptoticLO,
  title={Asymptotic L1-decay of solutions of the porous medium equation to self-similarity},
  author={Jos{\'e} Antonio Carrillo and Giuseppe Toscani},
  journal={Indiana University Mathematics Journal},
  year={2000},
  volume={49},
  pages={0113-142},
  url={https://api.semanticscholar.org/CorpusID:63108}
}

@article{Carrillo2003KineticER,
  title={Kinetic equilibration rates for granular media and related equations: entropy dissipation and mass transportation estimates},
  author={Jos{\'e} Antonio Carrillo and Robert J. McCann and C{\'e}dric Villani},
  journal={Revista Matematica Iberoamericana},
  year={2003},
  volume={19},
  pages={971-1018},
  url={https://api.semanticscholar.org/CorpusID:73606371}
}

@article{Lott2004RicciCF,
  title={Ricci curvature for metric-measure spaces via optimal transport},
  author={John Lott and C{\'e}dric Villani},
  journal={Annals of Mathematics},
  year={2004},
  volume={169},
  pages={903-991},
  url={https://api.semanticscholar.org/CorpusID:15556613}
}

@article{lang2022learning,
  title={Learning interaction kernels in mean-field equations of first-order systems of interacting particles},
  author={Lang, Quanjun and Lu, Fei},
  journal={SIAM Journal on Scientific Computing},
  volume={44},
  number={1},
  pages={A260--A285},
  year={2022},
  publisher={SIAM}
}

@article{zhou2008derivative,
  title={Derivative reproducing properties for kernel methods in learning theory},
  author={Zhou, Ding-Xuan},
  journal={Journal of computational and Applied Mathematics},
  volume={220},
  number={1-2},
  pages={456--463},
  year={2008},
  publisher={Elsevier}
}

@article{schaeffer2017learning,
  title={Learning partial differential equations via data discovery and sparse optimization},
  author={Schaeffer, Hayden},
  journal={Proceedings of the Royal Society A: Mathematical, Physical and Engineering Sciences},
  volume={473},
  number={2197},
  pages={20160446},
  year={2017},
  publisher={The Royal Society Publishing}
}

@article{chow2020wasserstein,
  title={Wasserstein Hamiltonian flows},
  author={Chow, Shui-Nee and Li, Wuchen and Zhou, Haomin},
  journal={Journal of Differential Equations},
  volume={268},
  number={3},
  pages={1205--1219},
  year={2020},
  publisher={Elsevier}
}

@book{villani2009optimal,
  title={Optimal Transport: Old and New},
  author={Villani, C{\'e}dric},
  volume={338},
  year={2009},
  publisher={Springer}
}

@article{feng2021learning,
  title={Learning particle swarming models from data with {G}aussian processes},
  author={Feng, Jinchao and Kulick, Charles and Ren, Yunxiang and Tang, Sui},
  journal={Mathematics of Computation},
  year={2023}
}

@article{lu2019nonparametric,
  title={Nonparametric inference of interaction laws in systems of agents from trajectory data},
  author={Lu, Fei and Zhong, Ming and Tang, Sui and Maggioni, Mauro},
  journal={Proceedings of the National Academy of Sciences},
  volume={116},
  number={29},
  pages={14424--14433},
  year={2019},
  publisher={National Academy of Sciences}
}

@article{chow2019discrete,
  title={A discrete Schr{\"o}dinger equation via optimal transport on graphs},
  author={Chow, Shui-Nee and Li, Wuchen and Zhou, Haomin},
  journal={Journal of Functional Analysis},
  volume={276},
  number={8},
  pages={2440--2469},
  year={2019},
  publisher={Elsevier}
}

@article{lafferty1988density,
  title={The density manifold and configuration space quantization},
  author={Lafferty, John D},
  journal={Transactions of the American Mathematical Society},
  volume={305},
  number={2},
  pages={699--741},
  year={1988}
}

@article{ambrosio2008hamiltonian,
  title={Hamiltonian ODEs in the Wasserstein space of probability measures},
  author={Ambrosio, Luigi and Gangbo, Wilfrid},
  journal={Communications on Pure and Applied Mathematics: A Journal Issued by the Courant Institute of Mathematical Sciences},
  volume={61},
  number={1},
  pages={18--53},
  year={2008},
  publisher={Wiley Online Library}
}

@article{carlen2003constrained,
  title={Constrained steepest descent in the 2-Wasserstein metric},
  author={Carlen, Eric A and Gangbo, Wilfrid},
  journal={Annals of mathematics},
  pages={807--846},
  year={2003},
  publisher={JSTOR}
}

@article{gao2024self,
  title={Self-test loss functions for learning weak-form operators and gradient flows},
  author={Gao, Yuan and Lang, Quanjun and Lu, Fei},
  journal={arXiv preprint arXiv:2412.03506},
  year={2024}
}

@article{RCSP3,
author = {Hu, Jianyu and Ortega, Juan-Pablo and Yin, Daiying},
journal = {Journal of Nonlinear Science},
number = {79},
title = {{A global structure-preserving kernel method for the learning of Poisson systems}},
volume = {35},
year = {2025}
}

@article{hu2024energetic,
  title={Energetic variational neural network discretizations of gradient flows},
  author={Hu, Ziqing and Liu, Chun and Wang, Yiwei and Xu, Zhiliang},
  journal={SIAM Journal on Scientific Computing},
  volume={46},
  number={4},
  pages={A2528--A2556},
  year={2024},
  publisher={SIAM}
}

@article{lu2024learning,
  title={Learning generalized diffusions using an energetic variational approach},
  author={Lu, Yubin and Li, Xiaofan and Liu, Chun and Tang, Qi and Wang, Yiwei},
  journal={arXiv preprint arXiv:2412.04480},
  year={2024}
}

@article{Gangbo2011,
   author = {Wilfrid Gangbo and Hwa Kim and Tommaso Pacini},
   doi = {10.1090/S0065-9266-2010-00610-0},
   issn = {0065-9266},
   issue = {993},
   journal = {Memoirs of the American Mathematical Society},
   pages = {0-0},
   title = {Differential forms on Wasserstein space and infinite-dimensional Hamiltonian systems},
   volume = {211},
   year = {2011}
}

@article{ebin1970groups,
  title={Groups of diffeomorphisms and the motion of an incompressible fluid},
  author={Ebin, David G and Marsden, Jerrold},
  journal={Annals of Mathematics},
  volume={92},
  number={1},
  pages={102--163},
  year={1970},
  publisher={JSTOR}
}

@inproceedings{arnold1966geometrie,
  title={Sur la g{\'e}om{\'e}trie diff{\'e}rentielle des groupes de Lie de dimension infinie et ses applications {\`a} l'hydrodynamique des fluides parfaits},
  author={Arnold, Vladimir},
  booktitle={Annales de l'institut Fourier},
  volume={16},
  number={1},
  pages={319--361},
  year={1966}
}

@article{Gangbo2009,
   author = {W. Gangbo and T. Nguyen and A. Tudorascu},
   doi = {10.1007/s00205-008-0148-y},
   issn = {0003-9527},
   issue = {3},
   journal = {Archive for Rational Mechanics and Analysis},
   month = {6},
   pages = {419-452},
   title = {Euler–Poisson systems as action-minimizing paths in the Wasserstein space},
   volume = {192},
   year = {2009}
}

@article{Marsden1982TheHS,
  title={The Hamiltonian structure of the Maxwell-Vlasov equations},
  author={Jerrold E. Marsden and Alan D. Weinstein},
  journal={Physica D: Nonlinear Phenomena},
  year={1982},
  volume={4},
  pages={394-406},
  url={https://api.semanticscholar.org/CorpusID:120219138}
}

@book{evans2022partial,
  title={Partial Differential Equations},
  author={Evans, Lawrence C},
  volume={19},
  year={2022},
  publisher={American Mathematical Society}
}

@book{warner1983foundations,
  title={Foundations of Differentiable Manifolds and Lie Groups},
  author={Warner, Frank W},
  volume={94},
  year={1983},
  publisher={Springer Science \& Business Media}
}

@article{lazaro2008stochastic,
author = {L{\'{a}}zaro-Cam{\'{i}}, Andreu and Ortega, Juan-Pablo},
journal = {Reports on Mathematical Physics},
number = {1},
pages = {65--122},
publisher = {Elsevier},
title = {{Stochastic Hamiltonian dynamical systems}},
volume = {61},
year = {2008}
}

@article{lazaro:ortega3,
author = {L{\'{a}}zaro-Cam{\'{i}}, Andreu and Ortega, Juan-Pablo},
journal = {Journal of Geometric Mechanics},
number = {3},
pages = {295--315},
title = {{The stochastic Hamilton-Jacobi equation}},
volume = {1},
year = {2009}
}

@article{ChenCruzeiroRatiu2023,
  author  = {Chen, Xin and Cruzeiro, Ana Bela and Ratiu, Tudor S.},
  title   = {Stochastic Variational Principles for Dissipative Equations with Advected Quantities},
  journal = {Journal of Nonlinear Science},
  volume  = {33},
  number  = {1},
  pages   = {5},
  year    = {2023},
  doi     = {10.1007/s00332-022-09846-1}
}

@article{CruzeiroHolmRatiu2018,
  author  = {Cruzeiro, Ana Bela and Holm, Darryl D. and Ratiu, Tudor S.},
  title   = {Momentum Maps and Stochastic {C}lebsch Action Principles},
  journal = {Communications in Mathematical Physics},
  volume  = {357},
  number  = {2},
  pages   = {873--912},
  year    = {2018},
  doi     = {10.1007/s00220-017-3048-x}
}

@inproceedings{CampoloHuOrtegaYin2025,
  author    = {Campolo, Domenico and Hu, Jianyu and
               Ortega, Juan-Pablo and Yin, Daiying},
  title     = {A kernel-based global method for the learning
               of elastic potentials on {Lie} groups},
  booktitle = {Geometric Science of Information},
  editor    = {Nielsen, Frank and Barbaresco, Fr{\'e}d{\'e}ric},
  series    = {Lecture Notes in Computer Science},
  volume    = {16035},
  pages     = {93--102},
  publisher = {Springer},
  year      = {2025},
  doi       = {10.1007/978-3-032-03924-8_10}
}
\end{document}